\documentclass[a4paper, reqno]{amsart}

\usepackage[margin=1.5in]{geometry}			
\usepackage[utf8]{inputenc}					
\usepackage[T1]{fontenc}					
\usepackage[USenglish]{babel}				
\usepackage{lmodern}						
\usepackage{setspace}						

\usepackage{amsmath,amsthm,amssymb,amsrefs}			
\usepackage[all]{xy}						
\SelectTips{cm}{}							
\usepackage{graphicx}						

\usepackage{xspace} 						
\usepackage{etoolbox}						
\usepackage{enumitem}						
\usepackage{mathtools}						
\usepackage{caption}						
\usepackage[labelformat=simple]{subcaption}	

\usepackage{subdepth}						

\usepackage{eqparbox}

\DeclareMathOperator{\Spf}{Spf}
\DeclareMathOperator{\Spec}{Spec}
\DeclareMathOperator{\Tr}{Tr}
\DeclareMathOperator{\Nm}{Nm}
\DeclareMathOperator{\Trd}{Trd}

\DeclareMathOperator{\End}{End}
\DeclareMathOperator{\Lie}{Lie}
\DeclareMathOperator{\Hom}{Hom}
\DeclareMathOperator{\disc}{disc}
\DeclareMathOperator{\charp}{char}
\DeclareMathOperator{\Gal}{Gal}
\DeclareMathOperator{\SU}{SU}
\DeclareMathOperator{\GU}{GU}
\DeclareMathOperator{\U}{U}
\DeclareMathOperator{\SL}{SL}
\DeclareMathOperator{\GL}{GL}
\DeclareMathOperator{\GSp}{GSp}

\DeclareMathOperator{\QIsog}{QIsog}

\DeclareMathOperator{\red}{red}
\DeclareMathOperator{\nat}{nat}
\DeclareMathOperator{\Gr}{Gr}
\DeclareMathOperator{\Nilp}{Nilp_{\breve{O}_F}}
\DeclareMathOperator{\weq}{\widehat{=}}
\DeclareMathOperator{\Defo}{Def}

\newcommand{\NEnaive}{\mathcal{N}_E^{\mathrm{naive}}}
\newcommand{\NEL}{\mathcal{N}_{E, \Lambda}}
\newcommand{\NELpr}{\mathcal{N}_{E, \Lambda'}}
\newcommand{\NE}{\mathcal{N}_E}
\newcommand{\MDr}{\mathcal{M}_{Dr}}

\newcommand{\NEnaiverig}{\mathcal{N}_E^{\mathrm{naive,rig}}}
\newcommand{\NErig}{\mathcal{N}_E^{\mathrm{rig}}}

\newcommand{\ME}{\mathcal{M}_E}
\newcommand{\MEpol}{\mathcal{M}_{E,\mathrm{pol}}}
\newcommand{\MEloc}{\mathrm{M}_E^{\mathrm{loc}}}
\newcommand{\MEpolloc}{\mathrm{M}_{E,\mathrm{pol}}^{\mathrm{loc}}}
\newcommand{\MElarge}{\mathcal{M}_E^{\mathrm{large}}}
\newcommand{\MEpollarge}{\mathcal{M}_{E,\mathrm{pol}}^{\mathrm{large}}}
\newcommand{\MElochat}{\widehat{\mathrm{M}}_E^{\mathrm{loc}}}
\newcommand{\MEpollochat}{\widehat{\mathrm{M}}_{E,\mathrm{pol}}^{\mathrm{loc}}}

\theoremstyle{plain}
\newtheorem{thm}{Theorem}[section]
\newtheorem{lem}[thm]{Lemma}
\newtheorem{prop}[thm]{Proposition}

\theoremstyle{definition}
\newtheorem{rmk}[thm]{Remark}
\newtheorem{defn}[thm]{Definition}

\newtheorem{qu}[thm]{Question}

\newcommand{\overbar}[1]{\mkern 1.5mu\overline{\mkern-1.5mu#1\mkern-1.5mu}\mkern 1.5mu}

\renewenvironment{enumerate}				
	{\begin{list}
		{\textup{(\theenumi)} }
		{\usecounter{enumi}
			\setlength{\labelwidth}{0pt}
			\setlength{\labelsep}{0pt}
			\setlength{\leftmargin}{0pt}
			\setlength{\itemsep}{\the\smallskipamount}
		}
	}
	{\end{list}}

\renewenvironment{itemize}					
	{\begin{list}
		{$\bullet$}
		{\setlength{\labelwidth}{0pt}
			\setlength{\itemindent}{5pt}
			\setlength{\labelsep}{5pt}
			\setlength{\leftmargin}{0pt}
			\setlength{\itemsep}{\the\smallskipamount}
		}
	}
	{\end{list}}

\newcommand{\isoarrow}						
	{\ifbool{@display}
		{\overset{\smash{\raisebox{-0.65ex}{$\scriptstyle\sim$}}}{\longrightarrow}}
		{\xrightarrow{\,\smash{\raisebox{-0.65ex}{$\scriptstyle\sim$}}\,}}
	}

\newcommand*{\longhookrightarrow}{\ensuremath{\lhook\joinrel\relbar\joinrel\rightarrow}}
\newcommand*{\inj}{
	\ifbool{@display}{\longhookrightarrow}{\hookrightarrow}%
	}

\renewcommand{\to}{
   \ifbool{@display}{\longrightarrow}{\rightarrow}%
   }

\let\shortmapsto\mapsto						
\renewcommand{\mapsto}{%
   \ifbool{@display}{\longmapsto}{\shortmapsto}%
   }

\numberwithin{equation}{section} 			

\entrymodifiers={+!!<0pt,\fontdimen22\textfont2>}

\begin{document}
\hyphenation{Dieu-don-n\'e}
\nocite{Dri76}

\title[Construction of a RZ-spaces for $2$-adic ramified $\GU(1,1)$]{Construction of a Rapoport-Zink space for $\GU(1,1)$ in the ramified $2$-adic case}
\author{Daniel Kirch}
\date{\today}
\begin{abstract}
	Let $F|\mathbb{Q}_2$ be a finite extension.
	In this paper, we construct an RZ-space $\NE$ for split $\GU(1,1)$ over a ramified quadratic extension $E|F$.
	For this, we first introduce the naive moduli problem $\NEnaive$ and then define $\NE \subseteq \NEnaive$ as a canonical closed formal subscheme, using the so-called straightening condition.
	We establish an isomorphism between $\NE$ and the Drinfeld moduli problem, proving the $2$-adic analogue of a theorem of Kudla and Rapoport.
	The formulation of the straightening condition uses the existence of certain polarizations on the points of the moduli space $\NEnaive$.
	We show the existence of these polarizations in a more general setting over any quadratic extension $E|F$, where $F|\mathbb{Q}_p$ is a finite extension for any prime $p$.
	
	[2010 Mathematics Subject Classification: 11G18,14G35]
\end{abstract}
\maketitle
\vfill
\tableofcontents
\vfill

\newpage

\section{Introduction}

Rapoport-Zink spaces (short RZ-spaces) are moduli spaces of $p$-divisible groups endowed with additional structure.
In \cite{RZ96}, Rapoport and Zink study two major classes of RZ-spaces, called (EL) type and (PEL) type.
The abbreviations (EL) and (PEL) indicate, in analogy to the case of Shimura varieties, whether the extra structure comes in form of \textbf{E}ndomorphisms and \textbf{L}evel structure or in form of \textbf{P}olarizations, \textbf{E}ndomorphisms and \textbf{L}evel structure.
\cite{RZ96} develops a theory of these spaces, including important theorems about the existence of local models and non-archimedean uniformization of Shimura varieties, for the (EL) type and for the (PEL) type whenever $p \neq 2$.

The blanket assumption $p \neq 2$ made by Rapoport and Zink in the (PEL) case is by no means of cosmetical nature, but originates to various serious difficulties that arise for $p=2$.
However, we recall that one can still use their definition in that case to obtain ``naive'' moduli spaces that still satisfy basic properties like being representable by a formal scheme.

In this paper, we construct the $2$-adic Rapoport-Zink space $\NE$ corresponding to the group of unitary similitudes of size $2$ relative to any (wildly) ramified quadratic extension $E|F$, where $F|\mathbb{Q}_2$ is a finite extension.
It is given as the closed formal subscheme of the corresponding naive RZ-space $\NEnaive$ described by the so-called ``straightening condition'', which is defined below.
The main result of this paper is a natural isomorphism $\eta: \MDr \isoarrow \NE$, where $\MDr$ is Deligne's formal model of the Drinfeld upper halfplane (\emph{cf.}\ \cite{BC91}).
This result is in analogy with \cite{KR14}, where Kudla and Rapoport construct a corresponding isomorphism for $p \neq 2$ and also for $p=2$ when $E|F$ is an unramified extension.
The formal scheme $\MDr$ solves a certain moduli problem of $p$-divisible groups and, in this way, it carries the structure of an RZ-space of (EL) type.
In particular, $\MDr$ is defined even for $p=2$. 

As in loc.\ cit., there are natural group actions by $\SL_2(F)$ and the split $\SU_2(F)$ on the spaces $\MDr$ and $\NE$, respectively.
The isomorphism $\eta$ is hence a geometric realization of the exceptional isomorphism of these groups.
As a consequence, one cannot expect a similar result in higher dimensions.
Of course, the existence of ``good'' RZ-spaces is still expected, but a general definition will probably need a different approach.

The study of residue characteristic $2$ is interesting and important for the following reasons:
First of all, from the general philosophy of RZ-spaces and, more generally, of local Shimura varieties \cite{RV14}, it follows that there should be uniform approach for all primes $p$.
In this sense, the present paper is in the same spirit as the recent constructions of RZ-spaces of Hodge type of W. Kim \cite{Kim}, Howard and Pappas \cite{HP} and Bültel and Pappas \cite{BP}.
Second, Rapoport-Zink spaces have been used to determine the arithmetic intersection numbers of special cycles on Shimura varieties \cite{KRY}; in this kind of problem, it is necessary to deal with all places, even those of residue characteristic $2$.
Finally, studying the cases of residue characteristic $2$ also throws light on the cases previously known.
In the specific case at hand, the methods we develop in the present paper also give a simplification of the proof for $p\neq 2$ of Kudla and Rapoport \cite{KR14}, see Remark \ref{POL_p>2} \eqref{POL_p>2eq}.
\smallskip

We will now explain the results of this paper in greater detail.
Let $F$ be a finite extension of $\mathbb{Q}_2$ and $E|F$ a ramified quadratic extension.
Following \cite{Jac62}, we consider the following dichotomy for this extension (see section \ref{LA}):
\begin{itemize}
	\item[(R-P)] There is a uniformizer $\pi_0 \in F$, such that $E = F[\Pi]$ with $\Pi^2 + \pi_0 = 0$.
	Then the rings of integers $O_F$ of $F$ and $O_E$ of $E$ satisfy $O_E = O_F[\Pi]$.
	\item[(R-U)] $E|F$ is given by an Eisenstein equation of the form $\Pi^2 - t\Pi + \pi_0 = 0$.
	Here, $\pi_0$ is again a uniformizer in $F$ and $t \in O_F$ satisfies $\pi_0|t|2$. We still have $O_E = O_F[\Pi]$.
	Note that in this case $E|F$ is generated by a square root of the unit $1-4\pi_0/t^2$ in $F$.
\end{itemize}
An example for an extension of type (R-P) is $\mathbb{Q}_2(\sqrt{-2})| \mathbb{Q}_2$, whereas $\mathbb{Q}_2(\sqrt{-1})| \mathbb{Q}_2$ is of type (R-U).
Note that for $p>2$, any ramified quadratic extension over $\mathbb{Q}_p$ is of the form (R-P).

Our results in the cases (R-P) and (R-U) are similar, but different.
We first describe the results in the case (R-P).
Let $E|F$ be of type (R-P).
\smallskip

We first define a naive moduli problem $\NEnaive$, that merely copies the definition from $p \neq 2$ (\emph{cf.}\ \cite{KR14}).
Let $\breve{F}$ be the completion of the maximal unramified extension of $F$ and $\breve{O}_F$ its ring of integers.
Then $\NEnaive$ is a set-valued functor on $\Nilp$, the category of $\breve{O}_F$-schemes where $\pi_0$ is locally nilpotent.
For $S \in \Nilp$, the set $\NEnaive(S)$ is the set of equivalence classes of tuples $(X,\iota,\lambda,\varrho)$.
Here, $X/S$ is a formal $O_F$-module of height $4$ and dimension $2$, equipped with an action $\iota: O_E \to \End(X)$.
This action satisfies the Kottwitz condition of signature $(1,1)$, \emph{i.e.}, for any $\alpha \in O_E$, the characteristic polynomial of $\iota(\alpha)$ on $\Lie X$ is given by
\begin{equation*}
	\charp(\Lie X, T \mid \iota(\alpha)) = (T - \alpha)(T - \overbar{\alpha}).
\end{equation*}
Here, $\alpha \mapsto \overbar{\alpha}$ denotes the Galois conjugation of $E|F$.
The right hand side of this equation is a polynomial with coefficients in $\mathcal{O}_S$ via the structure map $O_F \inj \breve{O}_F \to \mathcal{O}_S$.
The third entry $\lambda$ is a principal polarization $\lambda: X \to X^{\vee}$ such that the induced Rosati involution satisfies $\iota(\alpha)^{\ast} = \iota(\overbar{\alpha})$ for all $\alpha \in O_E$. (Here, $X^{\vee}$ is the dual of $X$ as formal $O_F$-module.)
Finally, $\varrho$ is a quasi-isogeny of height $0$ (and compatible with all previous data) to a fixed framing object $(\mathbb{X}, \iota_{\mathbb{X}}, \lambda_{\mathbb{X}})$ over $\overbar{k} = \breve{O}_F / \pi_0$.
This framing object is unique up to isogeny under the condition that
\begin{equation*}
	\{ \varphi \in \End^0(\mathbb{X}, \iota_{\mathbb{X}}) \mid \varphi^{\ast}(\lambda_{\mathbb{X}}) = \lambda_{\mathbb{X}} \} \simeq \U(C,h),
\end{equation*}
for a split $E|F$-hermitian vector space $(C,h)$ of dimension $2$, see Lemma \ref{RP_frnaive}.

Recall that this is exactly the definition used in loc.\ cit.\ for the ramified case with $p > 2$.
There, $\NE = \NEnaive$ and we have natural isomorphism
\begin{equation*}
	\eta: \MDr \isoarrow \NE,
\end{equation*}
where $\MDr$ is the Drinfeld moduli problem mentioned above.

However, for $p=2$, it turns out that the definition of $\NEnaive$ is not the ``correct'' one in the sense that it is not isomorphic to the Drinfeld moduli problem.
Hence this naive definition of the moduli space is not in line with the results from \cite{KR14} and the general philosophy of (conjectural) local Shimura varieties (see \cite{RV14}).
In order to remedy this, we will describe a new condition on $\NEnaive$, which we call the \emph{straightening condition}, and show that this cuts out a closed formal subscheme $\NE \subseteq \NEnaive$ that is naturally isomorphic to $\MDr$.
Interestingly, the straightening condition is not trivial on the rigid-analytic generic fiber of $\NEnaive$ (as originally assumed by the author), but it cuts out an (admissible) open and closed subspace, see Remark \ref{RP_genfiber}.

We would like to explicate the defect of the naive moduli space.
For this, let us recall the definition of $\MDr$.
It is a functor on $\Nilp$, mapping a scheme $S$ to the set $\MDr(S)$ of equivalence classes  of tuples $(X,\iota_B,\varrho)$.
Again, $X/S$ is a formal $O_F$-module of height $4$ and dimension $2$.
Let $B$ be the quaternion division algebra over $F$ and $O_B$ its ring of integers.
Then $\iota_B$ is an action of $O_B$ on $X$, satisfying the \emph{special} condition of Drinfeld (see \cite{BC91} or section \ref{RP3} below).
The last entry $\varrho$ is an $O_B$-linear quasi-isogeny of height $0$ to a fixed framing object $(\mathbb{X}, \iota_{\mathbb{X},B})$ over $\overbar{k}$.
This framing object is unique up to isogeny (\emph{cf.}\ \cite[II.\ Prop.\ 5.2]{BC91}).

Fix an embedding $O_E \inj O_B$ and consider the involution $b \mapsto b^{\ast} = \Pi b' \Pi^{-1}$ on $B$, where $b \mapsto b'$ is the standard involution.
By Drinfeld (see Proposition \ref{RP_Dr} below), there exists a principal polarization $\lambda_{\mathbb{X}}$ on the framing object $(\mathbb{X}, \iota_{\mathbb{X},B})$ of $\MDr$, such that the induced Rosati involution satisfies $\iota_{\mathbb{X},B}(b)^{\ast} = \iota_{\mathbb{X},B}(b^{\ast})$ for all $b \in O_B$.
This polarization is unique up to a scalar in $O_F^{\times}$.
Furthermore, for any $(X,\iota_B,\varrho) \in \MDr(S)$, the pullback $\lambda = \varrho^{\ast}(\lambda_{\mathbb{X}})$ is a principal polarization on $X$.

We now set
\begin{equation*}
	\eta(X,\iota_B,\varrho) = (X,\iota_B|_{O_E},\lambda,\varrho).
\end{equation*}
By Lemma \ref{RP_clemb}, this defines a closed embedding $\eta: \MDr \inj \NEnaive$.
But $\eta$ is far from being an isomorphism, as the following proposition shows:

\begin{prop}
	The induced map $\eta(\overbar{k}): \MDr(\overbar{k}) \to \NEnaive(\overbar{k})$ is not surjective.
\end{prop}

Let us sketch the proof here.
Using Dieudonn\'e theory, we can write $\NEnaive(\overbar{k})$ naturally as a union
\begin{equation*}
	\NEnaive(\overbar{k}) = \smashoperator[r]{\bigcup_{\Lambda \subseteq C}} \mathbb{P}(\Lambda / \Pi \Lambda)(\overbar{k}),
\end{equation*}
where the union runs over all $O_E$-lattices $\Lambda$ in the hermitian vector space $(C,h)$ that are $\Pi^{-1}$-modular, \emph{i.e.}, the dual $\Lambda^{\sharp}$ of $\Lambda$ with respect to $h$ is given by $\Lambda = \Pi^{-1} \Lambda^{\sharp}$ (see Lemma \ref{RP_IS}).
By Jacobowitz (\cite{Jac62}), there exist different types (\emph{i.e.}, $\U(C,h)$-orbits) of such lattices $\Lambda \subseteq C$ that are parametrized by their norm ideal $\Nm (\Lambda) = \langle \{ h(x,x) | x \in \Lambda \} \rangle \subseteq F$.
In the case at hand, $\Nm(\Lambda)$ can be any ideal with $2 O_F \subseteq \Nm(\Lambda) \subseteq O_F$.
It is easily checked (see Chapter \ref{LA}) that the norm ideal of $\Lambda$ is minimal, that is $\Nm (\Lambda) = 2 O_F$, if and only if $\Lambda$ admits a basis consisting of isotropic vectors, and hence we call these lattices \emph{hyperbolic}.
Now, the image under $\eta$ of $\MDr(\overbar{k})$ is the union of all lines $\mathbb{P}(\Lambda / \Pi \Lambda)(\overbar{k})$ where $\Lambda \subseteq C$ is hyperbolic.
This is a consequence of Remark \ref{RP_rmkNE} and Theorem \ref{RP_thm} below.

On the framing object $(\mathbb{X}, \iota_{\mathbb{X}}, \lambda_{\mathbb{X}})$ of $\NEnaive$, there exists a principal polarization $\widetilde{\lambda}_{\mathbb{X}}$ such that the induced Rosati involution is the identity on $O_E$.
This polarization is unique up to a scalar in $O_E^{\times}$ (see Thm.\ \ref{POL_thm} \eqref{POL_thm1}).
On $C$, the polarization $\widetilde{\lambda}_{\mathbb{X}}$ induces an $E$-linear alternating form $b$, such that $\det b$ and $\det h$ differ only by a unit (for a fixed basis of $C$).
After possibly rescaling $b$ by a unit in $O_E^{\times}$, a $\Pi^{-1}$-modular lattice $\Lambda \subseteq C$ is hyperbolic if and only if $b(x,y) + h(x,y) \in 2 O_F$ for all $x,y \in \Lambda$.
This enables us to describe the ``hyperbolic'' points of $\NEnaive$ (\emph{i.e.}, those that lie on a projective line corresponding to a hyperbolic lattice $\Lambda \subseteq C$) in terms of polarizations.

We now formulate the closed condition that characterizes $\NE$ as a closed formal subscheme of $\NEnaive$.
For a suitable choice of $(\mathbb{X}, \iota_{\mathbb{X}}, \lambda_{\mathbb{X}})$ and $\widetilde{\lambda}_{\mathbb{X}}$, we may assume that $\frac{1}{2}(\lambda_{\mathbb{X}} + \widetilde{\lambda}_{\mathbb{X}})$ is a polarization on $\mathbb{X}$.
The following definition is a reformulation of Definition \ref{RP_strdef}.

\begin{defn}
	Let $S \in \Nilp$.
	An object $(X,\iota,\lambda,\varrho) \in \NEnaive(S)$ satisfies the \emph{straightening} condition, if $\lambda_1 = \frac{1}{2} (\lambda + \widetilde{\lambda})$ is a polarization on $X$.
	Here, $\widetilde{\lambda} = \varrho^{\ast}(\widetilde{\lambda}_{\mathbb{X}})$.
\end{defn}

We remark that $\widetilde{\lambda} = \varrho^{\ast}(\widetilde{\lambda}_{\mathbb{X}})$ is a polarization on $X$.
This is a consequence of Theorem \ref{POL_thm}, which states the existence of certain polarizations on points of a larger moduli space $\ME$ containing $\NEnaive$, see below.

For $S \in \Nilp$, let $\NE(S) \subseteq \NEnaive(S)$ be the subset of all tuples $(X,\iota,\lambda,\varrho)$ that satisfy the straightening condition.
By \cite[Prop.\ 2.9]{RZ96}, this defines a closed formal subscheme $\NE \subseteq \NEnaive$.
An application of Drinfeld's Proposition (Proposition \ref{RP_Dr}, see also \cite{BC91}) shows that the image of $\MDr$ under $\eta$ lies in $\NE$.
The main theorem in the (R-P) case can now be stated as follows, see Theorem \ref{RP_thm}.

\begin{thm}
	$\eta: \MDr \to \NE$ is an isomorphism of formal schemes.
\end{thm}

This concludes our discussion of the (R-P) case.
From now on, we assume that $E|F$ is of type (R-U).

In the case (R-U), we have to make some adaptions for $\NEnaive$.
For $S \in \Nilp$, let $\NEnaive(S)$ be the set of equivalence classes of tuples $(X,\iota,\lambda,\varrho)$ with $(X,\iota)$ as in the \mbox{(R-P)} case.
But now, the polarization $\lambda: X \to X^{\vee}$ is supposed to have kernel $\ker \lambda = X[\Pi]$ (in contrast to the (R-P) case, where $\lambda$ is a principal polarization).
As before, the Rosati involution of $\lambda$ induces the conjugation on $O_E$.
There exists a framing object $(\mathbb{X}, \iota_{\mathbb{X}}, \lambda_{\mathbb{X}})$ over $\Spec \overbar{k}$ for $\NEnaive$, which is unique up to isogeny under the condition that
\begin{equation*}
	\{ \varphi \in \End^0(\mathbb{X}, \iota_{\mathbb{X}}) \mid \varphi^{\ast}(\lambda_{\mathbb{X}}) = \lambda_{\mathbb{X}} \} \simeq \U(C,h),
\end{equation*}
where $(C,h)$ is a split $E|F$-hermitian vector space of dimension $2$ (see Proposition \ref{RU_frnaive}).
Finally, $\varrho$ is a quasi-isogeny of height $0$ from $X$ to $\mathbb{X}$, respecting all structure.

Fix an embedding $E \inj B$.
Using some subtle choices of elements in $B$ (these are described in Lemma \ref{LA_quat} \eqref{LA_quatRU}) and by Drinfeld's Proposition, we can construct a polarization $\lambda$ as above for any $(X,\iota_B,\varrho) \in \MDr(S)$.
This induces a closed embedding
\begin{equation*}
	\eta: \MDr \to \NEnaive, (X,\iota_B,\varrho) \mapsto (X,\iota_B|_{O_E},\lambda,\varrho).
\end{equation*}
We can write $\NEnaive(\overbar{k})$ as a union of projective lines,
\begin{equation*}
	\NEnaive(\overbar{k}) = \smashoperator[r]{\bigcup_{\Lambda \subseteq C}} \mathbb{P}(\Lambda / \Pi \Lambda)(\overbar{k}),
\end{equation*}
where the union now runs over all selfdual $O_E$-lattices $\Lambda \subseteq (C,h)$ with $\Nm (\Lambda) \subseteq \pi_0 O_F$.
As in the (R-P) case, these lattices $\Lambda \subseteq C$ are classified up to isomorphism by their norm ideal $\Nm(\Lambda)$.
Since $\Lambda$ is selfdual with respect to $h$, the norm ideal can be any ideal satisfying $t O_F \subseteq \Nm(\Lambda) \subseteq O_F$.
We call $\Lambda$ \emph{hyperbolic} when the norm ideal is minimal, \emph{i.e.}, $\Nm (\Lambda) = t O_F$.
Equivalently, the lattice $\Lambda$ has a basis consisting of isotropic vectors.
Recall that here $t$ is the element showing up in the Eisenstein equation for the (R-U) extension $E|F$ and that $\pi_0|t|2$.
Hence there exists at least one type of selfdual lattices $\Lambda \subseteq C$ with $\Nm (\Lambda) \subseteq \pi_0 O_F$.
In the case (R-U), it may happen that $|t| = |\pi_0|$, in which case all lattices $\Lambda$ in the description of $\NEnaive(\overbar{k})$ are hyperbolic.

The image of $\MDr(\overbar{k})$ under $\eta$ in $\NEnaive(\overbar{k})$ is the union of all projective lines corresponding to hyperbolic lattices.
Unless $|t| = |\pi_0|$, it follows that $\eta(\overbar{k})$ is not surjective and thus $\eta$ cannot be an isomorphism.
For the case $|t| = |\pi_0|$, we will show that $\eta$ is an isomorphism on reduced loci $(\MDr)_{\red} \isoarrow (\NEnaive)_{\red}$ (see Remark \ref{RU_rmkNE}), but $\eta$ is not an isomorphism of formal schemes.
This follows from the non-flatness of the deformation ring for certain points of $\NEnaive$, see section \ref{LM_naive}.

On the framing object $(\mathbb{X}, \iota_{\mathbb{X}}, \lambda_{\mathbb{X}})$ of $\NEnaive$, there exists a polarization $\widetilde{\lambda}_{\mathbb{X}}$ such that $\ker \widetilde{\lambda}_{\mathbb{X}} = \mathbb{X}[\Pi]$ and such that the Rosati involution induces the identity on $O_E$.
After a suitable choice of $(\mathbb{X}, \iota_{\mathbb{X}}, \lambda_{\mathbb{X}})$ and $\widetilde{\lambda}_{\mathbb{X}}$, we may assume that $\frac{1}{t} (\lambda_{\mathbb{X}} + \widetilde{\lambda}_{\mathbb{X}})$ is a polarization on $\mathbb{X}$.
The straightening condition for the (R-U) case is given as follows (see Definition \ref{RU_strdef}).

\begin{defn}
	Let $S \in \Nilp$.
	An object $(X,\iota,\lambda,\varrho) \in \NEnaive(S)$ satisfies the \emph{straightening} condition, if $\lambda_1 = \frac{1}{t}(\lambda + \widetilde{\lambda})$ is a polarization on $X$.
	Here, $\widetilde{\lambda} = \varrho^{\ast}(\widetilde{\lambda}_{\mathbb{X}})$.
\end{defn}

Note that $\widetilde{\lambda} = \varrho^{\ast}(\widetilde{\lambda}_{\mathbb{X}})$ is a polarization on $X$ by Theorem \ref{POL_thm}.

The straightening condition defines a closed formal subscheme $\NE \subseteq \NEnaive$ that contains the image of $\MDr$ under $\eta$.
The main theorem in the (R-U) case can now be stated as follows, compare Theorem \ref{RU_thm}.

\begin{thm}
	$\eta: \MDr \to \NE$ is an isomorphism of formal schemes.
\end{thm}

When formulating the straightening condition in the (R-U) and the (R-P) case, we mentioned that $\widetilde{\lambda} =  \varrho^{\ast}(\widetilde{\lambda}_{\mathbb{X}})$ is a polarization for any $(X,\iota,\lambda,\varrho) \in \NEnaive(S)$.
This fact is a corollary of Theorem \ref{POL_thm}, that states the existence of this polarization in the following more general setting.

Let $F|\mathbb{Q}_p$ be a finite extension for any prime $p$ and $E|F$ an arbitrary quadratic extension.
We consider the following moduli space $\ME$ of (EL) type.
For $S \in \Nilp$, the set $\ME(S)$ consists of equivalence classes of tuples $(X,\iota_E,\varrho)$, where $X$ is a formal $O_F$-module of height $4$ and dimension $2$ and $\iota_E$ is an $O_E$-action on $X$ satisfying the Kottwitz condition of signature $(1,1)$ as above.
The entry $\varrho$ is an $O_E$-linear quasi-isogeny of height $0$ to a supersingular framing object $(\mathbb{X}, \iota_{\mathbb{X},E})$.

The points of $\ME$ are equipped with polarizations in the following natural way, see Theorem \ref{POL_thm}.

\begin{thm} \label{IN_POL}
	\begin{enumerate}
		\item \label{IN_POL1} There exists a principal polarization $\widetilde{\lambda}_{\mathbb{X}}$ on $(\mathbb{X},\iota_{\mathbb{X},E})$ such that the Rosati involution induces the identity on $O_E$, \emph{i.e.}, $\iota(\alpha)^{\ast} = \iota(\alpha)$ for all $\alpha \in O_E$.
		This polarization is unique up to a scalar in $O_E^{\times}$.
		\item Fix $\widetilde{\lambda}_{\mathbb{X}}$ as in part \eqref{IN_POL1}.
		For any $S \in \Nilp$ and $(X,\iota_E,\varrho) \in \ME(S)$, there exists a unique principal polarization $\widetilde{\lambda}$ on $X$ such that the Rosati involution induces the identity on $O_E$ and such that $\widetilde{\lambda} = \varrho^{\ast}(\widetilde{\lambda}_{\mathbb{X}})$.
	\end{enumerate}
\end{thm}

If $p = 2$ and $E|F$ is ramified of (R-P) or (R-U) type, then there is a canonical closed embedding $\NE \inj \ME$ that forgets about the polarization $\lambda$.
In this way, it follows that $\widetilde{\lambda}$ is a polarization for any $(X,\iota,\lambda,\varrho) \in \NEnaive(S)$.

The statement of Theorem \ref{IN_POL} can also be expressed in terms of an isomorphism of moduli spaces $\MEpol \isoarrow \ME$.
Here $\MEpol$ is a moduli space of (PEL) type, defined by mapping $S \in \Nilp$ to the set of tuples $(X,\iota,\widetilde{\lambda},\varrho)$ where $(X,\iota,\varrho) \in \ME(S)$ and $\widetilde{\lambda}$ is a polarization as in the theorem.
\smallskip

We now briefly describe the contents of the subsequent sections of this paper.
In section \ref{LA}, we recall some facts about the quadratic extensions of $F$, the quaternion algebra $B|F$ and hermitian forms.
In the next two sections, sections \ref{RP} and \ref{RU}, we define the moduli spaces $\NEnaive$, introduce the straightening condition describing $\NE \subseteq \NEnaive$ and prove our main theorem in both the cases (R-P) and (R-U).
Although the techniques are quite similar in both cases, we decided to treat these cases separately, since the results in both cases differ in important details.
Finally, in Section \ref{POL} we prove Theorem \ref{IN_POL} on the existence of the polarizations $\widetilde{\lambda}$.
\smallskip

\paragraph{\bf{Acknowledgements}}
First of all, I am very grateful to my advisor M.\ Rapoport for suggesting this topic and for his constant support and helpful discussions.
I thank the members of our Arbeitsgruppe in Bonn for numerous discussions and I also would like to thank the audience of my two AG talks for many helpful questions and comments.
Furthermore, I owe thanks to A.\ Genestier for many useful remarks and for pointing out a mistake in an earlier version of this paper.
I would also like to thank the referee for helpful comments.

This work is the author's PhD thesis at the University of Bonn, which was supported by the SFB/TR45 `Periods, Moduli Spaces and Arithmetic of Algebraic Varieties' of the DFG (German Research Foundation).
Parts of this paper were written during the fall semester program `New Geometric Methods in Number Theory and Automorphic Forms' at the MSRI in Berkeley.

\section{Preliminaries on quaternion algebras and hermitian forms}
\label{LA}

Let $F|\mathbb{Q}_2$ be a finite extension.
In this section we will recall some facts about the quadratic extensions of $F$, the quaternion division algebra $B|F$ and certain hermitian forms.
For more information on quaternion algebras, see for example the book by Vigneras \cite{Vig80}.
A systematic classification of hermitian forms over local fields has been done by Jacobowitz in \cite{Jac62}.

Let $E|F$ be a quadratic field extension and denote by $O_F$ resp.\ $O_E$ the rings of integers.
There are three mutually exclusive possibilities for $E|F$:
\begin{itemize} \label{LA_quadext}
	\item $E|F$ is unramified.
	Then $E = F[\delta]$ for $\delta$ a square root of a unit in $F$.
	We can choose $\delta$ such that $\delta^2 = 1 + 4u$ for some $u \in O_F^{\times}$.
	In this case, $O_E = O_F[\frac{1+\delta}{2}]$.
	The element $\gamma = \frac{1+\delta}{2}$ satisfies the Eisenstein equation $\gamma^2 - \gamma - u = 0$.
	In the following we will write $F^{(2)}$ instead of $E$ and $O_F^{(2)}$ instead of $O_E$ when talking about the unramified extension of $F$.
	\item $E|F$ is ramified and $E$ is generated by the square root of a uniformizer in $F$.
	That is, $E = F[\Pi]$ and $\Pi$ is given by the Eisenstein equation $\Pi^2 + \pi_0 = 0$ for a uniformizing element $\pi_0 \in O_F$.
	We also have $O_E = O_F[\Pi]$.
	Following Jacobowitz, we will say $E|F$ is of type (R-P) (which stands for ``ramified-prime'').
	\item  Finally, $E|F$ can be given by an Eisenstein equation of the form $\Pi^2 - t\Pi + \pi_0 = 0$ for a uniformizer $\pi_0$ and $t \in O_F$ such that $\pi_0|t|2$.
	Then $E|F$ is ramified and $O_E = O_F[\Pi]$.
	Here, $E$ is generated by the square root of a unit in $F$.
	Indeed, for $\vartheta = 1 - 2\Pi/t$ we have $\vartheta^2 = 1 - 4\pi_0/t^2 \in O_F^{\times}$.
	Thus $E|F$ is said to be of type (R-U) (for ``ramified-unit'').
\end{itemize}

We will use this notation throughout the paper.

\begin{rmk} 
		The isomorphism classes of quadratic extension of $F$ correspond to the non-trivial equivalence classes of $F^{\times} / (F^{\times})^2$.
		We have $F^{\times} / (F^{\times})^2 \simeq \operatorname{H}^1(G_F, \mathbb{Z}/2\mathbb{Z})$ for  the absolute Galois group $G_F$ of $F$ and $\dim \operatorname{H}^1(G_F, \mathbb{Z}/2\mathbb{Z}) = 2 + d$, where $d = [F: \mathbb{Q}_2]$ is the degree of $F$ over $\mathbb{Q}_2$ (see, for example, \cite[Cor.\ 7.3.9]{NSW00}).
	
		A representative of an equivalence class in $F^{\times} / F^{\times 2}$ can be chosen to be either a prime or a unit, and exactly half of the classes are represented by prime elements, the others being represented by units.
		It follows that there are, up to isomorphism, $2^{1 + d}$ different extensions $E|F$ of type (R-P) and $2^{1+d} -2$ extension of type (R-U).
		(We have to exclude the trivial element $1 \in F^{\times} / F^{\times 2}$ and one unit element corresponding to the unramified extension.)
\end{rmk}

\begin{lem} \label{LA_diff}
	The inverse different of $E|F$ is given by $\mathfrak{D}_{E|F}^{-1} = \frac{1}{2\Pi}O_E$ in the case \emph{(R-P)} and by $\mathfrak{D}_{E|F}^{-1} = \frac{1}{t}O_E$ in the case \emph{(R-U)}.
\end{lem}

\begin{proof}
	The inverse different is defined as
	\begin{equation*}
		\mathfrak{D}_{E|F}^{-1} = \{ \alpha \in E \mid \Tr_{E|F}(\alpha O_E) \subseteq O_F \}.
	\end{equation*}
	It is enough to check the condition on the trace for the elements $1$ and $\Pi \in O_E$.
	If we write $\alpha = \alpha_1 + \Pi \alpha_2$ with $\alpha_1,\alpha_2 \in F$, we get
	\begin{align*}
		\Tr_{E|F}(\alpha \cdot 1) & = \alpha + \overbar{\alpha} = 2\alpha_1 + \alpha_2(\Pi + \overbar{\Pi}), \\
		\Tr_{E|F}(\alpha \cdot \Pi) & = \alpha \Pi + \overbar{\alpha \Pi} = \alpha_1 (\Pi + \overbar{\Pi}) + \alpha_2(\Pi^2 + \overbar{\Pi}^2).
	\end{align*}
	In the case (R-P) we have $\Pi + \overbar{\Pi} = 0$ and $\Pi^2 + \overbar{\Pi}^2 = 2\pi_0$, while in the case (R-U), $\Pi + \overbar{\Pi} = t$ and $\Pi^2 + \overbar{\Pi}^2 = t^2 - 2\pi_0$.
	It is now easy to deduce that the inverse different is of the claimed form.
\end{proof}

Over $F$, there exists up to isomorphism exactly one quaternion division algebra $B$, with unique maximal order $O_B$.
For every quadratic extension $E|F$, there exists an embedding $E \inj B$ and this induces an embedding $O_E \inj O_B$.
If $E|F$ is ramified, a basis for $O_E$ as $O_F$-module is given by $(1, \Pi)$.
We would like to extend this to an $O_F$-basis of $O_B$.

\begin{lem} \label{LA_quat}
	\begin{enumerate}
		\item \label{LA_quatRP} If $E|F$ is of type \emph{(R-P)}, there exists an embedding $F^{(2)} \inj B$ such that $\delta \Pi = - \Pi \delta$.
		An $O_F$-basis of $O_B$ is then given by $(1,\gamma,\Pi,\gamma \cdot \Pi)$, where $\gamma = \frac{1+\delta}{2}$.
		\item \label{LA_quatRU} If $E|F$ is of type \emph{(R-U)}, there exists an embedding $E_1 \inj B$, where $E_1|F$ is of type \emph{({R-P})} with uniformizer $\Pi_1$ such that $\vartheta \Pi_1 = - \Pi_1 \vartheta$.
		The tuple $(1,\vartheta,\Pi_1,\vartheta \Pi_1)$ is an $F$-basis of $B$.
		
		Furthermore, there is also an embedding $\widetilde{E} \inj B$ with $\widetilde{E}|F$ of type \emph{(R-U)} with elements $\widetilde{\Pi}$ and $\widetilde{\vartheta}$ as above, such that $\vartheta \widetilde{\vartheta} = - \widetilde{\vartheta} \vartheta$ and $\widetilde{\vartheta}^2 = 1 + (t^2/\pi_0) \cdot u$ for some unit $u \in F$.
		In terms of this embedding, an $O_F$-basis of $O_B$ is given by $(1,\Pi,\widetilde{\Pi},\Pi \cdot \widetilde{\Pi} / \pi_0)$.
		Also,
		\begin{equation} \label{LA_quatRUunr}
			\frac{\Pi \cdot \widetilde{\Pi}}{\pi_0} = \gamma
		\end{equation}
		for some embedding $F^{(2)} \inj B$ of the unramified extension and $\gamma^2 - \gamma - u = 0$.
		Hence, $O_B = O_F[\Pi, \gamma]$ as $O_F$-algebra.
	\end{enumerate}
\end{lem}

\begin{proof}
	\eqref{LA_quatRP} This is \cite[II.\ Cor.\ 1.7]{Vig80}.

	\eqref{LA_quatRU} By \cite[I.\ Cor.\ 2.4]{Vig80}, it suffices to find a uniformizer $\Pi_1^2 \in F^{\times} \setminus \Nm_{E|F}(E^{\times})$ in order to prove the first part.
	But $\Nm_{E|F}(E^{\times}) \subseteq F^{\times}$ is a subgroup of order $2$ and $F^{\times 2} \subseteq \Nm_{E|F}(E^{\times})$.
	On the other hand, the residue classes of uniformizing elements in $F^{\times} / F^{\times 2}$ generate the whole group.
	Thus they cannot all be contained in $\Nm_{E|F}(E^{\times})$.
	
	For the second part, choose a unit $\delta \in F^{(2)}$ with $\delta^2 = 1 + 4 u \in F^{\times} \setminus F^{\times2}$ for some $u \in O_F^{\times}$ and set $\gamma = \frac{1+\delta}{2}$.
	Let $\widetilde{E}|F$ be of type (R-U), generated by $\widetilde{\vartheta}$ with $\widetilde{\vartheta}^2 = 1 + (t^2/\pi_0)  \cdot u$.
	We have to show that $\widetilde{\vartheta}^2$ is not contained in $\Nm_{E|F}(E^{\times})$.
	
	Assume it is a norm, so $\widetilde{\vartheta}^2 = \Nm_{E|F}(b)$ for a unit $b \in E^{\times}$.
	Then $b$ is of the form $b = 1 + x \cdot (t/ \Pi)$ for some $x \in O_E$.
	Indeed, let $\ell$ be the $\Pi$-adic valuation of $b-1$, \emph{i.e.}, $b = 1 + x \cdot \Pi^{\ell}$ and $x \in O_E^{\times}$.
	We have
	\begin{equation} \label{LA_Nmeq}
		1 + (t^2/ \pi_0) \cdot u = \Nm_{E|F}(b) = 1 + \Tr_{E|F}(x \Pi^{\ell}) + \Nm_{E|F}(x \Pi^{\ell})
	\end{equation}
	Let $v$ be the $\pi_0$-adic valuation on $F$.
	Then $v(\Nm_{E|F}(x \Pi^{\ell})) = \ell$ and $v(\Tr_{E|F}(x \Pi^{\ell})) \geq v(t) + \lfloor \frac{\ell}{2} \rfloor$, by Lemma \ref{LA_diff}.
	On the left hand side, we have $v((t^2/ \pi_0) \cdot u) = 2v(t) -1$.
	Comparing the valuations on both sides of \eqref{LA_Nmeq}, the assumption $\ell < 2v(t)-1$ now quickly leads to a contradiction.
	
	Hence $\ell \geq 2v(t)-1$ and $b = 1 + x \cdot (t/ \Pi)$ for some $x \in O_E$.
	Again,
	\begin{equation*}
		1 + (t^2/ \pi_0) \cdot u = \Nm_{E|F}(b) = 1 + \Tr_{E|F}(xt/ \Pi) + \Nm_{E|F}(xt/ \Pi).
	\end{equation*}
	An easy calculation shows that the residue $\overbar{x} \in k = O_E / \Pi = O_F / \pi_0$ of $x$ satisfies $u = x + x^2$.
	But this equation has no solution in $k$, since a solution of $\gamma^2 - \gamma - u = 0$ generates the unramified quadratic extension of $F$.
	It follows that $\widetilde{\vartheta}^2$ cannot be a norm.
	
	Using again \cite[I.\ Cor.\ 2.4]{Vig80}, we find an embedding $\widetilde{E} \inj B$ such that $\vartheta \widetilde{\vartheta} = - \widetilde{\vartheta} \vartheta$.
	

	We have $\Pi = t(1 + \vartheta) / 2$ and $\widetilde{\Pi} = \pi_0(1 + \widetilde{\vartheta}) / t$, thus
	\begin{equation*}
		\frac{\Pi \cdot \widetilde{\Pi}}{\pi_0} = \frac{(1 + \vartheta) \cdot (1 + \widetilde{\vartheta})}{2} = \frac{1 + \vartheta + \widetilde{\vartheta} + \vartheta \cdot \widetilde{\vartheta}}{2},
	\end{equation*}
	and
	\begin{align*}
		(\vartheta + \widetilde{\vartheta} + \vartheta \cdot \widetilde{\vartheta})^2 & = \vartheta^2 + \widetilde{\vartheta}^2 - \vartheta^2 \cdot \widetilde{\vartheta}^2 \\
		& = (1 - 4\pi_0/t^2) + (1 + t^2u/\pi_0) - (1 - 4\pi_0/t^2) (1 + t^2u/\pi_0) \\
		& = 1 + 4u.
	\end{align*}
	Hence $\gamma \mapsto \frac{\Pi \cdot \widetilde{\Pi}}{\pi_0}$ induces an embedding $F^{(2)} \inj B$.
	
	It remains to prove that the tuple $u = (1,\Pi,\widetilde{\Pi},\Pi \cdot \widetilde{\Pi} / \pi_0)$ is a basis of $O_B$ as $O_F$-module.
	By \cite[I.\ Cor.\ 4.8]{Vig80}, it suffices to check that the discriminant
	\begin{equation*}
		\disc(u) = \det(\Trd(u_i u_j)) \cdot O_F
	\end{equation*}
	is equal to $\disc(O_B)$.
	An easy calculation shows $\det(\Trd(u_i u_j)) \cdot O_F = \pi_0 O_F$ and then the assertion follows from \cite[V, II.\ Cor.\ 1.7]{Vig80}.
\end{proof}

For the remainder of this section, we will consider lattices $\Lambda$ in a $2$-dimensional $E$-vector space $C$ with a split $E|F$-hermitian\footnote{Here and in the following, sesquilinear forms will be linear from the left and semi-linear from the right.} form $h$.
Recall from \cite{Jac62} that, up to isomorphism, there are $2$ different $E|F$-hermitian vector spaces $(C,h)$ of fixed dimension $n$, parametrized by the discriminant $\disc(C,h) \in F^{\times} / \Nm_{E|F}(E^{\times})$.
A hermitian space $(C,h)$ is called \emph{split} whenever $\disc(C,h) = 1$.
In our case, where $(C,h)$ is split of dimension $2$, we can find a basis $(e_1,e_2)$ of $C$ with $h(e_i,e_i) = 0$ and $h(e_1,e_2) = 1$.

Denote by $\Lambda^{\sharp}$ the dual of a lattice $\Lambda \subseteq C$ with respect to $h$.
The lattice $\Lambda$ is called $\Pi^i$-\emph{modular} if $\Lambda = \Pi^i \Lambda^{\sharp}$ (resp.\ \emph{unimodular} or \emph{selfdual} when $i=0$).
In contrast to the $p$-adic case with $p > 2$, there exists more than one type of $\Pi^i$-modular lattices in our case (\emph{cf.}\ \cite{Jac62}):

\begin{prop} \label{LA_latt}
	Define the norm ideal $\Nm(\Lambda)$ of $\Lambda$ by
	\begin{equation}
		\Nm(\Lambda) = \langle \{ h(x,x) | x \in \Lambda \} \rangle \subseteq F.
	\end{equation}
	Any $\Pi^i$-modular lattice $\Lambda \subseteq C$ is determined up to the action of $\U(C,h)$ by the ideal $\Nm(\Lambda) = \pi_0^{\ell} O_F \subseteq F$.
	For $i= 0$ or $1$, the exponent $\ell$ can be any integer such that 
	\begin{align*}
		|2| & \leq |\pi_0|^{\ell} \leq |1| \text{ (for } E|F \text{ (R-P), unimodular } \Lambda), \\
		|2\pi_0| & \leq |\pi_0|^{\ell} \leq |\pi_0| \text{ (for } E|F \text{ (R-P), } \Pi \text{-modular } \Lambda), \\
		|t| & \leq |\pi_0|^{\ell} \leq |1| \text{ (for } E|F \text{ (R-U), unimodular } \Lambda), \\
		|t| & \leq |\pi_0|^{\ell} \leq |\pi_0| \text{ (for } E|F \text{ (R-U), } \Pi \text{-modular } \Lambda),
	\end{align*}
	where $|\cdot|$ is the (normalized) absolute value on $F$.
	Two $\Pi^i$-modular lattices $\Lambda$ and $\Lambda'$ are isomorphic if and only if $\Nm(\Lambda) = \Nm(\Lambda')$.
	\qed
\end{prop}

For any other $i$, the possible values of $\ell$ for a given $\Pi^i$-modular lattice $\Lambda$ are easily obtained by shifting.
In fact, we can choose an integer $j$ such that $\Pi^j \Lambda$ is either unimodular or $\Pi$-modular.
Then $\Nm (\Lambda) = \pi_0^{-j} \Nm (\Pi^j \Lambda)$ and we can apply the proposition above.

Since $(C,h)$ is split, any $\Pi^i$-modular lattice $\Lambda$ contains an \emph{isotropic} vector $v$ (\emph{i.e.}, with $h(v,v) = 0$).
After rescaling with a suitable power of $\Pi$, we can extend $v$ to a basis of $\Lambda$.
Hence there always exists a basis $(e_1,e_2)$ of $\Lambda$ such that $h$ is represented by a matrix of the form
\begin{equation} \label{LA_lattmatrix}
	H_{\Lambda} = \begin{pmatrix}
					  x & \overbar{\Pi}^i \\
		\Pi^i &		\\
	\end{pmatrix}, \quad x \in F.
\end{equation}
If $x = 0$ in this representation, then $\Nm(\Lambda) = \pi_0^{\ell} O_F$ is as small as possible, or in other words, the absolute value of $|\pi_0|^{\ell}$ is minimal.
On the other hand, whenever $|\pi_0|^{\ell}$ takes the minimal absolute value for a given $\Pi^i$-modular lattice $\Lambda$, there exists a basis $(e_1,e_2)$ of $\Lambda$ such that $h$ is represented by $H_{\Lambda}$ with $x = 0$.
Indeed, this follows because the ideal $\Nm(\Lambda)$ already determines $\Lambda$ up to isomorphism.
In this case (when $x = 0$), we call $\Lambda$ a \emph{hyperbolic} lattice.
By the arguments above, a $\Pi^i$-modular lattice is thus hyperbolic if and only if its norm is minimal.
In all other cases, where $\Lambda$ is $\Pi^i$-modular but not hyperbolic, we have $\Nm(\Lambda) = x O_F$.

For further reference, we explicitly write down the norm of a hyperbolic lattice for the cases that we need later.
For other values of $i$, the norm can easily be deduced from this by shifting (see also \cite[Table 9.1]{Jac62}).

\begin{lem} \label{LA_hyp}
	A $\Pi^i$-modular lattice $\Lambda$ is hyperbolic if and only if
	\begin{align*}
		\Nm(\Lambda) & = 2 O_F, & & \text{for } E|F \text{ (R-P), } i = 0 \text{ or } -1, \\
		\Nm(\Lambda) & = t O_F, & & \text{for } E|F \text{ (R-U), } i = 0 \text{ or } 1.
	\end{align*}
	The norm ideal of $\Lambda$ is minimal among all norm ideals for $\Pi^i$-modular lattices in $C$.
	\qed
\end{lem}

In the following, we will only consider the cases $i = 0$ or $-1$ for $E|F$ (R-P) and the cases $i = 0$ or $1$ for $E|F$ (R-U), since these are the cases we will need later.
We want to study the following question:

\begin{qu} \label{LA_lattqu}
	Assume $E|F$ is (R-P).
	Fix a $\Pi^{-1}$-modular lattice $\Lambda_{-1} \subseteq C$ (not necessarily hyperbolic).
	How many unimodular lattices $\Lambda_0 \subseteq \Lambda_{-1}$ are there and what norms $\Nm(\Lambda_0)$ can appear?
	Dually, for a fixed unimodular lattice $\Lambda_0 \subseteq C$, how many $\Pi^{-1}$-modular lattices $\Lambda_{-1}$ with $\Lambda_0 \subseteq \Lambda_{-1}$ do exist and what are their norms?
	
	Same question for $E|F$ (R-U) and unimodular resp.\ $\Pi$-modular lattices.
\end{qu}

Of course, such an inclusion is always of index $1$.
The inclusions $\Lambda_0 \subseteq \Lambda_{-1}$ of index $1$ correspond to lines in $\Lambda_{-1} / \Pi \Lambda_{-1}$.
Denote by $q$ the number of elements in the common residue field of $O_F$ and $O_E$.
Then there exist at most $q+1$ such $\Pi$-modular lattices $\Lambda_0$ for a given $\Lambda_{-1}$.
The same bound holds in the dual case, \emph{i.e.}, there are at most $q+1$ $\Pi^{-1}$-modular lattices containing a given unimodular lattice $\Lambda_0$.
The Propositions \ref{LA_lattRP} and \ref{LA_lattRU} below provide an exhaustive answer to Question \ref{LA_lattqu}.
Since the proofs consist of a lengthy but simple case-by-case analysis, we will leave it to the interested reader.

\begin{prop} \label{LA_lattRP} Let $E|F$ of type \emph{(R-P)}.
	\begin{enumerate}
		\item Let $\Lambda_{-1} \subseteq C$ be a $\Pi^{-1}$-modular hyperbolic lattice.
		There are $q+1$ hyperbolic unimodular lattices contained in $\Lambda_{-1}$.
		\item Let $\Lambda_{-1} \subseteq C$ be a $\Pi^{-1}$-modular non-hyperbolic lattice.
		Let $\Nm(\Lambda_{-1}) = \pi_0^{\ell}O_F$.
		Then $\Lambda_{-1}$ contains one unimodular lattice $\Lambda_0$ with $\Nm(\Lambda_0) = \pi_0^{\ell+1}O_F$ and $q$ unimodular lattices of norm $\pi_0^{\ell}O_F$.
		\item Let $\Lambda_0 \subseteq C$ be a unimodular hyperbolic lattice.
		There are two hyperbolic $\Pi^{-1}$-modular lattices $\Lambda_{-1}\supseteq \Lambda_0$ and $q-1$ non-hyperbolic $\Pi^{-1}$-modular lattices $\Lambda_{-1} \supseteq \Lambda_0$ with $\Nm(\Lambda_{-1}) = 2/\pi_0 O_F$.
		\item \label{LA_lattRPmin} Let $\Lambda_0 \subseteq C$ be unimodular non-hyperbolic.
		Let $\Nm(\Lambda_0) = \pi_0^{\ell}O_F$.
		There exists one $\Pi^{-1}$-modular lattice $\Lambda_{-1} \supseteq \Lambda_0$ with $\Nm(\Lambda_{-1}) = \pi_0^{\ell}O_F$ and, unless $\ell = 0$, there are $q$ non-hyperbolic $\Pi^{-1}$-modular lattices $\Lambda_{-1} \supseteq \Lambda_0$ with $\Nm(\Lambda_{-1}) = \pi_0^{\ell-1}O_F$.
	\end{enumerate}
\end{prop}

Note that the total amount of unimodular resp.\ $\Pi^{-1}$-modular lattices found for $\Lambda = \Lambda_{-1}$ resp.\ $\Lambda_0$ is $q+1$ except in the case of Proposition \ref{LA_lattRP} \eqref{LA_lattRPmin} when $\ell = 0$.
In that particular case, there is just one $\Pi^{-1}$-modular lattice contained in $\Lambda_0$.
The same phenomenon also appears in the case (R-U), see part \eqref{LA_lattRUmin} of the following proposition.

\begin{prop} \label{LA_lattRU} Let $E|F$ of type \emph{(R-U)}.
	\begin{enumerate}
		\item Let $\Lambda_0 \subseteq C$ be a unimodular hyperbolic lattice.
		There are $q+1$ hyperbolic $\Pi$-modular lattices $\Lambda_1 \subseteq \Lambda_0$.
		\item \label{LA_lattRUmin} Let $\Lambda_0 \subseteq C$ be unimodular non-hyperbolic with $\Nm(\Lambda_0) = \pi_0^{\ell}O_F$.
		There is one $\Pi$-modular lattice $\Lambda_1 \subseteq \Lambda_0$ with norm ideal $\Nm(\Lambda_1) = \pi_0^{\ell+1}O_F$ and if $\ell \neq 0$, there are also $q$ non-hyperbolic $\Pi$-modular lattices $\Lambda_1 \subseteq \Lambda_0$ with $\Nm(\Lambda_1) = \pi_0^{\ell}O_F$.
		\item \label{LA_lattRUh} Let $\Lambda_1 \subseteq C$ be a $\Pi$-modular hyperbolic lattice.
		There are two unimodular hyperbolic lattices containing $\Lambda_1$ and $q-1$ unimodular lattices $\Lambda_0$ with $\Lambda_1 \subseteq \Lambda_0$ and $\Nm(\Lambda_0) =  t/ \pi_0 O_F$.
		\item \label{LA_lattRUnh} Let $\Lambda_1 \subseteq C$ be a $\Pi$-modular non-hyperbolic lattice and let $\Nm(\Lambda_1) = \pi_0^{\ell}O_F$.
		The lattice $\Lambda_1$ is contained in $q$ unimodular lattices of norm $\pi_0^{\ell-1}O_F$ and in one unimodular lattice $\Lambda_0$ with $\Nm(\Lambda_0) = \pi_0^{\ell}O_F$.
	\end{enumerate}
\end{prop}

%

If $E|F$ is a quadratic extension of type (R-U) such that $|t| = |\pi_0|$, there exist only hyperbolic $\Pi$-modular lattices in $C$ and hence case \eqref{LA_lattRUnh} of Proposition \ref{LA_lattRU} does not appear.


\section{The moduli problem in the case (R-P)}
\label{RP}

Throughout this section, $E|F$ is a quadratic extension of type (R-P), \emph{i.e.}, there exist uniformizing elements $\pi_0 \in F$ and $\Pi \in E$ such that $\Pi^2 + \pi_0 = 0$.
Then $O_E = O_F[\Pi]$ for the rings of integers $O_F$ and $O_E$ of $F$ and $E$, respectively.
Let $k$ be the common residue field with $q$ elements, $\overbar{k}$ an algebraic closure, and $\breve{F}$ the completion of the maximal unramified extension of $F$, with ring of integers $\breve{O}_F = W_{O_F}(\overbar{k})$.
Let $\sigma$ be the lift of the Frobenius in $\Gal(\overbar{k}|k)$ to $\Gal(\breve{O}_F|O_F)$.

\subsection{The definition of the naive moduli problem $\NEnaive$}

We first construct a functor $\NEnaive$ on $\Nilp$, the category of $\breve{O}_F$-schemes $S$ such that $\pi_0 \mathcal{O}_S$ is locally nilpotent.
We consider tuples $(X, \iota, \lambda)$, where
\begin{itemize}
	\item $X$ is a formal $O_F$-module over $S$ of dimension $2$ and height $4$.
	\item $\iota: O_E \to \End(X)$ is an action of $O_E$ satisfying the \emph{Kottwitz condition}:
	The characteristic polynomial of $\iota(\alpha)$ on $\Lie X$ for any $\alpha \in O_E$ is
	\begin{equation*}
		\charp(\Lie X, T \mid \iota(\alpha)) = (T - \alpha)(T - \overbar{\alpha}).
	\end{equation*}
	Here $\alpha \mapsto \overbar{\alpha}$ is the non-trivial Galois automorphism and the right hand side is a polynomial with coefficients in $\mathcal{O}_S$ via the composition $O_F[T] \inj \breve{O}_F[T] \to \mathcal{O}_S[T]$.
	
	\item $\lambda: X \to X^{\vee}$ is a principal polarization on $X$ such that the Rosati involution satisfies $\iota(\alpha)^{\ast} = \iota(\overbar{\alpha})$ for $\alpha \in O_E$.
\end{itemize}

\begin{defn} \label{RP_isonaive}
	A \emph{quasi-isogeny} (resp.\ an \emph{isomorphism}) $\varphi: (X,\iota,\lambda) \to (X', \iota', \lambda')$ of two such tuples $(X,\iota,\lambda)$ and $(X', \iota', \lambda')$ over $S$ is an $O_E$-linear quasi-isogeny of height $0$ (resp.\ an $O_E$-linear isomorphism) $\varphi: X \to X'$ such that $\lambda = \varphi^{\ast}(\lambda')$.
	
	Denote the group of quasi-isogenies $\varphi: (X,\iota,\lambda) \to (X,\iota,\lambda)$ by $\QIsog(X,\iota,\lambda)$.
\end{defn}

For $S = \Spec \overbar{k}$ we have the following proposition:

\begin{prop} \label{RP_frnaive}
	Up to isogeny, there exists precisely one tuple $(\mathbb{X},\iota_{\mathbb{X}},\lambda_{\mathbb{X}})$ over $\Spec \overbar{k}$ such that the group $\QIsog(\mathbb{X},\iota_{\mathbb{X}},\lambda_{\mathbb{X}})$ contains $\SU(C,h)$ as a closed subgroup.
	Here $\SU(C,h)$ is the special unitary group for a $2$-dimensional $E$-vector space $C$ with split $E|F$-hermitian form $h$.
\end{prop}


\begin{rmk} \label{RP_rmknaive}
	If $(\mathbb{X},\iota_{\mathbb{X}},\lambda_{\mathbb{X}})$ is as in the proposition, we always have $\QIsog(\mathbb{X},\iota_{\mathbb{X}},\lambda_{\mathbb{X}}) \cong \U(C,h)$.
	This follows directly from the proof and gives a more natural way to describe the framing object.
	However, we will need the slightly stronger statement of the Proposition later, in Lemma \ref{RP_clemb}.
\end{rmk}

\begin{proof}[Proof of Proposition \ref{RP_frnaive}] 
	We first show uniqueness.
	Let $(X,\iota,\lambda) / \Spec \overbar{k}$ be such a tuple.
	Its (relative) rational Dieudonn\'e module $N_X$ is a $4$-dimensional vector space over $\breve{F}$ with an action of $E$ and an alternating form $\langle \,,\rangle$ such that for all $x,y \in N_X$,
	\begin{equation} \label{RP_alt}
		\langle x, \Pi y \rangle = - \langle \Pi x, y \rangle.
	\end{equation}
	
	The space $N_X$ has the structure of a $2$-dimensional vector space over $\breve{E} = E \otimes_F \breve{F}$ and we can define an $\breve{E}|\breve{F}$-hermitian form on it via
	\begin{equation} \label{RP_herm}
		h(x,y) = \langle \Pi x, y \rangle + \Pi \langle x, y \rangle.
	\end{equation}
	The alternating form can be recovered from $h$ by
	\begin{equation} \label{RP_altherm}
		\langle x, y \rangle = \Tr_{\breve{E}|\breve{F}} \left(\frac{1}{2 \Pi} \cdot h(x.y) \right).
	\end{equation}
	Furthermore we have on $N_X$ a $\sigma$-linear operator ${\bf F}$, the Frobenius, and a $\sigma^{-1}$-linear operator ${\bf V}$, the Verschiebung, that satisfy ${\bf V}{\bf F} = {\bf F}{\bf V} = \pi_0$.
	Recall that $\sigma$ is the lift of the Frobenius on $\breve{O}_F$.
	Since $\langle \,, \rangle$ comes from a polarization, we have 
	\begin{align*}
		\langle {\bf F}x, y \rangle & = \langle x, {\bf V}y \rangle ^{\sigma}, \\
	\intertext{and} 
		h({\bf F}x, y) & = h(x,{\bf V}y)^{\sigma},
	\end{align*}
	for all $x,y \in N_X$.
	Let us consider the $\sigma$-linear operator $\tau = \Pi {\bf V}^{-1}$.
	Its slopes are all zero, since $N_X$ is isotypical of slope $\frac{1}{2}$.
	(This follows from the condition on $\QIsog(\mathbb{X},\iota_{\mathbb{X}},\lambda_{\mathbb{X}})$.)
	We set $C = N_X^{\tau}$.
	This is a $2$-dimensional vector space over $E$ and $N_X = C \otimes_E \breve{E}$.
	Now $h$ induces an $E|F$-hermitian form on $C$ since
	\begin{equation*}
		h(\tau x, \tau y) = h(-{\bf F} \Pi^{-1} x, \Pi {\bf V}^{-1}y) = - h(\Pi^{-1} x, \Pi y)^{\sigma} = h(x,y)^{\sigma}.
	\end{equation*}
	A priori, there are up to isomorphism two possibilities for $(C,h)$, either $h$ is split on $C$ or non-split.
	But automorphisms of $(C,h)$ correspond to elements of $\QIsog(\mathbb{X},\iota_{\mathbb{X}},\lambda_{\mathbb{X}})$.
	The unitary groups of $(C,h)$ for $h$ split and $h$ non-split are not isomorphic and they cannot contain each other as a closed subgroup.
	Hence the condition on $\QIsog(\mathbb{X},\iota_{\mathbb{X}},\lambda_{\mathbb{X}})$ implies that $h$ is split.
	
	Assume now we have two different objects $(X,\iota,\lambda)$ and $(X',\iota',\lambda')$ as in the proposition.
	These give us isomorphic vector spaces $(C,h)$ and $(C',h')$ and an isomorphism between these extends to an isomorphism between $N_X$ and $N_X'$ (respecting all rational structure) which corresponds to a quasi-isogeny between $(X,\iota,\lambda)$ and $(X',\iota',\lambda')$.
	
	The existence of $(\mathbb{X},\iota_{\mathbb{X}},\lambda_{\mathbb{X}})$ now follows from the fact that a $2$-dimensional $E$-vector space $(C,h)$ with split $E|F$-hermitian form contains a unimodular lattice $\Lambda$.
	Indeed, this gives us a lattice $M = \Lambda \otimes_{O_E} \breve{O}_E \subseteq C \otimes_E \breve{E}$.
	We extend $h$ to $N = C \otimes_E \breve{E}$ and define the $\breve{F}$-linear alternating form $\langle \,, \rangle$ as in \eqref{RP_altherm}.
	Now $M$ is unimodular with respect to $\langle \,, \rangle$, because $\frac{1}{2 \Pi} \breve{O}_E$ is the inverse different of $\breve{E}|\breve{F}$ (see Lemma \ref{LA_diff}).
	We choose the operators ${\bf F}$ and ${\bf V}$ on $M$ such that ${\bf F}{\bf V} = {\bf V}{\bf F} = \pi_0$ and $\Lambda = M^{\tau}$ for $\tau = \Pi {\bf V}^{-1}$.
	This makes $M$ a (relative) Dieudonn\'e module and we define $(\mathbb{X},\iota_{\mathbb{X}},\lambda_{\mathbb{X}})$ as the corresponding formal $O_F$-module.
\end{proof}

We fix such a framing object $(\mathbb{X},\iota_{\mathbb{X}},\lambda_{\mathbb{X}})$ over $\Spec \overbar{k}$.

\begin{defn} \label{RP_defnaive}
	For arbitrary $S \in \Nilp$, let $\overbar{S} = S \times_{\Spf \breve{O}_F} \Spec \overbar{k}$.
	Define $\NEnaive(S)$ as the set of equivalence classes of tuples $(X,\iota,\lambda,\varrho)$ over $S$, where $(X,\iota,\lambda)$ as above and
	\begin{equation*}
		\varrho: X \times_S \overbar{S} \to \mathbb{X} \times_{\Spec \overbar{k}} \overbar{S}
	\end{equation*}
	is a quasi-isogeny between the tuple $(X,\iota,\lambda)$ and the framing object $(\mathbb{X},\iota_{\mathbb{X}},\lambda_{\mathbb{X}})$ (after base change to $\overbar{S}$).
	Two objects $(X,\iota,\lambda,\varrho)$ and $(X',\iota',\lambda',\varrho')$ are equivalent if and only if there exists an isomorphism $\varphi: (X,\iota,\lambda) \to (X',\iota',\lambda')$ such that $\varrho = \varrho' \circ (\varphi \times_S \overbar{S})$.
\end{defn}

\begin{rmk} \label{RP_loftnaive}
	\begin{enumerate}
		\item \label{RP_loftnaive1} The morphism $\varrho$ is a quasi-isogeny in the sense of Definition \ref{RP_isonaive}, \emph{i.e.}, we have $\lambda = \varrho^{\ast}(\lambda_{\mathbb{X}})$.
		Similarly, we have $\lambda = \varphi^{\ast}(\lambda')$ for the isomorphism $\varphi$.
		We obtain an equivalent definition of $\NEnaive$ if we replace strict equality by the condition that, locally on $S$, $\lambda$ and $\varrho^{\ast}(\lambda_{\mathbb{X}})$ (resp.\ $\varphi^{\ast}(\lambda')$) only differ by a scalar in $O_F^{\times}$.
		This variant is used in the definition of RZ-spaces of (PEL) type for $p > 2$ in \cite{RZ96}.
		In this paper we will use the version with strict equality, since it simplifies the formulation of the straightening condition, see Definition \ref{RP_strdef} below.
		
		\item \label{RP_loftnaive2} $\NEnaive$ is pro-representable by a formal scheme, formally locally of finite type over $\Spf \breve{O}_F$.
	This follows from \cite[Thm.\ 3.25]{RZ96}.
	\end{enumerate}
\end{rmk}

As a next step, we use Dieudonn\'e theory in order to get a better understanding of the special fiber of $\NEnaive$.
Let $N = N_{\mathbb{X}}$ be the rational Dieudonn\'e module of the base point $(\mathbb{X},\iota_{\mathbb{X}},\lambda_{\mathbb{X}})$ of $\NEnaive$.
This is a $4$-dimensional vector space over $\breve{F}$, equipped with an $E$-action, an alternating form $\langle \,, \rangle$ and two operators ${\bf V}$ and ${\bf F}$.
As in the proof of Proposition \ref{RP_frnaive}, the form $\langle \,, \rangle$ satisfies condition \eqref{RP_alt}:
\begin{equation}
	\langle x, \Pi y \rangle = - \langle \Pi x, y \rangle.
\end{equation}

A point $(X,\iota,\lambda,\varrho) \in \NEnaive(\overbar{k})$ corresponds to an $\breve{O}_F$-lattice $M_X \subseteq N$.
It is stable under the actions of the operators ${\bf V}$ and ${\bf F}$ and of the ring $O_E$.
Furthermore $M_X$ is unimodular under $\langle \,, \rangle$, \emph{i.e.}, $M_X = M_X^{\vee}$, where
\begin{equation*}
	M_X^{\vee} = \{ x \in N \mid \langle x, y \rangle \in \breve{O}_F \text{ for all } y \in M_X \}.
\end{equation*}
We can regard $N$ as a $2$-dimensional vector space over $\breve{E}$ with the $\breve{E}|\breve{F}$-hermitian form $h$ defined by
\begin{equation}
	h(x,y) = \langle \Pi x, y \rangle + \Pi \langle x, y \rangle.
\end{equation}
Let $\breve{O}_E = O_E \otimes_{O_F} \breve{O}_F$.
Then $M_X \subseteq N$ is an $\breve{O}_E$-lattice and we have
\begin{equation*}
	M_X = M_X^{\vee} = M_X^{\sharp},
\end{equation*}
where $M_X^{\sharp}$ is the dual lattice of $M_X$ with respect to $h$.
The latter equality follows from the formula
\begin{equation}
	\langle x, y \rangle = \Tr_{\breve{E}|\breve{F}} \left(\frac{1}{2 \Pi} \cdot h(x.y) \right)
\end{equation}
and the fact that the inverse different of $E|F$ is $\mathfrak{D}_{E|F}^{-1} = \frac{1}{2\Pi}O_E$ (see Lemma \ref{LA_diff}).
We can thus write the set $\NEnaive(\overbar{k})$ as
\begin{equation} \label{RP_geompts}
	\NEnaive(\overbar{k}) = \{ \breve{O}_E \text{-lattices } M \subseteq N_{\mathbb{X}} \mid M^{\sharp} = M, \pi_0 M \subseteq {\bf V}M \subseteq M \}.
\end{equation}
Let $\tau = \Pi {\bf V}^{-1}$.
This is a $\sigma$-linear operator on $N$ with all slopes zero.
The elements invariant under $\tau$ form a $2$-dimensional $E$-vector space $C = N^{\tau}$.
The hermitian form $h$ is invariant under $\tau$, hence it induces a split hermitian form on $C$ which we denote again by $h$.
With the same proof as in \cite[Lemma 3.2]{KR14}, we have:

\begin{lem} \label{RP_latt}
	Let $M \in \NEnaive(\overbar{k})$.
	Then:
	\begin{enumerate}
		\item $M + \tau(M)$ is $\tau$-stable.
		\item Either $M$ is $\tau$-stable and $\Lambda_0 = M^{\tau} \subseteq C$ is unimodular $(\Lambda_0^{\sharp} = \Lambda_0)$ or $M$ is not $\tau$-stable and then $\Lambda_{-1} = (M + \tau(M))^{\tau} \subseteq C$ is $\Pi^{-1}$-modular $(\Lambda_{-1}^{\sharp} = \Pi \Lambda_{-1})$.
	\end{enumerate}
\end{lem}

Under the identification $N = C \otimes_E \breve{E}$, we get $M = \Lambda_0 \otimes_{O_E} \breve{O}_E$ for a $\tau$-stable Dieudonn\'e lattice $M$.
If $M$ is not $\tau$-stable, we have $M + \tau M = \Lambda_{-1} \otimes_{O_E} \breve{O}_E$ and $M \subseteq \Lambda_{-1} \otimes_{O_E} \breve{O}_E$ is a sublattice of index $1$.
The next lemma is the analogue of \cite[Lemma 3.3]{KR14}.

\begin{lem} \label{RP_IS}
	\begin{enumerate}
		\item \label{RP_IS1} Fix a $\Pi^{-1}$-modular lattice $\Lambda_{-1} \subseteq C$.
		There is an injective map
		\begin{equation*}
			i_{\Lambda_{-1}}: \mathbb{P}(\Lambda_{-1} / \Pi \Lambda_{-1})(\overbar{k}) \inj \NEnaive(\overbar{k})
		\end{equation*}
		mapping a line $\ell \subseteq (\Lambda_{-1} / \Pi \Lambda_{-1}) \otimes \overbar{k}$ to its preimage in $\Lambda_{-1} \otimes \breve{O}_E$.
		Identify $\mathbb{P}(\Lambda_{-1} / \Pi \Lambda_{-1})(\overbar{k})$ with its image in $\NEnaive(\overbar{k})$.
		Then $\mathbb{P}(\Lambda_{-1} / \Pi \Lambda_{-1})(k) \subseteq \mathbb{P}(\Lambda_{-1} / \Pi \Lambda_{-1})(\overbar{k})$ is the set of $\tau$-invariant Dieudonn\'e lattices $M \subseteq \Lambda_{-1} \otimes \breve{O}_E$.
		\item The set $\NEnaive(\overbar{k})$ is a union
		\begin{equation} \label{RP_ISunion}
			\NEnaive(\overbar{k}) = \smashoperator[r]{\bigcup_{\Lambda_{-1} \subseteq C}} \mathbb{P}(\Lambda_{-1} / \Pi \Lambda_{-1})(\overbar{k}),
		\end{equation}
		ranging over all $\Pi^{-1}$-modular lattices $\Lambda_{-1} \subseteq C$.
		The projective lines corresponding to the lattices $\Lambda_{-1}$ and $\Lambda_{-1}'$ intersect in $\NEnaive(\overbar{k})$ if and only if $\Lambda_0 = \Lambda_{-1} \cap \Lambda_{-1}'$ is unimodular.
		In this case, their intersection consists of the point $M = \Lambda_0 \otimes \breve{O}_E \in \NEnaive(\overbar{k})$.
	\end{enumerate}
\end{lem}

\begin{proof}
	We only have to prove that the map $i_{\Lambda_{-1}}$ is well-defined.
	Denote by $M$ the preimage of $\ell \subseteq (\Lambda_{-1} / \Pi \Lambda_{-1}) \otimes \overbar{k}$ in $\Lambda_{-1} \otimes \breve{O}_E$.
	We need to show that $M$ is an element in $\NEnaive(\overbar{k})$ under the identification of \eqref{RP_geompts}.
	It is clearly a sublattice of index $1$ in $\Lambda_{-1} \otimes \breve{O}_E$, stable under the actions of ${\bf F}$, ${\bf V}$ and $O_E$.
	
	Let $e_1 \in \Lambda_{-1}  \otimes \breve{O}_E$ such that $e_1 \otimes \overbar{k}$ generates $\ell$.
	We can extend this to a basis $(e_1,e_2)$ of $\Lambda_{-1}$ and with respect to this basis, $h$ is represented by a matrix of the form
	\begin{equation*}
		\begin{pmatrix}
					x & -\Pi^{-1} \\
			\Pi^{-1} &		   y \\
		\end{pmatrix},
	\end{equation*}
	with $x, y \in \Pi^{-1} \breve{O}_E \cap \breve{O}_F = \breve{O}_F$.
	The lattice $M \subseteq \Lambda_{-1} \otimes \breve{O}_E$ is generated by $e_1$ and $\Pi e_2$.
	With respect to this new basis, $h$ is now given by the matrix
	\begin{equation*}
		\begin{pmatrix*}
			x & 	   1 \\
			1 & \pi_0 y \\
		\end{pmatrix*}.
	\end{equation*}
	Since all entries of the matrix are integral, we have $M \subseteq M^{\sharp}$.
	But this already implies $M^{\sharp} = M$, because they both have index $1$ in $\Lambda_{-1} \otimes \breve{O}_E$.
	Thus $M \in \NEnaive(\overbar{k})$ and $i_{\Lambda_{-1}}$ is well-defined.
\end{proof}

\begin{rmk} \label{RP_rmklatt}
	\begin{enumerate}
		\item \label{RP_rmklatt1} Recall from Proposition \ref{LA_latt} that the isomorphism type of a $\Pi^i$-modular lattice $\Lambda \subseteq C$ only depends on its norm ideal $\Nm(\Lambda) = \langle \{ h(x,x) | x \in \Lambda \} \rangle = \pi_0^{\ell} O_F \subseteq F$.
		In the case that $\Lambda = \Lambda_0$ or $\Lambda_{-1}$ is unimodular or $\Pi^{-1}$-modular, $\ell$ can be any integer such that $|1| \geq |\pi_0|^{\ell} \geq |2|$.
		In particular, there are always at least two possible values for $\ell$.
		Recall from Lemma \ref{LA_hyp}, that $\Lambda$ is \emph{hyperbolic} if and only if $\Nm(\Lambda) = 2 O_F$.
		\item The intersection behaviour of the projective lines in $\NEnaive(\overbar{k})$ can be deduced from Proposition \ref{LA_lattRP}.
		In particular, for a given unimodular lattice $\Lambda_0 \subseteq C$ with $\Nm(\Lambda_0) \subseteq \pi_0 O_F$, there are $q+1$ lines intersecting in $M = \Lambda_0 \otimes \breve{O}_E$.
		If $\Nm(\Lambda_0) = O_F$, the lattice $M = \Lambda_0 \otimes \breve{O}_E$ is only contained in one projective line.
		On the other hand, a projective line $\mathbb{P}(\Lambda_{-1} / \Pi \Lambda_{-1})(\overbar{k}) \subseteq  \NEnaive(\overbar{k})$ contains $q+1$ points corresponding to unimodular lattices in $C$.
		By Lemma \ref{RP_IS} \eqref{RP_IS1}, these are exactly the $k$-rational points of $\mathbb{P}(\Lambda_{-1} / \Pi \Lambda_{-1})$.
		\item \label{RP_rmklatt3} If we restrict the union at the right hand side of \eqref{RP_ISunion} to hyperbolic $\Pi^{-1}$-modular lattices $\Lambda_{-1} \subseteq C$ (\emph{i.e.}, $\Nm(\Lambda_{-1}) = 2 O_F$, see Lemma \ref{LA_hyp}), we obtain a canonical subset $\NE(\overbar{k}) \subseteq \NEnaive(\overbar{k})$ and there is a description of $\NE$ as a pro-representable functor on $\Nilp$ (see below).
		We will see later (Theorem \ref{RP_thm}) that $\NE$ is isomorphic to the Drinfeld moduli space $\MDr$, described in \cite[I.3]{BC91}.
		In particular, the underlying topological space of $\NE$ is connected.
		(The induced topology on the projective lines is the Zariski topology, see Proposition \ref{RP_NEnaiveRL}.)
		Moreover, each projective line in $\NE(\overbar{k})$ has $q+1$ intersection points and there are $2$ projective lines intersecting in each such point (see also Proposition \ref{LA_lattRP}). \\
		We fix such an intersection point $P \in \NE(\overbar{k})$.
		Now going back to $\NEnaive(\overbar{k})$, there are $q-1$ additional lines going through $P \in \NEnaive(\overbar{k})$ that correspond to non-hyperbolic lattices in $C$ (see Proposition \ref{LA_lattRP}).
		Each of these additional lines contains $P$ as its only ``hyperbolic'' intersection point, all other intersection points on this line and the line itself correspond to unimodular resp.\ $\Pi^{-1}$-modular lattices $\Lambda \subseteq C$ of norm $\Nm (\Lambda) = (2 / \pi_0)O_F$ (whereas all hyperbolic lattices occuring have the norm ideal $2 O_F$, see Lemma \ref{LA_hyp}).
		Assume $\mathbb{P}(\Lambda / \Pi \Lambda)(\overbar{k}) \subseteq \NEnaive(\overbar{k})$ is such a line and let $P' \in \mathbb{P}(\Lambda / \Pi \Lambda) (\overbar{k})$ be an intersection point, where $P \neq P'$.
		There are again $q$ more lines going through $P'$ (always $q+1$ in total) that correspond to lattices with norm ideal $\Nm (\Lambda) = (2 / \pi_0^2)O_F$, and these lines again have more intersection points and so on.
		This goes on until we reach lines $\mathbb{P}(\Lambda' / \Pi \Lambda')(\overbar{k})$ with $\Nm (\Lambda') = O_F$.
		Each of these lines contains $q$ points that correspond to unimodular lattices $\Lambda_0 \subseteq C$ with $\Nm(\Lambda_0) = O_F$.
		Such a lattice is only contained in one $\Pi^{-1}$-modular lattice (see part \ref{LA_lattRPmin} of Proposition \ref{LA_lattRP}).
		Hence, these points are only contained in one projective line, namely $\mathbb{P}(\Lambda' / \Pi \Lambda')(\overbar{k})$. \\	
		In other words, each intersection point $P \in \NE(\overbar{k})$ has a ``tail'', consisting of finitely many projective lines, which is the connected component of $P$ in $(\NEnaive (\overbar{k}) \setminus \NE(\overbar{k})) \cup \{ P \}$.
		Figure \ref{RP_NEnaivejpg} shows a drawing of $(\NEnaive)_{\red}$ for the cases $F = \mathbb{Q}_2$ (on the left hand side) and $F|\mathbb{Q}_2$ a ramified quadratic extension (on the right hand side).
		The ``tails'' are indicated by dashed lines.
	\end{enumerate}
\end{rmk}

\begin{figure}[hbt]
	\centering
	\begin{subfigure}[b]{0.4\textwidth}
		\includegraphics[width = \textwidth]{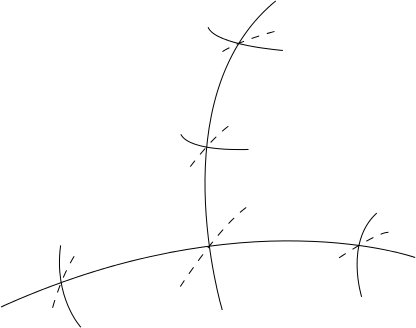}
		\caption{$e = 1$, $f = 1$.}
	\end{subfigure}
	\hspace{1cm}
	\begin{subfigure}[b]{0.4\textwidth}
		\includegraphics[width = \textwidth]{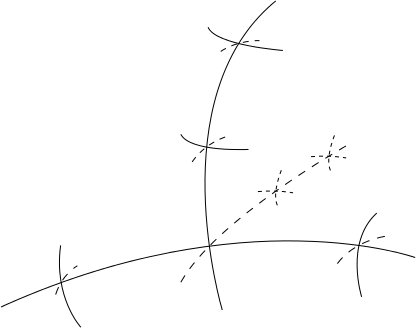}
		\caption{$e= 2$, $f = 1$.}
	\end{subfigure}
	\caption{
		The reduced locus of $\NEnaive$ for $E|F$ of type (R-P) where $F|\mathbb{Q}_2$ has ramification index $e$ and inertia degree $f$.
		Solid lines are given by subschemes $\NEL$ for hyperbolic lattices $\Lambda$.
	}
	\label{RP_NEnaivejpg}
\end{figure}

Fix a $\Pi^{-1}$-modular lattice $\Lambda = \Lambda_{-1} \subseteq C$.
Let $X_{\Lambda}^{+}$ be the formal $O_F$-module over $\Spec \overbar{k}$ associated to the Dieudonn\'e lattice $M = \Lambda \otimes \breve{O}_E \subseteq N$.
It comes with a canonical quasi-isogeny
\begin{equation*}
	\varrho_{\Lambda}^{+} : \mathbb{X} \to X_{\Lambda}^{+}
\end{equation*}
of $F$-height $1$.
We define a subfunctor $\NEL \subseteq \NEnaive$ by mapping $S \in \Nilp$ to
\begin{equation} \label{RP_NEL}
	\NEL(S) = \{ (X,\iota,\lambda,\varrho) \in \NEnaive(S) \mid (\varrho_{\Lambda}^{+} \times S) \circ \varrho \text{ is an isogeny} \}.
\end{equation}
Note that the condition of \eqref{RP_NEL} is closed, \emph{cf.}\ \cite[Prop.\ 2.9]{RZ96}.
Hence $\NEL$ is representable by a closed formal subscheme of $\NEnaive$.
On geometric points, we have a bijection
\begin{equation} \label{RP_NEL=P1geom}
	\NEL(\overbar{k}) \isoarrow \mathbb{P}(\Lambda / \Pi \Lambda)(\overbar{k}),
\end{equation}
as a consequence of Lemma \ref{RP_IS} \eqref{RP_IS1}.

\begin{prop} \label{RP_NEnaiveRL}
	The reduced locus of $\NEnaive$ is given by
	\begin{equation*}
		(\NEnaive)_{\red} = \smashoperator[r]{\bigcup_{\Lambda \subseteq C}} \NEL,
	\end{equation*}
	where $\Lambda$ runs over all $\Pi^{-1}$-modular lattices in $C$.
	For each $\Lambda$, there is an isomorphism of reduced schemes
	\begin{equation*}
		\NEL \isoarrow \mathbb{P}(\Lambda / \Pi \Lambda),
	\end{equation*}
	inducing the map \eqref{RP_NEL=P1geom} on $\overbar{k}$-valued points.
\end{prop}

\begin{proof}
	The embedding
	\begin{equation} \label{RP_NEnaivered}
		\smashoperator[r]{\bigcup_{\Lambda \subseteq C}} (\NEL)_{\red} \inj (\NEnaive)_{\red}
	\end{equation}
	is closed, because each embedding $\NEL \subseteq \NEnaive$ is closed and, locally on $(\NEnaive)_{\red}$, the left hand side is always only a finite union of $(\NEL)_{\red}$.
	It follows already that \eqref{RP_NEnaivered} is an isomorphism, since it is a bijection on $\overbar{k}$-valued points (see the equations \eqref{RP_ISunion} and \eqref{RP_NEL=P1geom}) and $(\NEnaive)_{\red}$ is reduced by definition and locally of finite type over $\Spec \overbar{k}$ by Remark \ref{RP_loftnaive} \eqref{RP_loftnaive2}.
	
	For the second part of the proposition, we follow the proof presented in \cite[4.2]{KR14}.
	Fix a $\Pi^{-1}$-modular lattice $\Lambda \subseteq C$ and let $M = \Lambda \otimes \breve{O}_E \subseteq N$, as above.
	Now $X_{\Lambda}^{+}$ is the formal $O_F$-module associated to $M$, but we also get a formal $O_F$-module $X_{\Lambda}^{-}$ associated to the dual $M^{\sharp} = \Pi M$ of $M$.
	This comes with a natural isogeny
	\begin{equation*}
		\nat_{\Lambda}: X_{\Lambda}^{-} \to X_{\Lambda}^{+}
	\end{equation*}
	and a quasi-isogeny $\varrho_{\Lambda}^{-} : X_{\Lambda}^{-} \to \mathbb{X}$ of $F$-height $1$.
	For $(X,\iota,\lambda,\varrho) \in \NEnaive(S)$ where $S \in \Nilp$, we consider the composition 
	\begin{equation*}
		\varrho_{\Lambda, X}^{-} = \varrho^{-1} \circ (\varrho_{\Lambda}^{-} \times S): (X_{\Lambda}^{-} \times S) \to X.
	\end{equation*}
	By \cite[Lemma 4.2]{KR14}, this composition is an isogeny if and only if $(\varrho_{\Lambda}^{+} \times S) \circ \varrho$ is an isogeny, or, in other words, if and only if $(X,\iota,\lambda,\varrho) \in \NEL(S)$.
	Let $\mathbb{D}_{X_{\Lambda}^{-}}(S)$ be the (relative) Grothendieck-Messing crystal of $X_{\Lambda}^{-}$ evaluated at $S$ (\emph{cf.}\ \cite[Def.\ 3.24]{ACZ} or \cite[5.2]{Ahs11}).
	This is a locally free $\mathcal{O}_S$-module of rank $4$, isomorphic to $\Lambda / \pi_0 \Lambda \otimes_{O_F} \mathcal{O}_S$.
	The kernel of $\mathbb{D}(\nat_{\Lambda})(S)$ is given by $(\Lambda / \Pi \Lambda) \otimes_{O_F} \mathcal{O}_S$, locally a direct summand of rank $2$ of $\mathbb{D}_{X_{\Lambda}^{-}}(S)$.
	For any $(X,\iota,\lambda,\varrho) \in \NEL(S)$, the kernel of $\varrho_{\Lambda, X}^{-}$ is contained in $\ker (\nat_{\Lambda})$.
	It follows from \cite[Cor.\ 4.7]{VW11} (see also \cite[Prop.\ 4.6]{KR14}) that $\ker \mathbb{D}(\varrho_{\Lambda, X}^{-})(S)$ is locally a direct summand of rank $1$ of $(\Lambda / \Pi \Lambda) \otimes_{O_F} \mathcal{O}_S$.
	This induces a map
	\begin{equation*} 
		\NEL(S) \to \mathbb{P}(\Lambda / \Pi \Lambda)(S),
	\end{equation*}
	functorial in $S$, and the arguments of \cite[4.7]{VW11} show that it is an isomorphism.
	(One easily checks that their results indeed carry over to the relative setting over $O_F$.)
\end{proof}

\subsection{Construction of the closed formal subscheme $\NE \subseteq \NEnaive$}
\label{RP2}

We now use a result from section \ref{POL}.
By Theorem \ref{POL_thm} and Remark \ref{POL_rmk} \eqref{POL_rmk2}, there exists a principal polarization $\widetilde{\lambda}_{\mathbb{X}}: \mathbb{X} \to \mathbb{X}^{\vee}$ on $(\mathbb{X},\iota_{\mathbb{X}},\lambda_{\mathbb{X}})$, unique up to a scalar in $O_E^{\times}$, such that the induced Rosati involution is the identity on $O_E$.
Furthermore, for any $(X,\iota,\lambda,\varrho) \in \NEnaive(S)$, the pullback $\widetilde{\lambda} = \varrho^{\ast}(\widetilde{\lambda}_{\mathbb{X}})$ is a principal polarization on $X$.

The next proposition is crucial for the construction of $\NE$.
Recall the notion of a \emph{hyperbolic} lattice from Proposition \ref{LA_latt} and the subsequent discussion.

\begin{prop} \label{RP_strprop}
	It is possible to choose $(\mathbb{X},\iota_{\mathbb{X}},\lambda_{\mathbb{X}})$ and $\widetilde{\lambda}_{\mathbb{X}}$ such that
	\begin{equation*}
		\lambda_{\mathbb{X},1} = \frac{1}{2} (\lambda_{\mathbb{X}} + \widetilde{\lambda}_{\mathbb{X}}) \in \Hom(\mathbb{X},\mathbb{X}^{\vee}).
	\end{equation*}
	Fix such a choice and let $(X,\iota,\lambda,\varrho) \in \NEnaive(\overbar{k})$.
	Then, $\frac{1}{2} (\lambda + \widetilde{\lambda}) \in \Hom(X,X^{\vee})$ if and only if $(X,\iota,\lambda,\varrho) \in \NEL(\overbar{k})$ for some hyperbolic lattice $\Lambda \subseteq C$.
\end{prop}

\begin{proof}
	The polarization $\widetilde{\lambda}_{\mathbb{X}}$ on $\mathbb{X}$ induces an alternating form $(\,,)$ on the rational Dieudonn\'e module $N = M_{\mathbb{X}} \otimes_{\breve{O}_F} \breve{F}$.
	For all $x,y \in N$, the form $(\,,)$ satisfies the equations
	\begin{align*}
		({\bf F}x, y) &= (x, {\bf V}y)^{\sigma}, \\
		(\Pi x, y) &= (x, \Pi y).
	\end{align*}
	It induces an $\breve{E}$-alternating form $b$ on $N$ via
	\begin{equation*}
		b(x,y) = \delta((\Pi x, y) + \Pi (x,y)),
	\end{equation*}
	where $\delta \in \breve{O}_F$ is a unit generating the unramified quadratic extension of $F$, chosen such that $\delta^{\sigma} = -\delta$ and $\frac{1+\delta}{2} \in \breve{O}_F$, see page \pageref{LA_quadext}.
	On the other hand, we can describe $(\,,)$ in terms of $b$,
	\begin{equation} \label{RP_altalt}
		(x,y) = \Tr_{\breve{E}|\breve{F}} \left( \frac{1}{2\Pi\delta} \cdot b(x,y) \right).
	\end{equation}
	The form $b$ is invariant under $\tau = \Pi {\bf V}^{-1}$, since
	\begin{equation*}
		b(\tau x, \tau y) = b(-{\bf F} \Pi^{-1} x, \Pi {\bf V}^{-1} y) = b(\Pi^{-1} x, \Pi y)^{\sigma} = b(x,y)^{\sigma}.
	\end{equation*}
	Hence $b$ defines an $E$-linear alternating form on $C = N^{\tau}$, which we again denote by $b$.
	Denote by $\langle \,, \rangle$ the alternating form on $M_{\mathbb{X}}$ induced by the polarization $\lambda_{\mathbb{X}}$ and let $h$ be the corresponding hermitian form, see \eqref{RP_herm}.
	On $N_{\mathbb{X}}$, we define the alternating form $\langle \,, \rangle_1$ by
	\begin{equation*}
		\langle x, y \rangle_1 = \frac{1}{2} (\langle x, y \rangle + (x,y)).
	\end{equation*}
	This form is integral on $M_{\mathbb{X}}$ if and only if $\lambda_{\mathbb{X},1} = \frac{1}{2} (\lambda_{\mathbb{X}} + \widetilde{\lambda}_{\mathbb{X}})$ is a polarization on $\mathbb{X}$.
	
	We choose $(\mathbb{X},\iota_{\mathbb{X}},\lambda_{\mathbb{X}})$ such that it corresponds to a unimodular hyperbolic lattice $\Lambda_0 \subseteq (C,h)$ under the identifications of \eqref{RP_geompts} and Lemma \ref{RP_latt}.
	There exists a basis $(e_1,e_2)$ of $\Lambda_0$ such that
	\begin{equation} \label{RP_matrices}
		h \, \weq \begin{pmatrix}
			& 1 \\
			1 & \\
		\end{pmatrix}, \quad b \, \weq \begin{pmatrix}
			& u \\
			-u & \\
		\end{pmatrix},
	\end{equation}
	for some $u \in E^{\times}$.
	Since $\widetilde{\lambda}_{\mathbb{X}}$ is principal, the alternating form $b$ is perfect on $\Lambda_0$, thus $u \in O_E^{\times}$.
	After rescaling $\widetilde{\lambda}_{\mathbb{X}}$, we may assume that $u = 1$.
	We now have
	\begin{equation*}
		\frac{1}{2} (h(x,y) + b(x,y)) \in O_E,
	\end{equation*}
	for all $x,y \in \Lambda_0$.
	Thus $\frac{1}{2} (h + b)$ is integral on $M_{\mathbb{X}} = \Lambda_0 \otimes_{O_E} \breve{O}_E$.
	This implies that
	\begin{align*}
		\langle x, y \rangle_1 & = \frac{1}{2} (\langle x, y \rangle + (x,y)) = \frac{1}{2} \Tr_{\breve{E}|\breve{F}} \left(\frac{1}{2 \Pi} \cdot h(x.y) + \frac{1}{2\Pi \delta} \cdot b(x,y) \right) \\
		& = \Tr_{\breve{E}|\breve{F}} \left(\frac{1}{4 \Pi} (h(x,y) + b(x,y)) \right) + \Tr_{\breve{E}|\breve{F}} \left(\frac{1-\delta}{4\Pi \delta} \cdot b(x,y) \right) \in \breve{O}_F,
	\end{align*}
	for all $x,y \in M_{\mathbb{X}}$.
	Indeed, in the definition of $b$, the unit $\delta$ has been chosen such that $\frac{1+\delta}{2} \in \breve{O}_F$, so the second summand is in $\breve{O}_F$.
	The first summand is integral, since $\frac{1}{2} (h + b)$ is integral.
	It follows that $\lambda_{\mathbb{X},1} = \frac{1}{2} (\lambda_{\mathbb{X}} + \widetilde{\lambda}_{\mathbb{X}})$ is a polarization on $\mathbb{X}$.
	
	Let $(X,\iota,\lambda,\varrho) \in \NEnaive(\overbar{k})$ and assume that $\lambda_1 = \frac{1}{2} (\lambda + \widetilde{\lambda}) = \varrho^{\ast}(\lambda_{\mathbb{X},1})$ is a polarization on $X$.
	Then $\langle \,, \rangle_1$ is integral on the Dieudonn\'e module $M \subseteq N$ of $X$.
	By the above calculation, this is equivalent to $\frac{1}{2}(h + b)$ being integral on $M$.
	In particular, this implies that
	\begin{equation*}
		h(x,x) = h(x,x) + b(x,x) \in 2 \breve{O}_F,
	\end{equation*}
	for all $x \in M$.
	Let $\Lambda = (M +\tau(M))^{\tau}$.
	Then $h(x,x) \in 2O_F$ for all $x \in \Lambda$, hence $\Nm (\Lambda) \subseteq 2O_F$.
	By Lemma \ref{LA_hyp} and the bound of norm ideals, we have $\Nm (\Lambda) = 2O_F$ and $\Lambda$ is a hyperbolic lattice.
	It follows that $(X,\iota,\lambda,\varrho) \in \NELpr(\overbar{k})$ for some hyperbolic $\Pi^{-1}$-modular lattice $\Lambda' \subseteq C$.
	Indeed, if $M^{\tau} \subsetneq \Lambda$ then $\Lambda$ is $\Pi^{-1}$-modular and $\Lambda' = \Lambda$. If $M^{\tau} = \Lambda$ then it is contained in some $\Pi^{-1}$-modular hyperbolic lattice $\Lambda'$ by Proposition \ref{LA_lattRP}.
	
	Conversely, assume that $(X,\iota,\lambda,\varrho) \in \NEL(\overbar{k})$ for some hyperbolic lattice $\Lambda \subseteq C$.
	It suffices to show that $\frac{1}{2} (h + b)$ is integral on $\Lambda$.
	Indeed, it follows that $\frac{1}{2} (h + b)$ is integral on the Dieudonn\'e module $M$.
	Thus $\langle \,, \rangle_1$ is integral on $M$ and this is equivalent to $\lambda_1 = \frac{1}{2} (\lambda + \widetilde{\lambda}) \in \Hom(X,X^{\vee})$.
	
	Let $\Lambda' \subseteq C$ be the $\Pi^{-1}$-modular lattice generated by $e_1$ and $\Pi^{-1} e_2$, where $(e_1,e_2)$ is the basis of the lattice $\Lambda_0$ corresponding to the framing object $(\mathbb{X},\iota_{\mathbb{X}},\lambda_{\mathbb{X}})$.
	By \eqref{RP_matrices}, $h$ and $b$ have the following form with respect to the basis $(e_1, \Pi^{-1} e_2)$,
	\begin{equation*}
		h \, \weq \begin{pmatrix}
			& -\Pi^{-1} \\
			\Pi^{-1} & \\
		\end{pmatrix}, \quad b \, \weq \begin{pmatrix}
			& \Pi^{-1} \\
			-\Pi^{-1} & \\
		\end{pmatrix}.
	\end{equation*}
	In particular, $\Lambda'$ is hyperbolic and $\frac{1}{2} (h + b)$ is integral on $\Lambda'$.
	By Proposition \ref{LA_latt}, there exists an automorphism $g \in \SU(C,h)$ mapping $\Lambda$ onto $\Lambda'$.
	Since $\det g = 1$, the alternating form $b$ is invariant under $g$.
	It follows that $\frac{1}{2} (h + b)$ is also integral on $\Lambda$.
\end{proof}

From now on, we assume $(\mathbb{X},\iota_{\mathbb{X}},\lambda_{\mathbb{X}})$ and $\widetilde{\lambda}_{\mathbb{X}}$ chosen in a way such that
\begin{equation*}
	\lambda_{\mathbb{X},1} = \frac{1}{2} (\lambda_{\mathbb{X}} + \widetilde{\lambda}_{\mathbb{X}}) \in \Hom(\mathbb{X},\mathbb{X}^{\vee}).
\end{equation*}
Note that this determines the polarization $\widetilde{\lambda}_{\mathbb{X}}$ up to a scalar in $1 + 2O_E$.
If we replace $\widetilde{\lambda}_{\mathbb{X}}$ by $\widetilde{\lambda}_{\mathbb{X}}' = \widetilde{\lambda}_{\mathbb{X}} \circ \iota_{\mathbb{X}}(1 +  2u)$ for some $u \in O_E$, then $\lambda_{\mathbb{X},1}' = \lambda_{\mathbb{X},1} + \widetilde{\lambda}_{\mathbb{X}} \circ \iota_{\mathbb{X}}(u)$.

We can now formulate the straightening condition.

\begin{defn} \label{RP_strdef}
	Let $S \in \Nilp$.
	An object $(X,\iota,\lambda,\varrho) \in \NEnaive(S)$ satisfies the \emph{straightening condition} if
	\begin{equation} \label{RP_strcond}
		\lambda_1 \in \Hom(X,X^{\vee}),
	\end{equation}
	where $\lambda_1 = \frac{1}{2} (\lambda + \widetilde{\lambda}) = \varrho^{\ast}(\lambda_{\mathbb{X},1})$.
\end{defn}

This definition is clearly independent of the choice of the polarization $\widetilde{\lambda}_{\mathbb{X}}$.
We define $\NE$ as the functor that maps $S \in \Nilp$ to the set of all tuples $(X,\iota,\lambda,\varrho) \in \NEnaive(S)$ that satisfy the straightening condition.
By \cite[Prop.\ 2.9]{RZ96}, $\NE$ is representable by a closed formal subscheme of $\NEnaive$.

\begin{rmk} \label{RP_rmkNE}
	The reduced locus of $\NE$ can be written as
	\begin{equation*}
		(\NE)_{\red} = \smashoperator[r]{\bigcup_{\Lambda \subseteq C}} \NEL \simeq \smashoperator[r]{\bigcup_{\Lambda \subseteq C}} \mathbb{P}(\Lambda / \Pi \Lambda),
	\end{equation*}
	where we take the unions over all \emph{hyperbolic} $\Pi^{-1}$-modular lattices $\Lambda \subseteq C$.
	By Proposition \ref{LA_lattRP} and Lemma \ref{RP_IS}, each projective line contains $q+1$ points corresponding to unimodular lattices and there are two lines intersecting in each such point.
	Recall from Remark \ref{RP_rmklatt} \eqref{RP_rmklatt1} that there exist non-hyperbolic $\Pi^{-1}$-modular lattices $\Lambda \subseteq C$, thus we have $\NE(\overbar{k}) \neq \NEnaive(\overbar{k})$, and in particular $(\NE)_{\red} \neq (\NEnaive)_{\red}$.
\end{rmk}

\begin{rmk} \label{RP_genfiber}
	As has been pointed out to the author by A.\ Genestier, the straightening condition is not trivial on the rigid-analytic generic fiber of $\NEnaive$.
	However, we can show that it is open and closed.
	Since a proper study of the generic fiber would go beyond the scope of this paper, we restrain ourselves to indications rather than complete proofs.
	
	Let $C$ be an algebraically closed extension of $F$ and $\mathcal{O}_C$ its ring of integers. 
	Take a point $x = (X,\iota,\lambda,\varrho) \in \NEnaive(\mathcal{O}_C)$ and consider its $2$-adic Tate module $T_2(x)$.
	It is a free $O_E$-module of rank $2$ and $\lambda$ endows $T_2(x)$ with a perfect (non-split) hermitian form $h$.
	If $x \in \NE(\mathcal{O}_C)$, then the straightening condition implies that $(T_2(x),h)$ is a lattice with minimal norm\footnote{Calling this lattice ``hyperbolic'' doesn't make much sense here since it is anisotropic.}
	$\Nm(T_2(x))$ in the vector space $V_2(x) = T_2(x) \otimes_{O_E} E$ (see Proposition \ref{LA_latt} and \cite{Jac62}).
	But $V_2(x)$ also contains selfdual lattices with non-minimal norm ideal.
	Let $\Lambda \subseteq V_2(x)$ be such a lattice with $\Nm(\Lambda) \neq \Nm(T_2(x))$.
	Let $\Lambda'$ be the intersection of $T_2(x)$ and $\Lambda$ in $V_2(x)$.
	The inclusions $\Lambda' \inj \Lambda$ and $\Lambda' \inj T_2(x)$ define canonically a formal $O_F$-module $Y$ with $T_2(Y) = \Lambda'$ and a quasi-isogeny $\varphi: X \to Y$.
	By inheriting all data, $Y$ becomes a point in $\NEnaive(\mathcal{O}_C)$ that does not satisfy the straightening condition.
	
	To see that the straightening condition is open and closed on the generic fiber, consider the universal formal $O_F$-module $\mathcal{X} = (\mathcal{X},\iota_{\mathcal{X}},\lambda_{\mathcal{X}})$ over $\NEnaive$ and let $T_2(\mathcal{X})$ be its Tate module.
	Then $T_2(\mathcal{X})$ is a locally constant sheaf over $\NEnaiverig$ with respect to the \'etale topology.
	The polarization $\lambda_{\mathcal{X}}$ defines a hermitian form $h$ on $T_2(\mathcal{X})$.
	Since $T_2(\mathcal{X})$ is a locally constant sheaf, the norm ideal $\Nm(T_2(\mathcal{X}))$ with respect to $h$ (see Proposition \ref{LA_latt}) is locally constant as well.
	Hence the locus where $\Nm(T_2(\mathcal{X}))$ is minimal is open and closed in $\NEnaiverig$.
	But this is exactly $\NErig \subseteq \NEnaiverig$.
\end{rmk}

\subsection{The isomorphism to the Drinfeld moduli problem}
\label{RP3}

We now recall the Drinfeld moduli problem $\MDr$ on $\Nilp$.
Let $B$ be the quaternion division algebra over $F$ and $O_B$ its ring of integers.
Let $S \in \Nilp$.
Then $\MDr(S)$ is the set of equivalence classes of objects $(X,\iota_B,\varrho)$ where
\begin{itemize}
	\item $X$ is a formal $O_F$-module over $S$ of dimension $2$ and height $4$,
	\item $\iota_B: O_B \to \End(X)$ is an action of $O_B$ on $X$ satisfying the \emph{special} condition, \emph{i.e.}, $\Lie X$ is, locally on $S$, a free $\mathcal{O}_S \otimes_{O_F} O_F^{(2)}$-module of rank 1, where $O_F^{(2)} \subseteq O_B$ is any embedding of the unramified quadratic extension of $O_F$ into $O_B$ (\emph{cf.}\ \cite{BC91}),
	\item $\varrho: X \times_S \overbar{S} \to \mathbb{X} \times_{\Spec \overbar{k}} \overbar{S}$ is an $O_B$-linear quasi-isogeny of height $0$ to a fixed framing object $(\mathbb{X},\iota_{\mathbb{X}}) \in \MDr(\overbar{k})$.
\end{itemize}

Such a framing object exists and is unique up to isogeny.
By a proposition of Drinfeld, \emph{cf.}\ \cite[p.\ 138]{BC91}, there always exist polarizations on these objects, as follows:

\begin{prop}[Drinfeld] \label{RP_Dr}
	Let $\Pi \in O_B$ a uniformizer with $\Pi^2 \in O_F$ and let $b \mapsto b'$ be the standard involution of $B$.
	Then $b \mapsto b^{\ast} = \Pi b' \Pi^{-1}$ is another involution on $B$.
	\begin{enumerate}
		\item \label{RP_Dr1} There exists a principal polarization $\lambda_{\mathbb{X}}: \mathbb{X} \to \mathbb{X}^{\vee}$ on $\mathbb{X}$ with associated Rosati involution $b \mapsto b^{\ast}$.
		It is unique up to a scalar in $O_F^{\times}$.
		\item \label{RP_Dr2} Let $\lambda_{\mathbb{X}}$ as in \eqref{RP_Dr1}.
		For $(X,\iota_B,\varrho) \in \MDr(S)$, there exists a unique principal polarization
		\begin{equation*}
			\lambda: X \to X^{\vee}
		\end{equation*}
		with Rosati involution $b \mapsto b^{\ast}$ such that $\varrho^{\ast}(\lambda_{\mathbb{X}}) = \lambda$ on $\overbar{S}$.
	\end{enumerate}
\end{prop}

We now relate $\MDr$ and $\NE$.
For this, we fix an embedding $E \inj B$.
Any choice of a uniformizer $\Pi \in O_E$ with $\Pi^2 \in O_F$ induces the same involution $b \mapsto b^{\ast} = \Pi b' \Pi^{-1}$ on $B$.

For the framing object $(\mathbb{X},\iota_{\mathbb{X}})$ of $\MDr$, let $\lambda_{\mathbb{X}}$ be a polarization associated to this involution by Proposition \ref{RP_Dr} \eqref{RP_Dr1}.
Denote by $\iota_{\mathbb{X},E}$ the restriction of $\iota_{\mathbb{X}}$ to $O_E \subseteq O_B$.
For any object $(X,\iota_B,\varrho) \in \MDr(S)$, let $\lambda$ be the polarization with Rosati involution $b \mapsto b^{\ast}$ that satisfies $\varrho^{\ast}(\lambda_{\mathbb{X}}) = \lambda$, see Proposition \ref{RP_Dr} \eqref{RP_Dr2}.
Let $\iota_E$ be the restriction of $\iota_B$ to $O_E$.

\begin{lem} \label{RP_clemb}
	$(\mathbb{X},\iota_{\mathbb{X},E},\lambda_{\mathbb{X}})$ is a framing object for $\NEnaive$.
	Furthermore, the map 
	\begin{equation*}
		(X,\iota_B,\varrho) \mapsto (X,\iota_E,\lambda,\varrho)
	\end{equation*}
	induces a closed immersion of formal schemes
	\begin{equation*}
		\eta: \MDr \inj \NEnaive.
	\end{equation*}
\end{lem}

\begin{proof}
	There are two things to check: that $\QIsog(\mathbb{X},\iota_{\mathbb{X}},\lambda_{\mathbb{X}})$ contains $\SU(C,h)$ as a closed subgroup
	and that $\iota_E$ satisfies the Kottwitz condition.
	Indeed, once these two assertions hold, we can take $(\mathbb{X},\iota_{\mathbb{X},E},\lambda_{\mathbb{X}})$ as a framing object for $\NEnaive$ and the morphism $\eta$ is well-defined.
	For any $S \in \Nilp$, the map $\eta(S)$ is injective, because $(X,\iota_B,\varrho)$ and $(X',\iota_B',\varrho') \in \MDr(S)$ map to the same point in $\NEnaive(S)$ under $\eta$ if and only if the quasi-isogeny $\varrho'  \circ \varrho$ on $\overbar{S}$ lifts to an isomorphism on $S$, \emph{i.e.}, if and only if $(X,\iota_B,\varrho)$ and $(X',\iota_B',\varrho')$ define the same point in $\MDr(S)$.
	The functor
	\begin{equation*}
		F: S \mapsto \{ (X,\iota,\lambda,\varrho) \in \NEnaive(S) \mid \iota \text{ extends to an } O_B \text{-action} \}
	\end{equation*}
	is pro-representable by a closed formal subscheme of $\NEnaive$ by \cite[Prop.\ 2.9]{RZ96}.
	Now, the formal subscheme $\eta(\MDr) \subseteq F$ is given by the special condition.
	But the special condition is open and closed (see \cite[p.\ 7]{RZ14}), thus $\eta$ is a closed embedding.
	
	It remains to show the two assertions from the beginning of this proof.
	We first check the condition on $\QIsog(\mathbb{X},\iota_{\mathbb{X}},\lambda_{\mathbb{X}})$.
	Let $G_{(\mathbb{X},\iota_{\mathbb{X}})}$ be the group of $O_B$-linear quasi-isogenies $\varphi: (\mathbb{X},\iota_{\mathbb{X}}) \to (\mathbb{X},\iota_{\mathbb{X}})$ of height $0$ such that the induced homomorphism of Dieudonn\'e modules has determinant $1$.
	Then we have (non-canonical) isomorphisms $G_{(\mathbb{X},\iota_{\mathbb{X}})}  \simeq \SL_{2,F}$ and $\SL_{2,F} \simeq \SU(C,h)$, since $h$ is split.
	The uniqueness of the polarization $\lambda_{\mathbb{X}}$ (up to a scalar in $O_F^{\times}$) implies that $G_{(\mathbb{X},\iota_{\mathbb{X}})} \subseteq \QIsog(\mathbb{X},\iota_{\mathbb{X}},\lambda_{\mathbb{X}})$.
	This is a closed embedding of linear algebraic groups over $F$, since a quasi-isogeny $\varphi \in \QIsog(\mathbb{X},\iota_{\mathbb{X}},\lambda_{\mathbb{X}})$ lies in $G_{(\mathbb{X},\iota_{\mathbb{X}})}$ if and only if it is $O_B$-linear and has determinant $1$, and these are closed conditions on $\QIsog(\mathbb{X},\iota_{\mathbb{X}},\lambda_{\mathbb{X}})$.
	
	Finally, the special condition implies the Kottwitz condition for any element $b \in O_B$ (see \cite[Prop.\ 5.8]{RZ14}), \emph{i.e.}, the characteristic polynomial for the action of $\iota(b)$ on $\Lie X$ is
	\begin{equation*}
		\charp(\Lie X, T \mid \iota(b)) = (T - b)(T - b'),
	\end{equation*}
	where the right hand side is a polynomial in $\mathcal{O}_S[T]$ via the structure homomorphism $O_F \inj \breve{O}_F \to \mathcal{O}_S$.
	From this, the second assertion follows.
\end{proof}

Let $O_F^{(2)} \subseteq O_B$ be an embedding such that conjugation with $\Pi$ induces the non-trivial Galois action on $O_F^{(2)}$, as in Lemma \ref{LA_quat} \eqref{LA_quatRP}.
Fix a generator $\gamma = \frac{1+\delta}{2}$ of $O_F^{(2)}$ with $\delta^2 \in O_F^{\times}$.
On $(\mathbb{X},\iota_{\mathbb{X}})$, the principal polarization $\widetilde{\lambda}_{\mathbb{X}}$ given by
\begin{equation*}
	\widetilde{\lambda}_{\mathbb{X}} = \lambda_{\mathbb{X}} \circ \iota_{\mathbb{X}}(\delta)
\end{equation*}
has a Rosati involution that induces the identity on $O_E$.
For any $(X,\iota_B,\varrho) \in \MDr(S)$, we set $\widetilde{\lambda} = \varrho^{\ast}(\widetilde{\lambda}_{\mathbb{X}}) = \lambda \circ \iota_B(\delta)$.
The tuple $(X,\iota_E,\lambda,\varrho) = \eta(X,\iota_B,\varrho)$ satisfies the straightening condition \eqref{RP_strcond}, since
\begin{equation*}
	\lambda_1 = \frac{1}{2} (\lambda + \widetilde{\lambda}) = \lambda \circ \iota_B(\gamma) \in \Hom(X,X^{\vee}).
\end{equation*}
In particular, the tuple $(\mathbb{X},\iota_{\mathbb{X},E},\lambda_{\mathbb{X}})$ is a framing object of $\NE$ and $\eta$ induces a natural transformation
\begin{equation} \label{RP_MDrtoNE}
	\eta: \MDr \inj \NE.
\end{equation}
Note that this map does not depend on the above choices, as $\NE$ is a closed formal subscheme of $\NEnaive$.

\begin{thm} \label{RP_thm}
	$\eta: \MDr \to \NE$ is an isomorphism of formal schemes.
\end{thm}

We will first prove this on $\overbar{k}$-valued points:

\begin{lem} \label{RP_thmgeom}
	$\eta$ induces a bijection $\eta(\overbar{k}): \MDr(\overbar{k}) \to \NE(\overbar{k})$.
\end{lem}

\begin{proof}
	We can identify the $\overbar{k}$-valued points of $\MDr$ with a subset $\MDr(\overbar{k}) \subseteq \NEnaive(\overbar{k})$.
	The rational Dieudonn\'e-module $N$ of $\mathbb{X}$ is equipped with an action of $B$.
	Fix an embedding $F^{(2)} \inj B$ as in Lemma \ref{LA_quat} \eqref{LA_quatRP}.
	This induces a $\mathbb{Z}/ 2$-grading $N = N_0 \oplus N_1$ of $N$, where
	\begin{align*}
		N_0 & = \{ x \in N \mid \iota(a)x = ax \text{ for all } a \in F^{(2)} \}, \\
		N_1 & = \{ x \in N \mid \iota(a)x = \sigma(a)x \text{ for all } a \in F^{(2)} \},
	\end{align*}
	for a fixed embedding $F^{(2)} \inj \breve{F}$.
	The operators ${\bf V}$ and ${\bf F}$ have degree $1$ with respect to this decomposition.
	Recall that $\lambda$ has Rosati involution $b \mapsto \Pi b' \Pi^{-1}$ on $O_B$ which restricts to the identity on $O_F^{(2)}$.
	The subspaces $N_0$ and $N_1$ are therefore orthogonal with respect to $\langle \,, \rangle$.
	
	Under the identification \eqref{RP_geompts}, a lattice $M \in \MDr(\overbar{k})$ respects this decomposition, \emph{i.e.}, $M = M_0 \oplus M_1$ with $M_i = M \cap N_i$.
	Furthermore it satisfies the special condition:
	\begin{equation*}
		\dim M_0 / {\bf V}M_1 = \dim M_1 / {\bf V}M_0 = 1.
	\end{equation*}
	We already know that $\MDr(\overbar{k}) \subseteq \NE(\overbar{k})$, so let us assume $M \in \NE(\overbar{k})$.
	We want to show that $M \in \MDr(\overbar{k})$, \emph{i.e.}, that the lattice $M$ is stable under the action of $O_B$ on $N$ and satisfies the special condition.
	It is stable under the $O_B$-action if and only if $M = M_0 \oplus M_1$ for $M_i = M \cap N_i$.
	Let $y \in M$ and $y = y_0 + y_1$ with $y_i \in N_i$.
	For any $x \in M$, we have
	\begin{equation} \label{RP_thmeq1}
		\langle x,y \rangle = \langle x,y_0 \rangle + \langle x, y_1 \rangle \in \breve{O}_F.
	\end{equation}
	We can assume that $\lambda_{\mathbb{X},1} = \lambda_{\mathbb{X}} \circ \iota_B(\gamma)$ with $\gamma \in O_F^{(2)}$ under our fixed embedding $F^{(2)} \inj B$.
	Recall that $\gamma^{\sigma} = 1 - \gamma$ from page \pageref{LA_quadext}.
	Let $\langle \,, \rangle_1$ be the alternating form on $M$ induced by $\lambda_{\mathbb{X},1}$.
	Then,
	\begin{equation} \label{RP_thmeq2}
		\langle x, y \rangle_1 = \gamma \cdot \langle x, y_0 \rangle + (1-\gamma) \cdot \langle x, y_1 \rangle \in \breve{O}_F.
	\end{equation}
	From the equations \eqref{RP_thmeq1} and \eqref{RP_thmeq2}, it follows that $\langle x,y_0 \rangle$ and $\langle x, y_1 \rangle$ lie in $\breve{O}_F$.
	Since $x \in M$ was arbitrary and $M = M^{\vee}$, this gives $y_0, y_1 \in M$.
	Hence $M$ respects the decomposition of $N$ and is stable under the action of $O_B$.
	
	It remains to show that $M$ satisfies the special condition:
	The alternating form $\langle \,, \rangle$ is perfect on $M$, thus the restrictions to $M_0$ and $M_1$ are perfect as well.
	If $M$ is not special, we have $M_i = {\bf V}M_{i+1}$ for some $i \in \{0, 1\}$.
	But then, $\langle \,, \rangle$ cannot be perfect on $M_i$.
	In fact, for any $x,y \in M_{i+1}$,
	\begin{equation*}
		\langle {\bf V}x, {\bf V}y \rangle^{\sigma} = \langle {\bf F}{\bf V} x, y \rangle = \pi_0 \cdot \langle x, y \rangle \in \pi_0 \breve{O}_F.
	\end{equation*} 
	Thus $M$ is indeed special, \emph{i.e.}, $M \in \MDr(\overbar{k})$, and this finishes the proof of the lemma.
\end{proof}

\begin{proof}[Proof of Theorem \ref{RP_thm}]
	We already know that $\eta$ is a closed embedding 
	\begin{equation*}
		\eta: \MDr \inj \NE.
	\end{equation*}
	Let $(\mathbb{X},\iota_{\mathbb{X}})$ be the framing object of $\MDr$ and choose an embedding $O_F^{(2)} \subseteq O_B$ and a generator $\gamma$ of $O_F^{(2)}$ as in Lemma \ref{LA_quat} \eqref{LA_quatRP}.
	We take $(\mathbb{X},\iota_{\mathbb{X},E},\lambda_{\mathbb{X}})$ as a framing object for $\NE$ and set $\widetilde{\lambda}_{\mathbb{X}} = \lambda_{\mathbb{X}} \circ \iota_{\mathbb{X}}(\delta)$.
	
	Let $(X,\iota,\lambda,\varrho) \in \NE(S)$ and $\widetilde{\lambda} = \varrho^{\ast}(\widetilde{\lambda}_{\mathbb{X}})$.
	We have
	\begin{equation*}
		\varrho^{-1} \circ \iota_{\mathbb{X}} (\gamma) \circ \varrho = \varrho^{-1} \circ \lambda_{\mathbb{X}}^{-1} \circ \lambda_{\mathbb{X},1} \circ \varrho = \lambda^{-1} \circ \lambda_1 \in \End(X),
	\end{equation*}
	where $\lambda_{\mathbb{X},1} = \frac{1}{2} (\lambda_{\mathbb{X}} + \widetilde{\lambda}_{\mathbb{X}})$ and $\lambda_1 = \frac{1}{2} (\lambda + \widetilde{\lambda})$.
	Since $O_B = O_F[\Pi, \gamma]$, this induces an $O_B$-action $\iota_B$ on $X$ and makes $\varrho$ an $O_B$-linear quasi-isogeny.
	We have to check that $(X,\iota_B,\varrho)$ satisfies the special condition.
	
	Recall that the special condition is open and closed (see \cite[p.\ 7]{RZ14}), so $\eta$ is an open and closed embedding.
	Furthermore, $\eta(\overbar{k})$ is bijective and the reduced loci $(\MDr)_{\red}$ and $(\NE)_{\red}$ are locally of finite type over $\Spec \overbar{k}$.
	Hence $\eta$ indcues an isomorphism on reduced subschemes.
	But any open and closed embedding of formal schemes, that is an isomorphism on the reduced subschemes, is already an isomorphism.
\end{proof}


\section{The moduli problem in the case (R-U)}
\label{RU}

Let $E|F$ be a quadratic extension of type (R-U), generated by a uniformizer $\Pi$ satisfying an Eisenstein equation of the form $\Pi^2 - t\Pi + \pi_0 = 0$ where $t \in O_F$ and $\pi_0|t|2$.
Let $O_F$ and $O_E$ be the rings of integers of $F$ and $E$. We have $O_E = O_F[\Pi]$.
As in the case (R-P), let $k$ be the common residue field, $\overbar{k}$ an algebraic closure, $\breve{F}$ the completion of the maximal unramified extension with ring of integers $\breve{O}_F = W_{O_F}(\overbar{k})$ and $\sigma$ the lift of the Frobenius in $\Gal(\overbar{k}|k)$ to $\Gal(\breve{O}_F|O_F)$.

\subsection{The naive moduli problem} \label{RU1}

Let $S \in \Nilp$.
Consider tuples $(X,\iota,\lambda)$, where
\begin{itemize}
	\item $X$ is a formal $O_F$-module over $S$ of dimension $2$ and height $4$.
	\item $\iota: O_E \to \End(X)$ is an action of $O_E$ on $X$ satisfying the \emph{Kottwitz condition}:
	The characteristic polynomial of $\iota(\alpha)$ for some $\alpha \in O_E$ is given by
	\begin{equation*}
		\charp(\Lie X, T \mid \iota(\alpha)) = (T - \alpha)(T - \overbar{\alpha}).
	\end{equation*}
	Here $\alpha \mapsto \overbar{\alpha}$ is the Galois conjugation of $E|F$ and the right hand side is a polynomial in $\mathcal{O}_S[T]$ via the structure morphism $O_F \inj \breve{O}_F \to \mathcal{O}_S$.
	\item $\lambda: X \to X^{\vee}$ is a polarization on $X$ with kernel $\ker \lambda = X[\Pi]$, where $X[\Pi]$ is the kernel of $\iota(\Pi)$.
	Further we demand that the Rosati involution of $\lambda$ satisfies $\iota(\alpha)^{\ast} = \iota(\overbar{\alpha})$ for all $\alpha \in O_E$.
\end{itemize}

We define quasi-isogenies $\varphi: (X,\iota,\lambda) \to (X',\iota',\lambda')$ and the group $\QIsog(X,\iota,\lambda)$ as in Definition \ref{RP_isonaive}.
%

\begin{prop} \label{RU_frnaive}
	Up to isogeny, there exists exactly one such tuple $(\mathbb{X},\iota_{\mathbb{X}},\lambda_{\mathbb{X}})$ over $S = \Spec \overbar{k}$ under the condition that the group $\QIsog(\mathbb{X}, \iota_{\mathbb{X}}, \lambda_{\mathbb{X}})$
	contains a closed subgroup isomorphic to $\SU(C,h)$ for a $2$-dimensional $E$-vector space $C$ with split $E|F$-hermitian form $h$.
\end{prop}

\begin{rmk} \label{RU_rmknaive}
	As in the case (R-P), we have $\QIsog(\mathbb{X}, \iota_{\mathbb{X}}, \lambda_{\mathbb{X}}) \cong \U(C,h)$ for $(\mathbb{X},\iota_{\mathbb{X}},\lambda_{\mathbb{X}})$ as in the Proposition.
\end{rmk}

\begin{proof}[Proof of Proposition \ref{RU_frnaive}]
	We first show uniqueness of the object.
	Let $(X,\iota,\lambda) / \Spec \overbar{k}$ be a tuple as in the proposition and consider its rational Dieudonn\'e-module $N_X$.
	This is a $4$-dimensional vector space over $\breve{F}$ equipped with an action of $E$ and an alternating form $\langle \,,\rangle$ such that
	\begin{equation}
		\langle x, \Pi y \rangle = \langle \overbar{\Pi} x, y \rangle
	\end{equation}
	for all $x,y \in N_X$.
	Let $\breve{E} = \breve{F} \otimes_F E$.
	We can see $N_X$ as $2$-dimensional vector space over $\breve{E}$ with a hermitian form $h$ given by
	\begin{equation} \label{RU_herm}
		h(x,y) = \langle \Pi x, y \rangle - \overbar{\Pi} \langle x,y \rangle.
	\end{equation}
	Let ${\bf F}$ and ${\bf V}$ be the $\sigma$-linear Frobenius and the $\sigma^{-1}$-linear Verschiebung on $N_X$.
	We have ${\bf F}{\bf V} = {\bf V}{\bf F} = \pi_0$ and, since $\langle \,, \rangle$ comes from a polarization,
	\begin{equation*}
		\langle {\bf F}x, y \rangle = \langle x, {\bf V}y \rangle^{\sigma}.
	\end{equation*}
	Consider the $\sigma$-linear operator $\tau = \Pi {\bf V}^{-1} = {\bf F} \overbar{\Pi}^{-1}$.
	The hermitian form $h$ is invariant under $\tau$:
	\begin{equation*}
		h(\tau x, \tau y) = h({\bf F} \overbar{\Pi}^{-1} x, \Pi {\bf V}^{-1} y) = h({\bf F} x, {\bf V}^{-1} y) = h(x,y)^{\sigma}.
	\end{equation*}
	From the condition on $\QIsog(\mathbb{X}, \iota_{\mathbb{X}}, \lambda_{\mathbb{X}})$ it follows that $N_X$ is isotypical of slope $\frac{1}{2}$ and thus the slopes of $\tau$ are all zero.
	Let $C = N_X^{\tau}$.
	This is a $2$-dimensional vector space over $E$ with $N_X = C \otimes_E \breve{E}$ and $h$ induces an $E|F$-hermitian form on $C$.
	A priori, there are two possibilities for $(C,h)$, either $h$ is split or non-split.
	The group $\U(C,h)$ of automorphisms is isomorphic to $\QIsog(\mathbb{X}, \iota_{\mathbb{X}}, \lambda_{\mathbb{X}})$.
	But the unitary groups for $h$ split and $h$ non-split are not isomorphic and do not contain each other as a closed subgroup.
	Thus the condition on $\QIsog(\mathbb{X}, \iota_{\mathbb{X}}, \lambda_{\mathbb{X}})$ implies that $h$ is split.
	
	Assume we are given two different objects $(X,\iota,\lambda)$ and $(X',\iota',\lambda')$ as in the proposition.
	Then there is an isomorphism between the spaces $(C,h)$ and $(C',h')$ extending to an isomorphism of $N_X$ and $N_{X'}$ respecting all structure.
	This corresponds to a quasi-isogeny $\varphi: (X,\iota,\lambda) \to (X',\iota',\lambda')$.
	
	Now we prove the existence of $(\mathbb{X},\iota_{\mathbb{X}},\lambda_{\mathbb{X}})$.
	We start with a $\Pi$-modular lattice $\Lambda$ in a $2$-dimensional vector space $(C,h)$ over $E$ with split hermitian form.
	Then $M = \Lambda \otimes_{O_E} \breve{O}_E$ is an $\breve{O}_E$-lattice in $N = C \otimes_{E} \breve{E}$.
	The $\sigma$-linear operator $\tau = 1 \otimes \sigma$ on $N$ has slopes are all $0$.
	We can extend $h$ to $N$ such that
	\begin{equation*}
		h(\tau x, \tau y) = h(x,y)^{\sigma},
	\end{equation*}
	for all $x,y \in N$.
	The operators ${\bf F}$ and ${\bf V}$ are given by the equations $\tau = \Pi {\bf V}^{-1} = {\bf F} \overbar{\Pi}^{-1}$.
	Finally, the alternating form $\langle \,, \rangle$ is defined via
	\begin{equation*}
		\langle x,y \rangle = \Tr_{\breve{E}|\breve{F}} \left(\frac{1}{t\vartheta} \cdot h(x,y) \right),
	\end{equation*}
	for $x,y \in N$.
	The lattice $M \subseteq N$ is the Dieudonn\'e module of the object $(\mathbb{X},\iota_{\mathbb{X}},\lambda_{\mathbb{X}})$.
	We leave it to the reader to check that this is indeed an object as considered above.
\end{proof}

We fix such an object $(\mathbb{X},\iota_{\mathbb{X}},\lambda_{\mathbb{X}})$ over $\Spec \overbar{k}$ from the proposition.
We define the functor $\NEnaive$ on $\Nilp$ as in Definition \ref{RP_defnaive}.

%

\begin{rmk}
	$\NEnaive$ is pro-representable by a formal scheme, formally locally of finite type over $\Spf \breve{O}_F$, \emph{cf.}\ \cite[Thm.\ 3.25]{RZ96}.
\end{rmk}

We now study the $\overbar{k}$-valued points of the space $\NEnaive$.
Let $N = N_{\mathbb{X}}$ be the rational Dieudonn\'e-module of $(\mathbb{X},\iota_{\mathbb{X}},\lambda_{\mathbb{X}})$.
This is a $4$-dimensional vector space over $\breve{F}$, equipped with an action of $E$, with two operators ${\bf F}$ and ${\bf V}$ and an alternating form $\langle \,, \rangle$.

Let $(X, \iota, \lambda, \varrho) \in \NEnaive(\overbar{k})$.
This corresponds to an $\breve{O}_F$-lattice $M = M_X \subseteq N$ which is stable under the actions of ${\bf F}$, ${\bf V}$ and $O_E$.
The condition on the kernel of $\lambda$ implies that $M = \Pi M^{\vee}$ for
\begin{equation*}
	M^{\vee} = \{ x \in N \mid \langle x,y \rangle \in \breve{O}_F \text{ for all } y \in M \}.
\end{equation*}
The alternating form $\langle \,, \rangle$ induces an $\breve{E}|\breve{F}$-hermitian form $h$ on $N$, seen as $2$-dimensional vector space over $\breve{E}$ (see equation \eqref{RU_herm}):
\begin{equation*}
	h(x,y) = \langle \Pi x, y \rangle - \overbar{\Pi} \langle x,y \rangle.
\end{equation*}
We can recover the form $\langle \,, \rangle$ from $h$ via
\begin{equation} \label{RU_altherm}
	\langle x, y \rangle = \Tr_{\breve{E}|\breve{F}} \left( \frac{1}{t\vartheta} \cdot h(x,y) \right).
\end{equation}
Since the inverse different of $E|F$ is $\mathfrak{D}_{E|F}^{-1} = \frac{1}{t}O_E$ (see Lemma \ref{LA_diff}), this implies that $M$ is $\Pi$-modular with respect to $h$, as $\breve{O}_E$-lattice in $N$.
We denote the dual of $M$ with respect to $h$ by $M^{\sharp}$.
There is a natural bijection
\begin{equation} \label{RU_geompts}
	\NEnaive (\overbar{k}) = \{ \breve{O}_E \text{-lattices } M \subseteq N \mid M = \Pi M^{\sharp}, \pi_0 M \subseteq {\bf V}M \subseteq M \}.
\end{equation}
Recall that $\tau = \Pi {\bf V}^{-1}$ is a $\sigma$-linear operator on $N$ with slopes all $0$.
Further $C = N^{\tau}$ is a $2$-dimensional $E$-vector space with hermitian form $h$.

\begin{lem} \label{RU_latt}
	Let $M \in \NEnaive(\overbar{k})$.
	Then:
	\begin{enumerate}
		\item $M + \tau(M)$ is $\tau$-stable.
		\item Either $M$ is $\tau$-stable and $\Lambda_1 = M^{\tau} \subseteq C$ is $\Pi$-modular with respect to $h$, or $M$ is not $\tau$-stable and then $\Lambda_0 = (M + \tau(M))^{\tau} \subseteq C$ is unimodular.
	\end{enumerate}
\end{lem}

The proof is the same as that of \cite[Lemma 3.2]{KR14}.
We identify $N$ with $C \otimes_{E} \breve{E}$.
For any $\tau$-stable lattice $M \in \NEnaive(\overbar{k})$, we have $M = \Lambda_1 \otimes_{O_E} \breve{O}_E$.
If $M \in \NEnaive(\overbar{k})$ is not $\tau$-stable, there is an inclusion $M \subseteq \Lambda_0 \otimes_{O_E} \breve{O}_E$ of index $1$.
Recall from Proposition \ref{LA_latt} that the isomorphism class of a $\Pi$-modular or unimodular lattice $\Lambda \subseteq C$ is determined by the norm ideal
\begin{equation*}
	\Nm(\Lambda) = \langle \{ h(x,x) | x \in \Lambda \} \rangle.
\end{equation*}
There are always at least two types of unimodular lattices.
However, not all of them appear in the description of $\NEnaive(\overbar{k})$.

\begin{lem} \label{RU_IS}
	\begin{enumerate}
		\item \label{RU_IS1} Let $\Lambda \subseteq C$ be a unimodular lattice with $\Nm (\Lambda) \subseteq \pi_0 O_F$.
		There is an injection
		\begin{equation*}
			i_{\Lambda}: \mathbb{P}(\Lambda / \Pi \Lambda)(\overbar{k}) \inj \NEnaive(\overbar{k}),
		\end{equation*}
		that maps a line $\ell \subseteq \Lambda / \Pi \Lambda \otimes_k \overbar{k}$ to its inverse image under the canonical projection
		\begin{equation*}
			\Lambda \otimes_{O_E} \breve{O}_E \to \Lambda / \Pi \Lambda \otimes_k \overbar{k}.
		\end{equation*}
		The $k$-valued points $\mathbb{P}(\Lambda / \Pi \Lambda)(k) \subseteq \mathbb{P}(\Lambda / \Pi \Lambda)(\overbar{k})$ are mapped to $\tau$-invariant Dieudonn\'e modules $M \subseteq \Lambda \otimes_{O_E} \breve{O}_E$ under this embedding.
		\item Identify $\mathbb{P}(\Lambda / \Pi \Lambda)(\overbar{k})$ with its image under $i_{\Lambda}$.
		The set $\NEnaive(\overbar{k})$ can be written as
		\begin{equation*}
			\NEnaive(\overbar{k}) = \smashoperator[r]{\bigcup_{\Lambda \subseteq C}} \mathbb{P}(\Lambda / \Pi \Lambda)(\overbar{k}),
		\end{equation*}
		where the union is taken over all lattices $\Lambda \subseteq C$ with $\Nm(\Lambda) \subseteq \pi_0 O_F$.
	\end{enumerate}
\end{lem}

\begin{proof}
	Let $\Lambda \subseteq C$ be a unimodular lattice.
	For any line $\ell \in \mathbb{P}(\Lambda / \Pi \Lambda)(\overbar{k})$, denote its preimage in $\Lambda \otimes \breve{O}_E$ by $M$.
	The inclusion $M \subseteq \Lambda \otimes \breve{O}_E$ has index $1$ and $M$ is an $\breve{O}_E$-lattice with $\Pi(\Lambda \otimes \breve{O}_E) \subseteq M$.
	Furthermore $\Lambda \otimes \breve{O}_E$ is $\tau$-invariant by construction, hence $\Pi(\Lambda \otimes \breve{O}_E) = {\bf V}(\Lambda \otimes \breve{O}_E) = {\bf F}(\Lambda \otimes \breve{O}_E)$.
	It follows that $M$ is stable under the actions of ${\bf F}$ and ${\bf V}$.
	Thus $M \in \NEnaive(\overbar{k})$ if and only if $M = \Pi M^{\sharp}$.
	The hermitian form $h$ induces a symmetric form $s$ on $\Lambda / \Pi \Lambda$.
	Now $M$ is $\Pi$-modular if and only if it is the preimage of an isotropic line $\ell \subseteq \Lambda / \Pi \Lambda \otimes \overbar{k}$.
	Note that $s$ is also anti-symmetric since we are in characteristic $2$.
	
	We first consider the case $\Nm (\Lambda) \subseteq \pi_0 O_F$.
	We can find a basis of $\Lambda$ such that $h$ has the form
	\begin{equation*}
		H_{\Lambda} = \begin{pmatrix}
			x & 1 \\
			1 &   \\
		\end{pmatrix}, \quad x \in \pi_0 O_F,
	\end{equation*}
	see \eqref{LA_lattmatrix}.
	It follows that the induced form $s$ is even alternating (because $x \equiv 0 \mod \pi_0$).
	Hence any line in $\Lambda / \Pi \Lambda \otimes \overbar{k}$ is isotropic.
	This implies that $i_{\Lambda}$ is well-defined, proving part \ref{RU_IS1} of the Lemma.
	
	Now assume that $\Nm (\Lambda) = O_F$.
	There is a basis $(e_1,e_2)$ of $\Lambda$ such that $h$ is represented by
	\begin{equation*}
		H_{\Lambda} = \begin{pmatrix}
			1 & 1 \\
			1 &   \\
		\end{pmatrix}.
	\end{equation*}
	The induced form $s$ is given by the same matrix and $\ell = \overbar{k} \cdot e_2$ is the only isotropic line in $\Lambda / \Pi \Lambda$.
	Since $\ell$ is already defined over $k$, the corresponding lattice $M \in \NEnaive(\overbar{k})$ is of the form $M = \Lambda_1 \otimes \breve{O}_E$ for a $\Pi$-modular lattice $\Lambda_1 \subseteq \Lambda$.
	But, by Proposition \ref{LA_lattRU}, any $\Pi$-modular lattice in $C$ is contained in a unimodular lattice $\Lambda'$ with $\Nm (\Lambda') \subseteq \pi_0 O_F$.
	
	It follows that we can write $\NEnaive(\overbar{k})$ as a union
	\begin{equation*}
		\NEnaive(\overbar{k}) = \smashoperator[r]{\bigcup_{\Lambda \subseteq C}} \mathbb{P}(\Lambda / \Pi \Lambda)(\overbar{k}),
	\end{equation*}
	where the union is taken over all unimodular lattices $\Lambda \subseteq C$ with $\Nm(\Lambda) \subseteq \pi_0 O_F$.
	This shows the second part of the Lemma.
\end{proof}

\begin{rmk} \label{RU_ISrmk}
	We can use Proposition \ref{LA_lattRU} to describe the intersection behaviour of the projective lines in $\NEnaive(\overbar{k})$.
	A $\tau$-invariant point $M \in \NEnaive(\overbar{k})$ corresponds to the $\Pi$-modular lattice $\Lambda_1 = M^{\tau} \subseteq C$.
	If $\Nm(\Lambda_1) \subseteq \pi_0^2 O_F$, there are $q+1$ lines going through $M$.
	If $\Nm(\Lambda_1) = \pi_0 O_F$, the point $M$ is contained in one or $2$ lines, depending on whether $\Lambda_1$ is hyperbolic or not, see part \eqref{LA_lattRUh} and \eqref{LA_lattRUnh} of Proposition \ref{LA_lattRU}.
	The former case (\emph{i.e.}, $\Lambda_1$ is hyperbolic) appears if and only if $\pi_0 O_F = \Nm(\Lambda_1) = tO_F$ (see Lemma \ref{LA_hyp}). 
	This happens only for a specific type of (R-U) extension $E|F$, see page \pageref{LA_quadext}.
	We refer to Remark \ref{RU_rmkred}, Remark \ref{RU_rmkNE} and Section \ref{LM_naive} for a further discussion of this special case.\\
	On the other hand, each projective line in $\NEnaive(\overbar{k})$ contains $q+1$ $\tau$-invariant points.
	Such a $\tau$-invariant point $M$ is an intersection point of $2$ or more projective lines if and only if $|t| = |\pi_0|$ or $\Lambda_1 = M^{\tau} \subseteq C$ has a norm ideal satisfying $\Nm(\Lambda_1) \subseteq \pi_0^2 O_F$.
\end{rmk}

\begin{figure}[hbt]
	\centering
	\begin{subfigure}[b]{0.4\textwidth}
		\includegraphics[width = \textwidth]{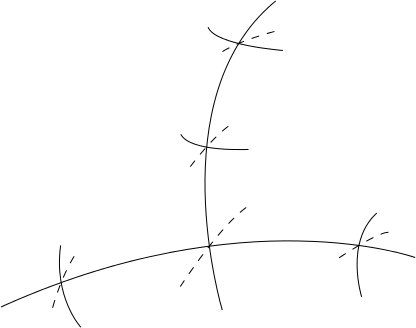}
		\caption{$e = 2$, $f = 1$, $v(t) = 2$.}
		\label{RU_NEnaivejpg1}
	\end{subfigure}
	\hspace{1cm}
	\begin{subfigure}[b]{0.4\textwidth}
		\includegraphics[width = \textwidth]{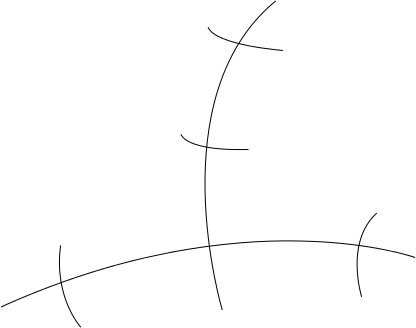}
		\caption{$e = 2$, $f = 1$, $v(t) = 1$.}
		\label{RU_NEnaivejpg2}
	\end{subfigure}
	\caption{
		The reduced locus of $\NEnaive$ for an (R-U) extension $E|F$ where $e$ and $f$ are the ramification index and the inertia degree of $F|\mathbb{Q}_2$ and $v(t)$ is the $\pi_0$-adic valuation of $t$.
		We always have $1 \leq v(t) \leq e$.
		The solid lines lie in $\NE \subseteq \NEnaive$.
	}
	\label{RU_NEnaivejpg}
\end{figure}

Let $\Lambda \subseteq C$ as in Lemma \ref{RU_IS}.
We denote by $X_{\Lambda}^{+}$ the formal $O_F$-module corresponding to the Dieudonn\'e module $M = \Lambda \otimes \breve{O}_E$.
There is a canonical quasi-isogeny
\begin{equation*}
	\varrho_{\Lambda}^{+}: \mathbb{X} \to X_{\Lambda}^{+}
\end{equation*}
of $F$-height $1$.
For $S \in \Nilp$, we define
\begin{equation*}
	\NEL(S) = \{ (X,\iota,\lambda,\varrho) \in \NEnaive(S) \mid (\varrho_{\Lambda}^{+} \times S) \circ \varrho \text{ is an isogeny} \}.
\end{equation*}
By \cite[Prop.\ 2.9]{RZ96}, the functor $\NEL$ is representable by a closed formal subscheme of $\NEnaive$.
On geometric points, we have
\begin{equation} \label{RU_NEL=P1geom}
	\NEL(\overbar{k}) \isoarrow \mathbb{P}(\Lambda / \Pi \Lambda)(\overbar{k}),
\end{equation}
as follows from Lemma \ref{RU_IS} \eqref{RU_IS1}.

\begin{prop} \label{RU_NEnaiveRL}
	The reduced locus of $\NEnaive$ is a union
	\begin{equation*}
		(\NEnaive)_{\red} = \smashoperator[r]{\bigcup_{\Lambda \subseteq C}} \NEL
	\end{equation*}
	where $\Lambda$ runs over all unimodular lattices in $C$ with $\Nm(\Lambda) \subseteq \pi_0 O_F$.
	For each $\Lambda$, there exists an isomorphism
	\begin{equation*}
		\NEL \isoarrow \mathbb{P}(\Lambda / \Pi \Lambda),
	\end{equation*}
	inducing the bijection \eqref{RU_NEL=P1geom} on $\overbar{k}$-valued points.
\end{prop}

The proof is analogous to that of Proposition \ref{RP_NEnaiveRL}.

\begin{rmk} \label{RU_rmkred}
	Similar to Remark \ref{RP_rmklatt} \eqref{RP_rmklatt3}, we let $(\NE)_{\red} \subseteq (\NEnaive)_{\red}$ be the union of all projective lines $\NEL$ corresponding to \emph{hyperbolic} unimodular lattices $\Lambda \subseteq C$.
	Later, we will define $\NE$ as a functor on $\Nilp$ and show that $\NE \simeq \MDr$, where $\MDr$ is the Drinfeld moduli problem (see Theorem \ref{RU_thm}, a description of the formal scheme $\MDr$ can be found in \cite[I.3]{BC91}).
	In particular, $(\NE)_{\red}$ is connected and each projective line in $(\NE)_{\red}$ has $q+1$ intersection points and there are $2$ lines intersecting in each such point. \\
	It might happen that $(\NE)_{\red} = (\NEnaive)_{\red}$ (see, for example, Figure \ref{RU_NEnaivejpg2}), if there are no non-hyperbolic unimodular lattices $\Lambda \subseteq C$ with $\Nm (\Lambda) \subseteq \pi_0 O_F$.
	In fact, this is the case if and only if $|t| = |\pi_0|$, see Proposition \ref{LA_latt} and Lemma \ref{LA_hyp}.
	(Note however that we still have $\NE \neq \NEnaive$, see Remark \ref{RU_rmkNE} and Section \ref{LM_naive}.)
	
	Assume $|t| \neq |\pi_0|$ and let $P \in \NE(\overbar{k})$ be an intersection point.
	Then, as in the case where $E|F$ is of type (R-P) (compare Remark \ref{RP_rmklatt} \eqref{RP_rmklatt3}), the connected component of $P$ in $((\NEnaive)_{\red} \setminus (\NE)_{\red}) \cup \{ P \}$ consists of a finite union of projective lines (corresponding to non-hyperbolic lattices, by definition of $(\NE)_{\red}$).
	In Figure \ref{RU_NEnaivejpg1}, these components are indicated by dashed lines (they consist of just one projective line in that case).
\end{rmk}

\subsection{The straightening condition}
\label{RU2}

As in the case (R-P), see section \ref{RP2}, we use the results of section \ref{POL} to define the straightening condition on $\NEnaive$.
By Theorem \ref{POL_thm} and Remark \ref{POL_rmk} \eqref{POL_rmk2}, there exists a principal polarization $\widetilde{\lambda}^0_{\mathbb{X}}$ on the framing object $(\mathbb{X},\iota_{\mathbb{X}},\lambda_{\mathbb{X}})$ such that the Rosati involution is the identity on $O_E$.
We set $\widetilde{\lambda}_{\mathbb{X}} = \widetilde{\lambda}^0_{\mathbb{X}} \circ \iota_{\mathbb{X}}(\Pi)$, which is again a polarization on $\mathbb{X}$ with the Rosati involution inducing the identity on $O_E$, but with kernel $\ker \widetilde{\lambda}_{\mathbb{X}} = \mathbb{X}[\Pi]$.
This polarization is unique up to a scalar in $O_E^{\times}$, \emph{i.e.}, any two polarizations $\widetilde{\lambda}_{\mathbb{X}}$ and $\widetilde{\lambda}_{\mathbb{X}}'$ with these properties satisfy
\begin{equation*}
	\widetilde{\lambda}_{\mathbb{X}}' = \widetilde{\lambda}_{\mathbb{X}} \circ \iota(\alpha),
\end{equation*}
for some $\alpha \in O_E^{\times}$.
For any $(X,\iota,\lambda,\varrho) \in \NEnaive(S)$,
\begin{equation*}
	\widetilde{\lambda} = \varrho^{\ast}(\widetilde{\lambda}_{\mathbb{X}}) = \varrho^{\ast}(\widetilde{\lambda}^0_{\mathbb{X}}) \circ \iota(\Pi)
\end{equation*}
is a polarization on $X$ with kernel $\ker \widetilde{\lambda} = X[\Pi]$, see Theorem \ref{POL_thm} \eqref{POL_thm2}.

Recall that a unimodular or $\Pi$-modular lattice $\Lambda \subseteq C$ is called \emph{hyperbolic} if there exists a basis $(e_1,e_2)$ of $\Lambda$ such that, with respect to this basis, $h$ has the form
\begin{equation*}
	\begin{pmatrix}
						& \Pi^i \\
		\overbar{\Pi}^i &		\\
	\end{pmatrix},
\end{equation*}
for $i = 0$ resp.\ $1$.
By Lemma \ref{LA_hyp}, this is the case if and only if $\Nm (\Lambda) = tO_F$.

\begin{prop} \label{RU_strprop}
	For a suitable choice of $(\mathbb{X},\iota_{\mathbb{X}},\lambda_{\mathbb{X}})$ and $\widetilde{\lambda}_{\mathbb{X}}$, the quasi-polarization
	\begin{equation*}
		\lambda_{\mathbb{X},1} = \frac{1}{t} (\lambda_{\mathbb{X}} + \widetilde{\lambda}_{\mathbb{X}})
	\end{equation*}
	is a polarization on $\mathbb{X}$.
	Let $(X,\iota,\lambda,\varrho) \in \NEnaive(\overbar{k})$ and $\widetilde{\lambda} = \varrho^{\ast}(\widetilde{\lambda}_{\mathbb{X}})$.
	Then $\lambda_1 = \frac{1}{t} (\lambda + \widetilde{\lambda})$ is a polarization if and only if $(X,\iota,\lambda,\varrho) \in \NEL(\overbar{k})$ for a hyperbolic unimodular lattice $\Lambda \subseteq C$.
\end{prop}

\begin{proof}
	On the rational Dieudonn\'e module $N = M_{\mathbb{X}} \otimes_{\breve{O}_F} \breve{F}$, denote by $\langle \,, \rangle$, $(\,,)$ and $\langle \,, \rangle_1$ the alternating forms induced by $\lambda_{\mathbb{X}}$, $\widetilde{\lambda}_{\mathbb{X}}$ and $\lambda_{\mathbb{X},1}$, respectively.
	The form $\langle \,, \rangle_1$ is integral on $M_{\mathbb{X}}$ if and only if $\lambda_{\mathbb{X},1}$ is a polarization on $\mathbb{X}$.
	We have
	\begin{align*}
		({\bf F}x, y) & = (x,{\bf V}y)^{\sigma}, \\
		(\Pi x ,y) & = (x, \Pi y), \\
		\langle x, y \rangle_1 & = \frac{1}{t} (\langle x,y \rangle + (x,y)),
	\end{align*}
	for all $x,y \in N$.
	The form $(\,,)$ induces an $\breve{E}$-bilinear alternating form $b$ on $N$ by the formula
	\begin{equation} \label{RU_altb}
		b(x,y) = c((\Pi x,y) - \overbar{\Pi} (x,y)).
	\end{equation}
	Here, $c$ is a unit in $\breve{O}_E$ such that $c \cdot \sigma(c)^{-1}= \overbar{\Pi}\Pi^{-1}$.
	Since $\frac{\overbar{\Pi}}{\Pi} = \frac{t - \Pi}{\Pi} \in 1 + \frac{t}{\Pi}\breve{O}_E$, we can even choose $c \in 1 + t\Pi^{-1}\breve{O}_E$.
	The dual of $M$ with respect to this form is again $M^{\sharp} = \Pi^{-1} M$, since
	\begin{equation*}
		(x,y) = \Tr_{\breve{E}|\breve{F}} \left(\frac{1}{t\vartheta c} \cdot b(x,y) \right),
	\end{equation*}
	and the inverse different of $E|F$ is given by $\mathfrak{D}_{E|F}^{-1} = t^{-1}O_E$, see Lemma \ref{LA_diff}.
	Now $b$ is invariant under the $\sigma$-linear operator $\tau = \Pi {\bf V}^{-1} = {\bf F} \overbar{\Pi}^{-1}$, because
	\begin{equation*}
		b(\tau x, \tau y) = b({\bf F} \overbar{\Pi}^{-1} x, \Pi {\bf V}^{-1} y) = \frac{c}{\sigma(c)} \cdot b(\overbar{\Pi}^{-1} x, \Pi y)^{\sigma} = b(x,y)^{\sigma}.
	\end{equation*}
	Hence $b$ defines an $E$-linear alternating form on $C$.
	
	We choose the framing object $(\mathbb{X},\iota_{\mathbb{X}},\lambda_{\mathbb{X}})$ such that $M_{\mathbb{X}}$ is $\tau$-invariant (see Lemma \ref{RU_latt}) and such that $\Lambda_1 = M_{\mathbb{X}}^{\tau}$ is hyperbolic.
	We can find a basis $(e_1,e_2)$ of $\Lambda_1$ such that
	\begin{equation*}
		h \, \weq \begin{pmatrix}
						  & \Pi \\
			\overbar{\Pi} &		\\
		\end{pmatrix}, \quad b \, \weq \begin{pmatrix}
			& u \\
			-u & \\
		\end{pmatrix},
	\end{equation*}
	for some $u \in E^{\times}$.
	Since $\widetilde{\lambda}_{\mathbb{X}}$ has the same kernel as $\lambda_{\mathbb{X}}$, we have $u = \overbar{\Pi} u'$ for some unit $u' \in O_E^{\times}$.
	We can choose $\widetilde{\lambda}_{\mathbb{X}}$ such that $u' = 1$ and $u = \overbar{\Pi}$.
	Now $\frac{1}{t} (h(x,y) + b(x,y))$ is integral for all $x,y \in \Lambda_1$.
	Hence $\frac{1}{t} (h(x,y) + b(x,y))$ is also integral for all $x,y \in M_{\mathbb{X}}$.
	For all $x,y \in M_{\mathbb{X}}$, we have
	\begin{align*}
		\langle x,y \rangle_1 & = \frac{1}{t} (\langle x,y \rangle + (x,y)) = \frac{1}{t} \Tr_{\breve{E}|\breve{F}} \left(\frac{1}{t\vartheta} \cdot h(x,y) + \frac{1}{t\vartheta c} \cdot b(x,y) \right) \\
		& = \Tr_{\breve{E}|\breve{F}} \left(\frac{1}{t^{2}\vartheta} \cdot (h(x,y) + b(x,y)) \right) + \Tr_{\breve{E}|\breve{F}} \left(\frac{1-c}{t^2\vartheta c} \cdot b(x,y) \right).
	\end{align*}
	The first summand is integral since $\frac{1}{t} (h(x,y) + b(x,y))$ is integral.
	The second summand is integral since $1-c$ is divisible by $t\Pi^{-1}$ and $b(x,y)$ lies in $\Pi \breve{O}_E$.
	It follows that the second summand above is integral as well.
	Hence $\langle \,, \rangle_1$ is integral on $M_{\mathbb{X}}$ and this implies that $\lambda_{\mathbb{X},1}$ is a polarization on $\mathbb{X}$.
	
	Now let $(X,\iota,\lambda,\varrho) \in \NEnaive(\overbar{k})$ and denote by $M \subseteq N$ its Dieudonn\'e module.
	Assume that $\lambda_1 = t^{-1}(\lambda + \widetilde{\lambda})$ is a polarization on $X$.
	Then $\langle \,, \rangle_1$ is integral on $M$.
	But this is equivalent to $t^{-1}(h(x,y) + b(x,y))$ being integral for all $x,y \in M$.
	For $x= y$, we have
	\begin{equation*}
		h(x,x) = h(x,x) + b(x,x) \in t\breve{O}_F.
	\end{equation*}
	Let $\Lambda \subseteq C$ be the unimodular or $\Pi$-modular lattice given by $\Lambda = M^{\tau}$ resp.\ $\Lambda = (M + \tau(M))^{\tau}$, see Lemma \ref{RU_latt}.
	Then $h(x,x) \in tO_F$ for all $x \in \Lambda$.
	Thus $\Nm(\Lambda) \subseteq tO_F$ and, by minimality, this implies that $\Nm(\Lambda) = tO_F$ and $\Lambda$ is hyperbolic (see Lemma \ref{LA_hyp}).
	Hence, in either case, the point corresponding to $(X,\iota,\lambda,\varrho)$ lies in $\NELpr$ for a hyperbolic lattice $\Lambda'$.

	Conversely, assume that $(X,\iota,\lambda,\varrho) \in \NEL(\overbar{k})$ for some hyperbolic lattice $\Lambda \subseteq C$.
	We want to show that $\lambda_1$ is a polarization on $X$.
	This follows if $\langle \,, \rangle_1$ is integral on $M$, or equivalently, if $t^{-1} (h(x,y) + b(x,y))$ is integral on $M$.
	For this, it is enough to show that $t^{-1} (h(x,y) + b(x,y))$ is integral on $\Lambda$.
	Let $\Lambda' \subseteq C$ be the unimodular lattice generated by $\overbar{\Pi}^{-1} e_1$ and $e_2$, where $(e_1,e_2)$ is the basis of the $\Pi$-modular lattice $\Lambda_1 = M_{\mathbb{X}}$.
	With respect to the basis $(\overbar{\Pi}^{-1} e_1, e_2)$, we have
	\begin{equation*}
		h \, \weq \begin{pmatrix}
			& 1 \\
			1 & \\
		\end{pmatrix}, \quad b \, \weq \begin{pmatrix}
			& 1 \\
			-1 & \\
		\end{pmatrix}.
	\end{equation*}
	In particular, $\Lambda'$ is a hyperbolic lattice and $t^{-1}(h + b)$ is integral on $\Lambda'$.
	By Proposition \ref{LA_latt}, there exists an element $g \in \SU(C,h)$ with $g\Lambda = \Lambda'$.
	Since $\det g = 1$, the alternating form $b$ is invariant under $g$.
	Thus $t^{-1} (h + b)$ is also integral on $\Lambda$.
\end{proof}

From now on, we assume that $(\mathbb{X},\iota_{\mathbb{X}},\lambda_{\mathbb{X}})$ and $\widetilde{\lambda}_{\mathbb{X}}$ are chosen in a way such that
\begin{equation*}
	\lambda_{\mathbb{X},1} = \frac{1}{t} (\lambda_{\mathbb{X}} + \widetilde{\lambda}_{\mathbb{X}}) \in \Hom(\mathbb{X},\mathbb{X}^{\vee}).
\end{equation*}

\begin{defn} \label{RU_strdef}
	A tuple $(X,\iota,\lambda,\varrho) \in \NEnaive(S)$ satisfies the \emph{straightening condition} if
	\begin{equation} \label{RU_strcond}
		\lambda_1 = \frac{1}{t} (\lambda + \widetilde{\lambda}) \in \Hom(X,X^{\vee}).
	\end{equation}
\end{defn}

This condition is independent of the choice of $\widetilde{\lambda}_{\mathbb{X}}$.
In fact, we can only change $\widetilde{\lambda}_{\mathbb{X}}$ by a scalar of the form $1 + t\Pi^{-1}u$, $u \in O_E$.
But if $\widetilde{\lambda}'_{\mathbb{X}} = \widetilde{\lambda}_{\mathbb{X}} \circ \iota(1 + t\Pi^{-1}u)$, then $\lambda'_{\mathbb{X},1} = \lambda_{\mathbb{X},1} + \widetilde{\lambda}_{\mathbb{X}} \circ \iota(\Pi^{-1}u) =  \lambda_{\mathbb{X},1} + \widetilde{\lambda}^0_{\mathbb{X}} \circ \iota(u)$ and $\lambda'_1 = \lambda_1 + \varrho^{\ast}(\widetilde{\lambda}^0_{\mathbb{X}}) \circ \iota(u)$.
Clearly, $\lambda'_1$ is a polarization if and only if $\lambda_1$ is one.

For $S \in \Nilp$, let $\NE(S)$ be the set of all tuples $(X,\iota,\lambda,\varrho) \in \NEnaive(S)$ that satisfy the straightening condition.
By \cite[Prop.\ 2.9]{RZ96}, the functor $\NE$ is representable by a closed formal subscheme of $\NEnaive$.

\begin{rmk} \label{RU_rmkNE}
	The reduced locus of $\NE$ is given by
	\begin{equation*}
		(\NE)_{\red} = \smashoperator[r]{\bigcup_{\Lambda \subseteq C}} \NEL \simeq \smashoperator[r]{\bigcup_{\Lambda \subseteq C}} \mathbb{P}(\Lambda / \Pi \Lambda),
	\end{equation*}
	where the union goes over all \emph{hyperbolic} unimodular lattices $\Lambda \subseteq C$.
	Note that, depending on the form of the (R-U) extension $E|F$, it may happen that all unimodular lattices are hyperbolic (when $|t| = |\pi_0|$) and in that case, we have $(\NE)_{\red} = (\NEnaive)_{\red}$.
	However, the equality does not extend to an isomorphism between $\NE$ and $\NEnaive$.
	This will be discussed in section \ref{LM_naive}.
\end{rmk}

\subsection{The main theorem for the case (R-U)}

As in the case (R-P), we want to establish a connection to the Drinfeld moduli problem.
Therefore, fix an embedding of $E$ into the quaternion division algebra $B$.
Let $(\mathbb{X}, \iota_{\mathbb{X}})$ be the framing object of the Drinfeld problem.
We want to construct a polarization $\lambda_{\mathbb{X}}$ on $\mathbb{X}$ with $\ker \lambda_{\mathbb{X}} = \mathbb{X}[\Pi]$ and Rosati involution given by $b \mapsto \vartheta b' \vartheta^{-1}$ on $B$.
Here $b \mapsto b'$ denotes the standard involution on $B$.

By Lemma \ref{LA_quat} \eqref{LA_quatRU}, there exists an embedding $E_1 \inj B$ of a ramified quadratic extension $E_1|F$ of type (R-P), such that $\Pi_1 \vartheta = - \vartheta \Pi_1$ for a prime element $\Pi_1 \in E_1$.
From Proposition \ref{RP_Dr} \eqref{RP_Dr1} we get a principal polarization $\lambda_{\mathbb{X}}^0$ on $\mathbb{X}$ with associated Rosati involution $b \mapsto \Pi_1 b' \Pi_1^{-1}$.
If we assume fixed choices of $E_1$ and $\Pi_1$, this is unique up to a scalar in $O_F^{\times}$.
We define
\begin{equation*}
	\lambda_{\mathbb{X}} = \lambda_{\mathbb{X}}^0 \circ \iota_{\mathbb{X}} (\Pi_1 \vartheta).
\end{equation*}
Since $\lambda_{\mathbb{X}}^0$ is a principal polarization and $\Pi_1 \vartheta$ and $\Pi$ have the same valuation in $O_B$, we have $\ker \lambda_{\mathbb{X}} = \mathbb{X}[\Pi]$. The Rosati involution of $\lambda_{\mathbb{X}}$ is $b \mapsto \vartheta b' \vartheta^{-1}$.
On the other hand, any polarization on $\mathbb{X}$ satisfying these two conditions can be constructed in this way (using the same choices for $E_1$ and $\Pi_1$).
Hence:

\begin{lem} \label{RU_pol}
	\begin{enumerate}
		\item \label{RU_pol1} There exists a polarization $\lambda_{\mathbb{X}}: \mathbb{X} \to \mathbb{X}^{\vee}$, unique up to a scalar in $O_F^{\times}$, with $\ker \lambda_{\mathbb{X}} = \mathbb{X}[\Pi]$ and associated Rosati involution $b \mapsto \vartheta b' \vartheta^{-1}$.
		\item \label{RU_pol2} Fix $\lambda_{\mathbb{X}}$ as in \eqref{RU_pol1} and let $(X,\iota_B,\varrho) \in \MDr(S)$.
		There exists a unique polarization $\lambda$ on $X$ with $\ker \lambda = X[\Pi]$ and Rosati involution $b \mapsto \vartheta b' \vartheta^{-1}$ such that $\varrho^{\ast}(\lambda_{\mathbb{X}}) = \lambda$ on $\overbar{S} = S \times_{\Spf \breve{O}_F} \overbar{k}$.
	\end{enumerate}
\end{lem}

Note also that the involution $b \mapsto \vartheta b' \vartheta^{-1}$ does not depend on the choice of $\vartheta \in E$.
We write $\iota_{\mathbb{X},E}$ for the restriction of $\iota_{\mathbb{X}}$ to $E \subseteq B$ and, in the same manner, we write $\iota_E$ for the restriction of $\iota_B$ to $E$ for any $(X,\iota_B,\varrho) \in \MDr(S)$.
Fix a polarization $\lambda_{\mathbb{X}}$ of $\mathbb{X}$ as in Lemma \ref{RU_pol} \eqref{RU_pol1}.
Accordingly for a tuple $(X,\iota_B,\varrho) \in \MDr(S)$, let $\lambda$ be the polarization given by Lemma \ref{RU_pol} \eqref{RU_pol2}.

\begin{lem} \label{RU_Dr}
	The tuple $(\mathbb{X},\iota_{\mathbb{X},E}, \lambda_{\mathbb{X}})$ is a framing object of $\NEnaive$.
	Moreover, the map
	\begin{equation*}
		(X,\iota_B,\varrho) \mapsto (X,\iota_E,\lambda,\varrho)
	\end{equation*}
	induces a closed embedding of formal schemes
	\begin{equation*}
		\eta: \MDr \inj \NEnaive.
	\end{equation*}
\end{lem}

\begin{proof}
	We follow the same argument as in the proof of Lemma \ref{RP_clemb}.
	Again it is enough to check that $\QIsog(\mathbb{X}, \iota_{\mathbb{X}}, \lambda_{\mathbb{X}})$ contains $\SU(C,h)$ as a closed subgroup and that $\iota_E$ satisfies the Kottwitz condition.
	
	By \cite[Prop.\ 5.8]{RZ14}, the special condition on $\iota_B$ implies the Kottwitz condition for $\iota_E$.
	It remains to show that $\SU(C,h) \subseteq  \QIsog(\mathbb{X}, \iota_{\mathbb{X}}, \lambda_{\mathbb{X}})$.
	But the group $G_{(\mathbb{X}, \iota_{\mathbb{X}})}$ of automorphisms of determinant $1$ of $(\mathbb{X}, \iota_{\mathbb{X}})$ is isomorphic to $\SL_{2,F}$ and $G_{(\mathbb{X}, \iota_{\mathbb{X}})} \subseteq \QIsog(\mathbb{X}, \iota_{\mathbb{X}}, \lambda_{\mathbb{X}})$ is a Zariski-closed subgroup by the same argument as in Lemma \ref{RP_clemb}.
	Hence the statement follows from the exceptional isomorphism $\SL_{2,F} \simeq \SU(C,h)$.
\end{proof}

As a next step, we want to show that this already induces a closed embedding
\begin{equation}
	\eta: \MDr \inj \NE.
\end{equation}
Let $\widetilde{E} \inj B$ an embedding of a ramified quadratic extension $\widetilde{E}|F$ of type (R-U) as in Lemma \ref{LA_quat} \eqref{LA_quatRU}.
On the framing object $(\mathbb{X},\iota_{\mathbb{X}})$ of $\MDr$, we define a polarization $\widetilde{\lambda}_{\mathbb{X}}$ via
\begin{equation*}
	\widetilde{\lambda}_{\mathbb{X}} = \lambda_{\mathbb{X}} \circ \iota_{\mathbb{X}} (\widetilde{\vartheta}),
\end{equation*}
where $\widetilde{\vartheta}$ is a unit in $\widetilde{E}$ of the form $\widetilde{\vartheta}^2 = 1 + (t^2/ \pi_0)\cdot u$, see Lemma \ref{LA_quat} \eqref{LA_quatRU}.
The Rosati involution of $\widetilde{\lambda}_{\mathbb{X}}$ induces the identity on $O_E$ and  we have
\begin{align*}
	\lambda_{\mathbb{X},1} & = \frac{1}{t} (\lambda_{\mathbb{X}} + \widetilde{\lambda}_{\mathbb{X}}) = \frac{1}{t} \cdot \lambda_{\mathbb{X}} \circ \iota_B (1 + \widetilde{\vartheta}) =  \lambda_{\mathbb{X}} \circ \iota_B (\widetilde{\Pi} / \pi_0) \\
	& =  \lambda_{\mathbb{X}} \circ \iota_B(\Pi^{-1} \gamma) \in \Hom({\mathbb{X}},{\mathbb{X}}^{\vee}),
\end{align*}
using the notation of Lemma \ref{LA_quat} \eqref{LA_quatRU}.
For $(X,\iota_B,\varrho) \in \MDr(S)$, we set $\widetilde{\lambda} = \lambda \circ \iota_B(\widetilde{\vartheta})$.
By the same calculation, we have $\lambda_1 = \frac{1}{t} (\lambda + \widetilde{\lambda}) \in \Hom(X,X^{\vee})$.
Thus the tuple $(X,\iota_E,\lambda,\varrho) = \eta(X,\iota_B,\varrho)$ satisfies the straightening condition.
Hence we get a closed embedding of formal schemes $\eta: \MDr \to \NE$ which is independent of the choice of $\widetilde{E}$.

\begin{thm} \label{RU_thm}
	$\eta: \MDr \to \NE$ is an isomorphism of formal schemes.
\end{thm}

We first check this for $\overbar{k}$-valued points:

\begin{lem} \label{RU_thmgeom}
	$\eta$ induces a bijection $\eta(\overbar{k}): \MDr(\overbar{k}) \to \NE(\overbar{k})$.
\end{lem}

\begin{proof}
	We only have to show surjectivity and we will use for this the Dieudonn\'e theory description of $\NEnaive(\overbar{k})$, see \eqref{RU_geompts}.
	The rational Dieudonn\'e-module $N = N_{\mathbb{X}}$ of $\mathbb{X}$ now carries additionally an action of $B$.
	The embedding $F^{(2)} \inj B$ given by
	\begin{equation}
		\gamma \mapsto \frac{\Pi \cdot \widetilde{\Pi}}{\pi_0},
	\end{equation}
	(see Lemma \ref{LA_quat} \eqref{LA_quatRU}) induces a $\mathbb{Z} / 2$-grading $N = N_0 \oplus N_1$.
	Here,
	\begin{align*}
		N_0 & = \{ x \in N \mid \iota(a) x = a x \text{ for all } a \in F^{(2)} \}, \\
		N_1 & = \{ x \in N \mid \iota(a) x = \sigma(a) x \text{ for all } a \in F^{(2)} \},
	\end{align*}
	for a fixed embedding $F^{(2)} \inj \breve{F}$.
	The operators ${\bf F}$ and ${\bf V}$ have degree $1$ with respect to this grading.
	The principal polarization 
	\begin{equation*}
		\lambda_{\mathbb{X},1} = \frac{1}{t} (\lambda_{\mathbb{X}} + \widetilde{\lambda}_{\mathbb{X}}) = \lambda_{\mathbb{X}} \circ \iota_{\mathbb{X}}(\Pi^{-1} \gamma)
	\end{equation*}
	induces an alternating form $\langle \,, \rangle_1$ on $N$ that satisfies
	\begin{equation*}
		\langle x,y \rangle_1 = \langle x, \iota(\Pi^{-1} \gamma) \cdot  y \rangle,
	\end{equation*}
	for all $x,y \in N$.
	Let $M \in \NE(\overbar{k}) \subseteq \NEnaive(\overbar{k})$ be an $\breve{O}_F$-lattice in $N$.
	We claim that $M \in \MDr(\overbar{k})$.
	For this, it is necessary that $M$ is stable under the action of $O_F^{(2)}$ (since $O_B = O_F[\Pi, \gamma] = O_F^{(2)}[\Pi]$, see Lemma \ref{LA_quat} \eqref{LA_quatRU}) or equivalently, that $M$ respects the grading of $N$, \emph{i.e.}, $M = M_0 \oplus M_1$ for $M_i = M \cap N_i$.
	Furthermore $M$ has to satisfy the \emph{special} condition:
	\begin{equation*}
		\dim M_0 / {\bf V}M_1 = \dim M_1 / {\bf V}M_0 = 1.
	\end{equation*}
	We first show that $M = M_0 \oplus M_1$.
	Let $y = y_0 + y_1 \in M$ with $y_i \in N_i$.
	Since $M = \Pi M^{\vee}$, we have
	\begin{equation*}
		\langle x, \iota(\Pi)^{-1} y \rangle = \langle x, \iota(\Pi)^{-1} y_0 \rangle + \langle x, \iota(\Pi)^{-1} y_1 \rangle \in \breve{O}_F,
	\end{equation*}
	for all $x \in M$.
	Together with
	\begin{align*}
		\langle x, y \rangle_1 & = \langle x, y_0 \rangle_1 + \langle x, y_1 \rangle_1 = \langle x, \iota( \widetilde{\Pi} / \pi_0) y_0 \rangle + \langle x, \iota( \widetilde{\Pi} / \pi_0) y_1 \rangle \\
		& = \gamma \cdot \langle x, \iota (\Pi^{-1}) y_0 \rangle + (1-\gamma) \cdot \langle x, \iota (\Pi^{-1}) y_1 \rangle \in \breve{O}_F,
	\end{align*}
	this implies that $\langle x, \iota (\Pi^{-1}) y_0 \rangle$ and $\langle x, \iota (\Pi^{-1}) y_1 \rangle$ lie in $\breve{O}_F$ for all $x \in M$.
	Hence, $y_0, y_1 \in M$ and this means that $M$ respects the grading.
	It follows that $M$ is stable under the action of $O_B$.
	
	In order to show that $M$ is special, note that
	\begin{equation*}
		\langle {\bf V}x, {\bf V}y \rangle_1^{\sigma} = \langle {\bf F}{\bf V}x, y \rangle_1 = \pi_0 \cdot \langle x,y \rangle_1 \in \pi_0 \breve{O}_F,
	\end{equation*}
	for all $x,y \in M$.
	The form $\langle \,, \rangle_1$ comes from a principal polarization, so it induces a perfect form on $M$.
	Now it is enough to show that also the restrictions of $\langle \,, \rangle_1$ to $M_0$ and $M_1$ are perfect.
	Indeed, if $M$ was not special, we would have $M_i = {\bf V}M_{i+1}$ for some $i$ and this would contradict $\langle \,, \rangle_1$ being perfect on $M_i$.
	We prove that $\langle \,, \rangle_1$ is perfect on $M_i$ by showing $\langle M_0, M_1 \rangle_1 \subseteq \pi_0 \breve{O}_F$.
	
	Let $x \in M_0$ and $y \in M_1$.
	Then,
	\begin{align*}
		\langle x, y \rangle_1 & = (1-\gamma) \cdot \langle x, \iota(\Pi)^{-1} y \rangle, \\
		\langle x, y \rangle_1 & = - \langle y, x \rangle_1 = - \gamma \cdot \langle y, \iota(\Pi)^{-1} x \rangle = \gamma \cdot \langle x, \iota(\overbar{\Pi})^{-1} y \rangle.
	\end{align*}
	We take the difference of these two equations.
	From $\Pi \equiv \overbar{\Pi} \mod \pi_0$, it follows that $\langle x, \iota(\Pi)^{-1} y \rangle \equiv 0 \mod \pi_0$ and thus also $\langle x, y \rangle_1 \equiv 0 \mod \pi_0$.
	The form $\langle \,, \rangle_1$ is hence perfect on $M_0$ and $M_1$ and the special condition follows.
	This finishes the proof of Lemma \ref{RU_thmgeom}.
\end{proof}

\begin{proof}[Proof of Theorem \ref{RU_thm}]
	Let $(\mathbb{X}, \iota_{\mathbb{X}})$ be a framing object for $\MDr$ and let further
	\begin{equation*}
		\eta(\mathbb{X}, \iota_{\mathbb{X}}) = (\mathbb{X}, \iota_{\mathbb{X},E}, \lambda_{\mathbb{X}})
	\end{equation*}
	be the corresponding framing object for $\NE$.
	We fix an embedding $F^{(2)} \inj B$ as in Lemma \ref{LA_quat} \eqref{LA_quatRU}.
	For $S \in \Nilp$, let $(X,\iota,\lambda,\varrho) \in \NE(S)$ and $\widetilde{\lambda} = \varrho^{\ast}(\widetilde{\lambda}_{\mathbb{X}})$.
	We have
	\begin{align*}
		\varrho^{-1} \circ \iota_{\mathbb{X}} (\gamma) \circ \varrho & = \varrho^{-1} \circ \iota_{\mathbb{X}} (\Pi) \circ \lambda_{\mathbb{X}}^{-1} \circ \lambda_{\mathbb{X},1} \circ \varrho \\
		& = \iota(\Pi) \circ \lambda^{-1} \circ \lambda_1 \in \End(X),
	\end{align*}
	for $\lambda_1 = t^{-1} (\lambda + \widetilde{\lambda})$, since $\ker \lambda = X[\Pi]$.
	But $O_B = O_F[\Pi, \gamma	]$ (see Lemma \ref{LA_quat} \eqref{LA_quatRU}), so this already induces an $O_B$-action $\iota_B$ on $X$.
	It remains to show that $(X,\iota_B,\varrho)$ satisfies the \emph{special} condition (see the discussion before Proposition \ref{RP_Dr} for a definition).
	
	The special condition is open and closed (see \cite[p.\ 7]{RZ14}) and $\eta$ is bijective on $\overbar{k}$-points.
	Hence $\eta$ induces an isomorphism on reduced subschemes
	\begin{equation*}
		(\eta)_{\red}: (\MDr)_{\red} \isoarrow (\NE)_{\red},
	\end{equation*}
	because $(\MDr)_{\red}$ and $(\NE)_{\red}$ are locally of finite type over $\Spec \overbar{k}$.
	It follows that $\eta: \MDr \to \NE$ is an isomorphism.
\end{proof}

\subsection{Deformation theory of intersection points} \label{LM_naive}

In this section, we will study the deformation rings of certain geometric points in $\NEnaive$ with the goal of proving that $\NE \subseteq \NEnaive$ is a strict inclusion even in the case $|t| = |\pi_0|$.
In contrast to the non-$2$-adic case, we are not able to use the theory of local models (see \cite{PRS13} for a survey) since there is in general no normal form for the lattices $\Lambda \subseteq C$, see Proposition \ref{LA_latt} and \cite[Thm.\ 3.16]{RZ96}.\footnote{It is possible define a local model for the non-naive spaces $\NE$ (also in the case (R-P)) and establish a local model diagram as in \cite[3.27]{RZ96}. The local model is then isomorphic to the local model of the Drinfeld moduli problem. This will be part of a future paper of the author.}
Thus we will take the more direct approach of studying the deformations of a fixed point $(X,\iota,\lambda,\varrho) \in \NEnaive(\overbar{k})$ and using the theory of Grothendieck-Messing (\cite{Me72}).

Let $\Lambda \subseteq C$ be a $\Pi$-modular hyperbolic lattice.
By Lemma \ref{RU_IS}, there is a unique point $x = (X,\iota,\lambda,\varrho) \in \NEnaive(\overbar{k})$ with a $\tau$-stable Dieudonn\'e module $M \subseteq C \otimes_E \breve{E}$ and $M^{\tau} = \Lambda$.
Since $\Lambda$ is hyperbolic, $x$ satisfies the straightening condition, \emph{i.e.}, $x \in \NE(\overbar{k})$.
(In Figure \ref{RU_NEnaivejpg}, $x$ would lie on the intersection of two solid lines.)

Let $\widehat{\mathcal{O}}_{\NEnaive,x}$  be the formal completion of the local ring at $x$.
It represents the following deformation functor $\Defo_x$.
For an artinian $\breve{O}_F$-algebra $R$ with residue field $\overbar{k}$, we have
\begin{equation*}
	\Defo_x(R) = \{ (Y,\iota_Y,\lambda_Y) / R \mid Y_{\overbar{k}} \cong X \},
\end{equation*}
where $(Y,\iota_Y,\lambda_Y)$ satisfies the usual conditions (see section \ref{RU1}) and the isomorphism $Y_{\overbar{k}} \cong X$ is actually an isomorphism of tuples $(Y_{\overbar{k}},\iota_Y,\lambda_Y) \cong (X,\iota,\lambda)$ as in Definition \ref{RP_isonaive}.

Now assume the quotient map $R \to \overbar{k}$ is an $O_F$-pd-thickening (cf.\ \cite{Ahs11}).
For example, this is the case when $\mathfrak{m}^2 = 0$ for the maximal ideal $\mathfrak{m}$ of $R$.
Then, by Grothendieck-Messing theory (see \cite{Me72} and \cite{Ahs11}), we get an explicit description of $\Defo_x(R)$ in terms of liftings of the Hodge filtration:

The (relative) Dieudonn\'e crystal $\mathbb{D}_X(R)$ of $X$ evaluated at $R$ is naturally isomorphic to the free $R$-module $\Lambda \otimes_{O_F} R$ and this isomorphism is equivariant under the action of $O_E$ induced by $\iota$ and respects the perfect form $\Phi = \langle \,, \rangle \circ (1,\Pi^{-1})$ induced by $\lambda \circ \iota(\Pi^{-1})$.
The Hodge-filtration of $X$ is given by $\mathcal{F}_X = V \cdot \mathbb{D}_X(\overbar{k}) \cong \Pi \cdot (\Lambda \otimes_{O_F} \overbar{k}) \subseteq \Lambda \otimes_{O_F} \overbar{k}$.

A point $Y \in \Defo_x(R)$ now corresponds, via Grothendieck-Messing, to a direct summand $\mathcal{F}_Y \subseteq \Lambda \otimes_{O_F} R$ of rank $2$ lifting $\mathcal{F}_X$, stable under the $O_E$-action and totally isotropic with respect to $\Phi$.
Furthermore, it has to satisfy the Kottwitz condition (see section \ref{RU1}): For the action of $\alpha \in O_E$ on $\Lie Y = (\Lambda \otimes_{O_F} R) / \mathcal{F}_Y$, we have
\begin{equation*}
	\charp(\Lie Y, T \mid \iota(\alpha)) = (T - \alpha)(T - \overbar{\alpha}).
\end{equation*}
Let us now fix an $O_E$-basis $(e_1, e_2)$ of $\Lambda$ and let us write everything in terms of the $O_F$-basis $(e_1,e_2,\Pi e_1, \Pi e_2)$.
Since $\Lambda$ is hyperbolic, we can fix $(e_1,e_2)$ such that $h$ is represented by the matrix
\begin{equation*}
	h \weq \left( \begin{array}{cc}
		& \Pi \\
		\overbar{\Pi} & \\
	\end{array} \right),
\end{equation*}
and then
\begin{equation*}
	\Phi = \Tr_{E|F} \frac{1}{t\vartheta} h(\cdot, \Pi^{-1} \cdot) \weq
	\left( \begin{array}{cc|cc}
			    &   t/\pi_0 &   & 1 \\
			    &    & -1 &   \\ \hline
			    & -1 + t^2/\pi_0 &   & t  \\
			 1 &    &   &   \\
		\end{array} \right).
\end{equation*}
An $R$-basis $(v_1,v_2)$ of $\mathcal{F}_Y$ can now be chosen such that
\begin{equation*}
	(v_1 v_2) = \begin{pmatrix}
				y_{11} & y_{12} \\
				y_{21} & y_{22} \\
				1	   &		\\
					   & 1		\\
			\end{pmatrix},
\end{equation*}
with $y_{ij} \in R$.
As an easy calculation shows, the conditions on $\mathcal{F}_Y$ above are now equivalent to the following conditions on the $y_{ij}$:
\begin{align*}
	y_{11} + y_{22} & = t, \\
	y_{11} y_{22} - y_{12}y_{21} &= \pi_0,\\
	t(\frac{ty_{22}}{\pi_0} + 2) = y_{11}(\frac{ty_{22}}{\pi_0} + 2) &= y_{21}(\frac{ty_{22}}{\pi_0} + 2) = y_{12} (\frac{ty_{22}}{\pi_0} + 2) = 0.
\end{align*}
Let $T$ be the closed subscheme of $\Spec O_F[y_{11},y_{12},y_{21},y_{22}]$ given by these equations.
Let $T_y$ be the formal completion of the localization at the ideal generated by the $y_{ij}$ and $\pi_0$.
Then we have $\Defo_x(R) \cong T_y(R)$ for any $O_F$-pd-thickening $R \to \overbar{k}$.
In particular, the first infinitesimal neighborhoods of $\Defo_x$ and $T_y$ coincide.
The first infinitesimal neighborhood of $T_y$ is given by $\Spec O_F[y_{ij}]/((y_{ij})^2,y_{11}+y_{22}-t,\pi_0)$, hence $T_y$ has Krull dimension $3$ and so has $\Defo_x$. However, $\MDr$ is regular of dimension $2$, cf.\ \cite{BC91}.
Thus,

\begin{prop}
	$\NEnaive \neq \MDr$, even when $|t| = |\pi_0|$.
\end{prop}

Indeed, $\dim \widehat{\mathcal{O}}_{\NEnaive,x} = \dim \Defo_x = 3 > 2 = \dim \widehat{\mathcal{O}}_{\NE,x}$.

\section{A theorem on the existence of polarizations}
\label{POL}

In this section, we will prove the existence of the polarization $\widetilde{\lambda}$ for any $(X,\iota,\lambda,\varrho) \in \NEnaive(S)$ as claimed in the sections \ref{RP2} and \ref{RU2} in both the cases (R-P) and (R-U).
In fact, we will show more generally that $\widetilde{\lambda}$ exists even for the points of a larger moduli space $\ME$ where we forget about the polarization $\lambda$.

We start with the definition of the moduli space $\ME$.
Let $F|\mathbb{Q}_p$ be a finite extension (not necessarily $p=2$) and let $E|F$ be a quadratic extension (not necessarily ramified).
We denote by $O_F$ and $O_E$ the rings of integers, by $k$ the residue field of $O_F$ and by $\overbar{k}$ the algebraic closure of $k$.
Furthermore, $\breve{F}$ is the completion of the maximal unramified extension of $F$ and $\breve{O}_F$ its ring of integers.
Let $B$ be the quaternion division algebra over $F$ and $O_B$ the ring of integers.

If $E|F$ is unramified, we fix a common uniformizer $\pi_0 \in O_F \subseteq O_E$.
If $E|F$ is ramified and $p>2$, we choose a uniformizer $\Pi \in O_E$ such that $\pi_0 = \Pi^2 \in O_F$.
If $E|F$ is ramified and $p=2$, we use the notations of section \ref{LA} for the cases (R-P) and (R-U).

For $S \in \Nilp$, let $\ME(S)$ be the set of isomorphism classes of tuples $(X,\iota_E,\varrho)$ over $S$.
Here, $X$ is a formal $O_F$-module of dimension $2$ and height $4$ and $\iota_E$ is an action of $O_E$ on $X$ satisfying the Kottwitz condition for the signature $(1,1)$, \emph{i.e.}, the characteristic polynomial for the action of $\iota_E(\alpha)$ on $\Lie(X)$ is
\begin{equation} \label{POL_Kottwitz}
	\charp(\Lie X, T \mid \iota(\alpha)) = (T - \alpha)(T - \overbar{\alpha}),
\end{equation}
for any $\alpha \in O_E$, compare the definition of $\NEnaive$ in the sections \ref{RP} and \ref{RU}.
The last entry $\varrho$ is an $O_E$-linear quasi-isogeny 
\begin{equation*}
	\varrho: X \times_S \overbar{S} \to \mathbb{X} \times_{\Spec \overbar{k}} \overbar{S},
\end{equation*}
of height $0$ to the framing object $(\mathbb{X},\iota_{\mathbb{X},E})$ defined over $\Spec \overbar{k}$.
The framing object for $\ME$ is the Drinfeld framing object $(\mathbb{X},\iota_{\mathbb{X},B})$ where we restrict the $O_B$-action to $O_E$ for an arbitrary embedding $O_E \inj O_B$.
The special condition on $(\mathbb{X},\iota_{\mathbb{X},B})$ implies the Kottwitz condition for any $\alpha \in O_E$ by \cite[Prop.\ 5.8]{RZ14}.

\begin{rmk} \label{POL_rmk}
	\begin{enumerate}
		\item Up to isogeny, there is more than one pair $(X,\iota_E)$ over $\Spec \overbar{k}$ satisfying the conditions above.
		Indeed, let $N_X$ be the rational Dieudonn\'e module of $(X,\iota_E)$.
		This is a $4$-dimensional $\breve{F}$-vector space with an action of $O_E$.
		The Frobenius ${\bf F}$ on $N_X$ commutes with the action of $O_E$.
		For a suitable choice of a basis of $N_X$, it may be of either of the following two forms,
		\begin{equation*}
			{\bf F} = \begin{pmatrix}
				  & & 1 & \\
				&   & & 1 \\
			\pi_0 & &   & \\
				& \pi_0 & & \\
			\end{pmatrix} \! \sigma \; \text{ or } \; {\bf F} = \begin{pmatrix}
				\pi_0 & & & \\
				& \pi_0 & & \\
				& & 1 & \\
				& & & 1 \\
			\end{pmatrix} \! \sigma.
		\end{equation*}
		This follows from the classification of isocrystals, see for example \cite[p.\ 3]{RZ96}.
		In the left case, ${\bf F}$ is isoclinic of slope $1/2$ (the supersingular case), and in the right case, the slopes are $0$ and $1$.
		Our choice of the framing object above assures that we are in the supersingular case, since the framing object for the Drinfeld moduli problem can be written as a product of two formal $O_F$-modules of dimension $1$ and height $2$ (\emph{cf.}\ \cite[p.\ 136-137]{BC91}).
	
		\item \label{POL_rmk2} Let $p=2$ and $E|F$ ramified of type (R-P) or (R-U).
		We can identify the framing objects $(\mathbb{X},\iota_{\mathbb{X},E})$ for $\NEnaive$, $\MDr$ and $\ME$ by Lemma \ref{RP_Dr} and Lemma \ref{RU_Dr}.
		In this way, we obtain a forgetful morphism $\NEnaive \to \ME$.
		This is a closed embedding, since the existence of a polarization $\lambda$ for $(X,\iota_E,\varrho) \in \ME(S)$ is a closed condition by \cite[Prop.\ 2.9]{RZ96}.
	\end{enumerate}
\end{rmk}

By \cite[Thm.\ 3.25]{RZ96}, $\ME$ is pro-representable by a formal scheme over $\Spf \breve{O}_F$.
We will prove the following theorem in this section.

\begin{thm} \label{POL_thm}
	\begin{enumerate}
		\item \label{POL_thm1} There exists a principal polarization $\widetilde{\lambda}_{\mathbb{X}}$ on $(\mathbb{X},\iota_{\mathbb{X},E})$ such that the Rosati involution induces the identity on $O_E$, \emph{i.e.}, $\iota(\alpha)^{\ast} = \iota(\alpha)$ for all $\alpha \in O_E$.
		This polarization is unique up to a scalar in $O_E^{\times}$, that is, for any two polarizations $\widetilde{\lambda}_{\mathbb{X}}$ and $\widetilde{\lambda}_{\mathbb{X}}'$ of this form, there exists an element $\alpha \in O_E^{\times}$ such that $\widetilde{\lambda}_{\mathbb{X}}' = \widetilde{\lambda}_{\mathbb{X}} \circ \iota_{\mathbb{X},E}(\alpha)$.
		\item \label{POL_thm2} Fix $\widetilde{\lambda}_{\mathbb{X}}$ as in part \eqref{POL_thm1}.
		For any $S \in \Nilp$ and $(X,\iota_E,\varrho) \in \ME(S)$, there exists a unique principal polarization $\widetilde{\lambda}$ on $X$ such that the Rosati involution induces the identity on $O_E$ and such that $\widetilde{\lambda} = \varrho^{\ast}(\widetilde{\lambda}_{\mathbb{X}})$.
	\end{enumerate}
\end{thm}

\begin{rmk} \label{POL_p>2}
	\begin{enumerate}
		\item We will see later that this theorem describes a natural isomorphism between $\ME$ and another space $\MEpol$ which solves the moduli problem for tuples $(X,\iota_E,\widetilde{\lambda},\varrho)$ where $\widetilde{\lambda}$ is a principal polarization with Rosati involution the identity on $O_E$.
		This is an RZ-space for the symplectic group $\GSp_2(E)$ and thus the theorem gives us another geometric realization of an exceptional isomorphism of reductive groups, in this case $\GSp_2(E) \cong \GL_2(E)$.
		
		Since there is no such isomorphism in higher dimensions, the theorem does not generalize to these cases and a different approach is needed to formulate the straightening condition.
		\item \label{POL_p>2eq} With the Theorem \ref{POL_thm} established, one can give an easier proof of the isomorphism $\NE \isoarrow \MDr$ for the cases where $E|F$ is unramified or $E|F$ is ramified and $p >2$, which is the main theorem of \cite{KR14}.
	Indeed, the main part of the proof in loc.\ cit.\ consists of the Propositions 2.1 and 3.1, which claim the existence of a certain principal polarization $\lambda_X^0$ for any point $(X,\iota,\lambda,\varrho) \in \NE(S)$.
	But there is a canonical closed embedding $\NE \inj \ME$ and under this embedding, $\lambda_X^0$ is just the polarization $\widetilde{\lambda}$ of Theorem \ref{POL_thm}, for a suitable choice of $\widetilde{\lambda}_{\mathbb{X}}$ on the framing object.
	More explicitly, using the notation on page $2$ of loc.\ cit., we take $\widetilde{\lambda}_{\mathbb{X}} = \lambda_{\mathbb{X}} \circ \iota_{\mathbb{X}}^{-1}(\Pi) = \lambda_{\mathbb{X}}^{0} \circ \iota_{\mathbb{X}}(-\delta)$ in the unramified case and $\widetilde{\lambda}_{\mathbb{X}} = \lambda_{\mathbb{X}} \circ \iota_{\mathbb{X}}(\zeta^{-1})$ in the ramified case.
	\end{enumerate}
\end{rmk}

We will split the proof of this theorem into several lemmata.
As a first step, we use Dieudonn\'e theory to prove the statement for all geometric points.

\begin{lem} \label{POL_geompts}
	Part \eqref{POL_thm1} of theorem holds.
	Furthermore, for a fixed polarization $\widetilde{\lambda}_{\mathbb{X}}$ on $(\mathbb{X},\iota_{\mathbb{X},E})$ and for any $(X,\iota_E,\varrho) \in \ME(\overbar{k})$, the pullback $\widetilde{\lambda} = \varrho^{\ast}(\widetilde{\lambda}_{\mathbb{X}})$ is a polarization on $X$.
\end{lem}

\begin{proof}
This follows almost immediately from the theory of affine Deligne-Lusztig varieties (see, for example, \cite{CV}) since we are comparing the geometric points of RZ-spaces for the isomorphic groups $\GL_2(E)$ and $\GSp_2(E)$.

It is also possible to check this via a more direct computation using Dieudonn\'e theory, as we will indicate briefly.
Proceeding very similarly to Proposition \ref{RP_frnaive} or Proposition \ref{RU_frnaive} (cf.\ \cite{KR14} in the unramified case), we can associate to $\mathbb{X}$ a lattice $\Lambda$ in the $2$-dimensional $E$-vector space $C$ (the Frobenius invariant points of the (rational) Dieudonn\'e module).
The choice of a principal polarization on $\mathbb{X}$ with trivial Rosati involution corresponds now exactly to a choice of perfect alternating form on $\Lambda$.
It immediately follows that such a polarization exists and that it is unique up to a scalar in $O_E^{\times}$.

For the second part, let $X \in \ME(\overbar{k})$ and $M \subseteq C \otimes_E \breve{E}$ be its Dieudonn\'e module.
Since $\varrho$ has height $0$, we have
\begin{equation*}
	[M : M \cap (\Lambda \otimes_E \breve{E})] = [(\Lambda \otimes_E \breve{E}) : M \cap (\Lambda \otimes_E \breve{E})],
\end{equation*}
and one easily checks that a perfect alternating form $b$ on $\Lambda$ is also perfect on $M$.
\end{proof}

In the following, we fix a polarization $\widetilde{\lambda}_{\mathbb{X}}$ on $(\mathbb{X},\iota_{\mathbb{X},E})$ as in Theorem \ref{POL_thm} \eqref{POL_thm1}.
Let $(X,\iota_E,\varrho) \in \ME(S)$ for $S \in \Nilp$ and consider the pullback $\widetilde{\lambda} = \varrho^{\ast}(\widetilde{\lambda}_{\mathbb{X}})$.
In general, this is only a quasi-polarization.
It suffices to show that $\widetilde{\lambda}$ is a polarization on $X$.
Indeed, since $\varrho$ is $O_E$-linear and of height $0$, this is then automatically a principal polarization on $X$ such that the Rosati involution is the identity on $O_E$.

Define a subfunctor $\MEpol \subseteq \ME$ by
\begin{equation*}
	\MEpol(S) = \{ (X,\iota_E,\varrho) \in \ME(S) \mid \widetilde{\lambda} = \varrho^{\ast}(\widetilde{\lambda}_{\mathbb{X}}) \text{ is a polarization on } X \}.
\end{equation*}
This is a closed formal subscheme by \cite[Prop.\ 2.9]{RZ96}.
Moreover, Lemma \ref{POL_geompts} shows that $\MEpol(\overbar{k}) = \ME(\overbar{k})$.

\begin{rmk}
	Equivalently, we can describe $\MEpol$ as follows.
	For $S \in \Nilp$, we define $\MEpol(S)$ to be the set of equivalence classes of tuples $(X,\iota_E,\widetilde{\lambda},\varrho)$ where
	\begin{itemize}
		\item $X$ is a formal $O_F$-module over $S$ of height $4$ and dimension $2$,
		\item $\iota_E$ is an action of $O_E$ on $X$ that satisfies the Kottwitz condition in \eqref{POL_Kottwitz} and
		\item $\widetilde{\lambda}$ is a principal polarization on $X$ such that the Rosati involution induces the identity on $O_E$.
		\item Furthermore, we fix a framing object $(\mathbb{X},\iota_{\mathbb{X},E},\widetilde{\lambda}_{\mathbb{X}})$ over $\Spec \overbar{k}$, where $(\mathbb{X},\iota_{\mathbb{X},E})$ is the framing object for $\ME$ and $\widetilde{\lambda}_{\mathbb{X}}$ is a polarization as in Theorem \ref{POL_thm} \eqref{POL_thm1}.
		Then $\varrho$ is an $O_E$-linear quasi-isogeny
		\begin{equation*}
				\varrho: X \times_S \overbar{S} \to \mathbb{X} \times_{\Spec \overbar{k}} \overbar{S},
		\end{equation*}
		of height $0$ such that, locally on $\overbar{S}$, the (quasi-)polarizations $\varrho^{\ast}(\widetilde{\lambda}_{\mathbb{X}})$ and $\widetilde{\lambda}$ on $X$ only differ by a scalar in $O_E^{\times}$, \emph{i.e.}, there exists an element $\alpha \in O_E^{\times}$ such that $\varrho^{\ast}(\widetilde{\lambda}_{\mathbb{X}}) = \widetilde{\lambda} \circ \iota_E(\alpha)$.
		Two tuples $(X,\iota_E,\widetilde{\lambda},\varrho)$ and $(X',\iota_E',\widetilde{\lambda}',\varrho')$ are equivalent if there exists an $O_E$-linear isomorphism $\varphi: X \isoarrow X'$ such that $\varphi^{\ast}(\widetilde{\lambda}')$ and $\widetilde{\lambda}$ only differ by a scalar in $O_E^{\times}$.
	\end{itemize}
	In this way, we gave a definition for $\MEpol$ by introducing extra data on points of the moduli space $\ME$, instead of extra conditions.
	It is now clear, that $\MEpol$ describes a moduli problem for $p$-divisible groups of (PEL) type.
	It is easily checked that the two descriptions of $\MEpol$ give rise to the same moduli space.
\end{rmk}

Theorem \ref{POL_thm} now holds if and only if $\MEpol = \ME$.
This equality is a consequence of the following statement.

\begin{lem} \label{POL_rings}
	For any point $x = (X,\iota_E,\varrho) \in \MEpol(\overbar{k})$, the embedding $\MEpol  \inj \ME$ induces an isomorphism of completed local rings $\widehat{\mathcal{O}}_{\MEpol, x} \cong \widehat{\mathcal{O}}_{\ME,x}$.
\end{lem}

For the proof of this Lemma, we use the theory of local models, \emph{cf.}\ \cite[Chap.\ 3]{RZ96}.
We postpone the proof of this lemma to the end of this section and we first introduce the local models $\MEloc$ and $\MEpolloc$ for $\ME$ and $\MEpol$.


Let $C$ be a $4$-dimensional $F$-vector space with an action of $E$ and let $\Lambda \subseteq C$ be an $O_F$-lattice that is stable under the action of $O_E$.
Furthermore, let $(\,,)$ be an $F$-bilinear alternating form on $C$ with
\begin{equation} \label{POL_altloc}
	(\alpha x,y) = (x,\alpha y),
\end{equation}
for all $\alpha \in E$ and $x,y\in C$ and such that $\Lambda$ is unimodular with respect to $(\,,)$.
It is easily checked that $(\,,)$ is unique up to an isomorphism of $C$ that commutes with the $E$-action and that maps $\Lambda$ to itself.

For an $O_F$-algebra $R$, let $\MEloc(R)$ be the set of all direct summands $\mathcal{F} \subseteq \Lambda \otimes_{O_F} R$ of rank $2$ that are $O_E$-linear and satisfy the \emph{Kottwitz condition}.
That means, for all $\alpha \in O_E$, the action of $\alpha$ on the quotient $(\Lambda \otimes_{O_F} R) / \mathcal{F}$ has the characteristic polynomial
\begin{equation*}
	\charp(\Lie X, T \mid \alpha) = (T - \alpha)(T - \overbar{\alpha}).
\end{equation*}
The subset $\MEpolloc(R) \subseteq \MEloc(R)$ consists of all direct summands $\mathcal{F} \in \MEloc(R)$ that are in addition totally isotropic with respect to $(\,,)$ on $\Lambda \otimes_{O_F} R$.

The functor $\MEloc$ is representable by a closed subscheme of $\Gr(2,\Lambda)_{O_F}$, the Grassmanian of rank $2$ direct summands of $\Lambda$, and $\MEpolloc$ is representable by a closed subscheme of $\MEloc$.
In particular, both $\MEloc$ and $\MEpolloc$ are projective schemes over $\Spec O_F$.

These local models have already been studied by Deligne and Pappas. In particular, we have:

\begin{prop}[\cite{DP}] \label{POL_loc}
	$\MEpolloc = \MEloc$.
	In other words, for an $O_F$-algebra $R$, any direct summand $\mathcal{F} \in \MEloc(R)$ is totally isotropic with respect to $(\,,)$.
\end{prop}

The moduli spaces $\ME$ and $\MEpol$ are related to the local models $\MEloc$ and $\MEpolloc$ via local model diagrams, \emph{cf.}\ \cite[Chap.\ 3]{RZ96}.
Let $\MElarge$ be the functor that maps a scheme $S \in \Nilp$ to the set of isomorphism classes of tuples $(X,\iota_E,\varrho; \gamma)$.
Here, 
\begin{equation*}
	(X,\iota_E,\varrho) \in \ME(S),
\end{equation*}
and $\gamma$ is an $O_E$-linear isomorphism
\begin{equation*}
	\gamma: \mathbb{D}_X(S) \isoarrow \Lambda \otimes_{O_F} \mathcal{O}_S.
\end{equation*}
On the left hand side, $\mathbb{D}_X(S)$ denotes the (relative) Grothendieck-Messing crystal of $X$ evaluated at $S$, \emph{cf.}\ \cite[5.2]{Ahs11}.

Let $\MElochat$ be the $\pi_0$-adic completion of $\MEloc \otimes_{O_F} \breve{O}_F$.
Then there is a local model diagram:
\begin{equation*}
	\xymatrix{
		& \MElarge \ar[ld]_f \ar[rd]^g & \\
		\ME & & \MElochat
	}
\end{equation*}
The morphism $f$ on the left hand side is the projection $(X,\iota_E,\varrho; \gamma) \mapsto (X,\iota_E,\varrho)$.
The morphism $g$ on the right hand side maps $(X,\iota_E,\varrho; \gamma) \in \MElarge(S)$ to
\begin{equation*}
	\mathcal{F} = \ker (\Lambda \otimes_{O_F} \mathcal{O}_S \xrightarrow{\gamma^{-1}} \mathbb{D}_X(S) \to \Lie X) \subseteq \Lambda \otimes_{O_F} \mathcal{O}_S.
\end{equation*}
By \cite[Thm.\ 3.11]{RZ96}, the morphism $f$ is smooth and surjective.
The morphism $g$ is formally smooth by Grothendieck-Messing theory, see \cite[V.1.6]{Me72}, resp.\ \cite[Chap.\ 5.2]{Ahs11} for the relative setting (\emph{i.e.}, when $O_F \neq \mathbb{Z}_p$).

%

We also have a local model diagram for the space $\MEpol$.
We define $\MEpollarge$ as the fiber product $\MEpollarge = \MEpol \times_{\ME} \MElarge$.
Then $\MEpollarge$ is closed formal subscheme of $\MElarge$ with the following moduli description.
A point $(X,\iota_E,\varrho; \gamma) \in \MElarge(S)$ lies in $\MEpollarge(S)$ if $\widetilde{\lambda} = \varrho^{\ast}(\widetilde{\lambda}_{\mathbb{X}})$ is a principal polarization on $X$.
In that case, $\widetilde{\lambda}$ induces an alternating form $(\,,)^X$ on $\mathbb{D}_X(S)$ which, under the isomorphism $\gamma$, is equal to the form $(\,,)$ on $\Lambda \otimes_{O_F} \mathcal{O}_S$, up to a unit in $O_E \otimes_{O_F} \mathcal{O_S}$.

The local model diagram for $\MEpol$ now looks as follows.
\begin{equation} \label{POL_LMD}
	\begin{aligned}
		\xymatrix{
			& \MEpollarge \ar[ld]_{f_{\mathrm{pol}}} \ar[rd]^{g_{\mathrm{pol}}} & \\
			\MEpol & & \MEpollochat
		}
	\end{aligned}
\end{equation}
Here, $\MEpollochat$ is the $\pi_0$-adic completion of $\MEpolloc \otimes_{O_F} \breve{O}_F$ and $f_{\mathrm{pol}}$ and $g_{\mathrm{pol}}$ are the restrictions of the morphisms $f$ and $g$ above.
Again, $g_{\mathrm{pol}}$ is formally smooth by Grothendieck-Messing theory and $f_{\mathrm{pol}}$ is smooth and surjective by construction.

We can now finish the proof of Lemma \ref{POL_rings}.

\begin{proof}[Proof of Lemma \ref{POL_rings}]
	We have the following commutative diagram.
	\begin{equation} \label{POL_commdiag}
		\begin{aligned}
			\xymatrix{
				\MEpol \, & \MEpollarge \ar[l]_{f_{\mathrm{pol}}} \ar[r]^{g_{\mathrm{pol}}} & \; \MEpollochat \ar@{=}[d] \\
				\ME & \MElarge \ar[l]_f \ar[r]^g & \MElochat
				\ar@{_{(}->}"1,1"*+\frm{};"2,1"*+\frm{}
				\ar@{_{(}->}"1,2"*+\frm{};"2,2"*+\frm{}
			}
		\end{aligned}
	\end{equation}
	The equality on the right hand side follows from Proposition \ref{POL_loc}.
	The other vertical arrows are closed embeddings.
	
	Let $x \in \MEpol(\overbar{k})$.
	By \cite[Prop.\ 3.33]{RZ96}, there exists an \'etale neighbourhood $U$ of $x$ in $\ME$ and section $s: U \to \MElarge$ such that $g \circ s$ is formally \'etale.
	Similarly, $U_{\mathrm{pol}} = U \times_{\ME} \MEpol$ and $s_{\mathrm{pol}}$ is the base change of $s$ to $U_{\mathrm{pol}}$.
	Then the composition $g_{\mathrm{pol}} \circ s_{\mathrm{pol}}$ is also formally \'etale.
	This formally \'etale maps induce isomorphism of local rings $\widehat{\mathcal{O}}_{\ME,x} \isoarrow \widehat{\mathcal{O}}_{\MElochat,x'}$ and $\widehat{\mathcal{O}}_{\MEpol,x} \isoarrow \widehat{\mathcal{O}}_{\MEpollochat,x'}$, $x' = s(g(x))$.
	By Proposition \ref{POL_loc}, we have $\widehat{\mathcal{O}}_{\MElochat,x'} = \widehat{\mathcal{O}}_{\MEpollochat,x'}$ and since this identification commutes with $g \circ s$ (resp.\ $g_{\mathrm{pol}} \circ s_{\mathrm{pol}}$), we get the desired isomorphism $\widehat{\mathcal{O}}_{\MEpol,x} \cong \widehat{\mathcal{O}}_{\ME,x}$.
\end{proof}

\bibliography{master.bib}

\end{document}